\documentclass[a4paper,10pt,reqno]{amsart}

\usepackage{color,array}
\usepackage{cite}
\usepackage{hyperref}
\usepackage{refcount}
\usepackage{amsmath,amssymb,amsthm}
\usepackage{nicefrac}

\hypersetup{linktocpage,colorlinks,linkcolor={red},citecolor={blue},bookmarksdepth=1,bookmarksopen,bookmarksnumbered}
\usepackage[english]{babel}
\usepackage[T1]{fontenc}
\usepackage[utf8]{inputenc}
\usepackage{dsfont}
\usepackage{enumitem}
\usepackage[a4paper,left=2.6cm,right=2.6cm,top=3cm,bottom=3.5cm]{geometry}%

\usepackage{esint}

\usepackage{lmodern}
\numberwithin{equation}{section}

\setlist{nosep}

\newtheorem{theorem}{Theorem}%

\newtheorem{proposition}{Proposition}[section]
\newtheorem{lemma}[proposition]{Lemma}%
\newtheorem{corollary}[proposition]{Corollary}%

\newtheorem{remark}[proposition]{Remark}%
\newtheorem{assumptions}[proposition]{Assumptions}

\newcommand{\dQ}{\mathbb{Q}}%
\newcommand{\dR}{\mathbb{R}}%
\newcommand{\dT}{\mathbb{T}}

\newcommand{\ve}{\varepsilon}

\newcommand{\dP}{\mathbb{P}}%
\newcommand{\dE}{\mathbb{E}}%
\newcommand{\Var}{\mathrm{Var}}%
\newcommand{\Cov}{\mathrm{Cov} }%

\newcommand{\mc}{\mathcal}

\newcommand{\diag}{\mathrm{diag}}
\newcommand{\dd}{\mathrm{d}}

\newcommand{\Hc}{\mathcal{H}}%
\newcommand{\Gap}{\mathrm{Gap}}
\newcommand{\dGi}{\dP_{N,\beta}}
\newcommand{\dGip}{\dP_{N',\beta}}
\newcommand{\dGipQ}{\widetilde{\dP}_{N',\beta}}
\NewDocumentCommand{\dGiQ}{e{^}}{\IfValueTF{#1}{\widetilde{\dP}_{N,\beta}^{#1}}{\dQ_{N,\beta}}}
\newcommand{\dV}{\mathrm{V}}

\newcommand{\FF}{\mathrm{F}}
\newcommand{\g}{\mathrm{g}}
\newcommand{\Riesz}{\mathrm{Riesz}}
\newcommand{\Conf}{\mathrm{Conf}}
\newcommand{\Sine}{\mathrm{Sine}}
\newcommand{\sM}{\mathsf{M}}
\newcommand{\sA}{\mathsf{A}}
\newcommand{\sR}{\mathsf{R}}
\newcommand{\sS}{\mathsf{S}}

\newcommand{\bn}{\bar{n}}

\newcommand{\tM}{\mathcal T}

\begin{document}
\title[Decay of correlations for the circular Riesz gas]{Decay of correlations and thermodynamic limit\\ for the circular Riesz gas}

\author{Jeanne Boursier}
\address{Department of Mathematics, Massachusetts Institute of Technology, Cambridge, MA 02139, USA}
\email{boursier@mit.edu}

\begin{abstract}
We study the circular Riesz gas, a one-dimensional system of particles with power-law repulsion. For the long-range gas, we prove that the correlation between two gaps at distance $d$ decays like $d^{s-2}$, up to an arbitrarily small loss in the exponent, and that this exponent is sharp. As a consequence, the microscopic point process converges as the number of particles tends to infinity, and its limit is hyperuniform. We also prove the decay of correlations and the convergence of the microscopic point process for the hypersingular gas. The proof analyzes the Helffer--Sjöstrand equation, an elliptic system in high dimension that encodes covariances. It combines the deformation of Combes and Thomas with discrete elliptic regularity estimates and a factorization around the equally spaced configuration.
\end{abstract}
\maketitle

\setcounter{tocdepth}{1}
\tableofcontents

\section{Introduction}

\subsection{The circular Riesz gas}
This paper studies the decay of correlations of an interacting particle system on the circle $\dT:=\dR/\mathbb{Z}$, called the circular Riesz gas, and its limit as the number of particles tends to infinity.

Given a parameter $s\in(0,1)\cup(1,+\infty)$, the Riesz $s$-kernel on $\dT$ is defined by
\begin{equation}\label{eq:def gs}
    g_{s}:x\in {\dT\setminus\{0\}}\mapsto\lim_{n\to\infty}\Bigl(\sum_{k=-n}^n \frac{1}{|x+k|^s}-\frac{2}{1-s}n^{1-s}\Bigr)
    \,.
\end{equation}
For $0<s<1$, the kernel~$g_s$ is, up to a constant factor, the fundamental solution of the fractional Laplacian~$(-\Delta)^{(1-s)/2}$ on~$\dT$.
Let $D_N$ be the set of configurations $X_N=(x_1,\ldots,x_N)\in\dT^N$ admitting real lifts $x_1<x_2<\cdots<x_N<x_1+1$. Differences of positions use these cyclic lifts, extended by $x_{i+N}=x_i+1$.
Now consider the pairwise energy
\begin{equation}\label{def:Hn}
    \Hc_N:X_N\in D_N\mapsto N^{-s}\sum_{i\neq j}g_s(x_i-x_j)
    \,.
\end{equation}
Finally, the circular Riesz gas at inverse temperature $\beta>0$ is the probability measure
\begin{equation}\label{eq:def gibbs}
   \dd\dGi=\frac{1}{Z_{N,\beta}}e^{-\beta \Hc_N(X_N)}\mathds{1}_{D_N}(X_N)\dd X_N
    \,.
\end{equation}

The model belongs to the family of \emph{Riesz gases}, with kernel $|x|^{-s}$ on $\dR^d$~\cite{lewin2022coulomb}. The logarithmic kernel, obtained as $s\downarrow0$, gives the $\beta$-ensembles of random matrix theory. In dimension one, the interaction is long-range for $0<s<1$, and short-range, or hypersingular, for $s>1$. In~\cite{boursier2021optimal}, we proved near-optimal rigidity estimates for the gaps of the long-range gas and a central limit theorem for linear statistics of low regularity.

\subsection{Correlations and the Combes--Thomas method}\label{subsection:combes thomas}
We study the correlations of the gaps $N(x_{i+1}-x_i)$. As in the fluctuation--dissipation theorem of statistical physics, the correlations of a Gibbs measure are tied to its response to a perturbation. Let $\mu$ be a probability measure with density proportional to $e^{-H}$ on a domain of $\dR^N$. For smooth functions $F$ and $G$, this relation reads
\begin{equation*}
\Cov_{\mu}[F,G]=\dE_{\mu}[\nabla G\cdot\nabla\phi],\qquad\text{where}\quad A_1^\mu\nabla\phi=\nabla F\quad\text{and}\quad A_1^\mu=\nabla^2H+\mc L^\mu\otimes I_N
    \,.
\end{equation*}
Here $\mc L^\mu=\nabla H\cdot\nabla-\Delta$. The equation $A_1^\mu\nabla\phi=\nabla F$ is the Helffer--Sjöstrand equation~\cite{helffer1994correlation}, an elliptic system in dimension $N$. If $F$ depends only on the first coordinate and $G$ only on the $j$-th coordinate, then the source $\nabla F$ is supported on the first coordinate, and the covariance only involves $\partial_j\phi$. Decay of correlations thus reduces to decay of the solution away from the source.

A classical tool for the decay of the solution is the deformation of Combes and Thomas, introduced for eigenfunctions of Schrödinger operators~\cite{Combes1973AsymptoticBO} and used by Helffer for Witten Laplacians~\cite{helffer1998remarks}. The solution, multiplied by a weight that grows away from the source, solves the same equation up to the commutator of the weight with the operator. If the operator mainly couples nearby coordinates, then this commutator is small, and an energy estimate bounds the weighted solution.

\subsection{The Riesz gas in gap coordinates}
The energy estimate in the Combes--Thomas method requires a uniformly convex energy. For the Riesz gas, $H=\beta\Hc_N$, and uniform convexity holds only in the gap coordinates $y_i=N(x_{i+1}-x_i)$, after adding to the energy a term that penalizes large gaps. When the gaps are close to $1$, the Hessian in these coordinates satisfies
\begin{equation*}
|\partial_{y_i}\partial_{y_j}\Hc_N|\asymp(1+d_N(i,j))^{-s}
    \,,
\end{equation*}
where $d_N(i,j)=\min(|i-j|,N-|i-j|)$. For $s>1$, these couplings are summable in $j$, and the Combes--Thomas method gives the decay of correlations of the short-range gas.

For $0<s<1$, these couplings are not summable, and the commutator is not small. The decay of correlations then comes from the structure of the Hessian, not from its locality. At the equally spaced configuration, the Hessian in gap coordinates is a circulant matrix. Its inverse has the decay of the kernel of the fractional Laplacian $(-\Delta)^{(1-s)/2}$, namely $d^{s-2}$ at distance $d$. For the Gaussian measure with this Hessian, the correlations of the gaps are the entries of this inverse. The difficulty is to show that this decay survives the fluctuations of the Hessian and the differential part $\mc L^\mu\otimes I_N$ of the equation.

For the long-range gas, the Combes--Thomas method is combined with three other tools.
\begin{itemize}[leftmargin=*]
\item \emph{Factorization around the equally spaced configuration.} Multiplying the equation by an approximate inverse of the circulant matrix bounds a weighted norm of the solution by a weighted Dirichlet energy, which the energy estimate no longer controls.
\item \emph{Elliptic regularity.} The equation controls a discrete fractional primitive of the solution in the labels. With the Combes--Thomas method and a discrete Gagliardo--Nirenberg inequality, this bounds the weighted Dirichlet energy by a weighted norm of the solution.
\item \emph{Localization.} Projecting the equation on a window of labels leaves an exterior field, which is compared with the equation at the equally spaced configuration. Iterating this comparison improves the decay of the solution.
\end{itemize}

\subsection{Thermodynamic limit}\label{subsection:thermo}
Let $\mathcal C_N=\sum_{i=1}^N\delta_{Nx_i}\mathds{1}_{|x_i|<\frac14}$ be the configuration near the origin, rescaled so that the typical spacing of the particles is $1$. We ask whether the law of $\mathcal C_N$ converges as $N\to\infty$. In statistical physics, the uniqueness of the infinite-volume limit is tied to the decay of correlations, as in the uniqueness theory of Dobrushin~\cite[Chapter~8]{Georgii+2011}. The particles far from the origin act on those near the origin as a boundary condition. If the correlations decay fast enough, then this boundary condition has a small effect near the origin. The limit is then unique, and the rate of decay of the correlations gives a rate of convergence. For $\beta$-ensembles, the limit is the $\Sine_\beta$ process~\cite{valko2009continuum,killip2009eigenvalue}, which was obtained from random matrix models rather than from decay of correlations.

\subsection{Main results}
Our first result concerns the correlations of the gaps of the long-range gas. For the gaps themselves, it states that, for every $\ve>0$, there exists $C(\beta,s,\ve)>0$ such that
\begin{equation*}
    |\Cov_{\dGi}[N(x_{i+1}-x_i),N(x_{j+1}-x_j)]|\leq\frac{C}{(1+d_N(i,j))^{2-s-\ve}}\qquad\text{for every }1\leq i,j\leq N
    \,.
\end{equation*}
More generally, the theorem bounds the covariance of functions of two gaps.

\begin{theorem}[Decay of correlations for the long-range Riesz gas]\label{theorem:decay gap correlations}
Let $0<s<1$ and $\beta>0$. For every $\ve>0$, there exist $C(\beta,s,\ve),c(\beta,s,\ve),\delta(s,\ve)>0$ such that, for every $N\geq2$, for all locally absolutely continuous $\xi,\chi:\dR \to \dR$ with bounded derivatives and for every $1\leq i,j\leq N$,
\begin{multline}\label{eq:decay corr stat}
    |\Cov_{\dGi}[\xi(N(x_{i+1}-x_i)),\chi(N(x_{j+1}-x_j))]|\\
    \leq C\frac{1}{(1+d_N(i,j))^{2-s-\ve}}
    \left(\dE_{\dGi}[\xi'(N(x_{i+1}-x_i))^2]^{1/2}
          +\|\xi'\|_\infty e^{-cd_N(i,j)^\delta}\right)\\
    {}\times\left(\dE_{\dGi}[\chi'(N(x_{j+1}-x_j))^2]^{1/2}
          +\|\chi'\|_\infty e^{-cd_N(i,j)^\delta}\right)\,.
\end{multline}
Moreover, there exists $c_0(\beta,s)>0$ such that, for every integer sequence $n_N\to\infty$ with $n_N/N\to0$ and for every sufficiently large $N$, there exist labels $i,j$ with $n_N/2\leq d_N(i,j)\leq n_N$ and
\begin{equation}\label{eq:l corr}
    \Cov_{\dGi}[N(x_{i+1}-x_i),N(x_{j+1}-x_j)]
    \leq -c_0\,d_N(i,j)^{s-2}
    \,.
\end{equation}
\end{theorem}
The bound~\eqref{eq:l corr} shows that the exponent $2-s$ cannot be improved. It follows from the asymptotics of the variance of a block of gaps, by comparing two separated blocks.

The $L^2$ norms of $\xi'$ and $\chi'$ in~\eqref{eq:decay corr stat} come from weighted $L^2$ estimates on the solution of the Helffer--Sjöstrand equation. The terms with $\|\xi'\|_\infty$ and $\|\chi'\|_\infty$ come from a maximum principle, which controls the solution on an event of exponentially small probability.

The next theorem concerns the short-range gas.

\begin{theorem}[Decay of correlations for the short-range Riesz gas]\label{theorem:hypersing}
Let $s>1$ and $\beta>0$. For every $\ve>0$, there exist $C(\beta,s,\ve),c(\beta,s,\ve),\delta(s,\ve)>0$ such that, for every $N\geq2$, for all locally absolutely continuous $\xi,\chi:\dR \to \dR$ with bounded derivatives and for every $1\leq i,j\leq N$,
\begin{multline}\label{eq:decay hyper}
    |\Cov_{\dGi}[\xi(N(x_{i+1}-x_i)),\chi(N(x_{j+1}-x_j))]|\\
    \leq C\left(\frac{1}{(1+d_N(i,j))^{s-\ve}}+\frac{1}{N^{1-\ve}}\right)
    \left(\dE_{\dGi}[\xi'(N(x_{i+1}-x_i))^2]^{1/2}
          +\|\xi'\|_\infty e^{-cd_N(i,j)^\delta}\right)\\
    {}\times\left(\dE_{\dGi}[\chi'(N(x_{j+1}-x_j))^2]^{1/2}
          +\|\chi'\|_\infty e^{-cd_N(i,j)^\delta}\right)\,.
\end{multline}
\end{theorem}

The term $N^{-1+\ve}$ in~\eqref{eq:decay hyper} accounts for the correlations created by the constraint $\sum_{j=1}^N N(x_{j+1}-x_j)=N$, as in finite-volume systems of statistical mechanics~\cite{Ernst1981NonequilibriumFI,pulvirenti2017boltzmann,bodineau2020statistical}. In the Helffer--Sjöstrand equation in gap coordinates, it comes from the Lagrange multiplier of this constraint. This term does not appear in Theorem~\ref{theorem:decay gap correlations}.

For $s>1$, the exponent~$s$ in~\eqref{eq:decay hyper} is that of the entries of the Hessian in gap coordinates. For $0<s<1$, the exponent~$2-s$ in~\eqref{eq:decay corr stat} is larger than~$s$, so the correlations decay faster than these entries.

Our next result concerns the law of the configuration near the origin. Since $\dGi$ is invariant under rotations, the same law describes the configuration near every point of $\dT$. Let $\Conf(\dR)$ be the set of locally finite, simple point configurations in $\dR$. For a Borel set $B\subset \dR$, let $N_{B}:\Conf(\dR)\to {\mathbb{N}\cup\{+\infty\}}$ be the number of points in $B$. We endow $\Conf(\dR)$ with the $\sigma$-algebra generated by the maps $N_B$. A point process is a probability measure on $\Conf(\dR)$. For $(x_1,\ldots,x_N)$ distributed according to $\dGi$, let
\begin{equation}\label{eq:def Qx}
    \dQ_{N,\beta}=\mathrm{Law}(\mathcal C_N)
\end{equation}
be the law of the configuration~$\mathcal C_N$ of Subsection~\ref{subsection:thermo}.

\begin{theorem}[Convergence of the microscopic point process]\label{theorem:convergence}
Let $s\in(0,1)\cup(1,+\infty)$ and $\beta>0$, and for $R>0$, let ${\pi_R(\mc C)=\mc C|_{[-R,R]}}$ be the restriction of a configuration to $[-R,R]$. Then there exists a translation invariant point process $\Riesz_{s,\beta}$ such that, for every $R>0$,
\begin{equation*}
    \lim_{N\to\infty}\operatorname{TV}
    \bigl(\pi_R\#\dQ_{N,\beta},\pi_R\#\Riesz_{s,\beta}\bigr)=0\,.
\end{equation*}
\end{theorem}

Theorem~\ref{theorem:convergence} is the analogue for Riesz gases of the convergence of $\beta$-ensembles to the $\Sine_\beta$ process~\cite{valko2009continuum,killip2009eigenvalue,bourgade2012bulk,bourgade2014universality}. For the labelled gaps, we compare with the Palm law $\Riesz_{s,\beta}^0$ of $\Riesz_{s,\beta}$, which has a particle at the origin. Under this law, let $(z_i)_{i\in\mathbb Z}$ be the points of the configuration in increasing order, with $z_0=0$.

\begin{theorem}[Quantitative convergence]\label{theorem:quantitative conv}
Let $s\in(0,1)\cup(1,+\infty)$, $\beta>0$ and $\ve>0$. There exists $C(\beta,s,\ve)>0$ such that, for every $N\geq2$, every integer $1\leq K\leq\lfloor N^{\min(s,2)/2}/32\rfloor$ and every Lipschitz $G:\dR^K\to\dR$,
\begin{multline*}
 \left|\dE_{\dGi}G\bigl(N(x_2-x_1),\ldots,N(x_{K+1}-x_K)\bigr)
 -\dE_{\Riesz_{s,\beta}^0}G(z_1-z_0,\ldots,z_K-z_{K-1})\right|\\
 \leq C N^{-\frac14\min(s,2)\min(s,1)+\ve}\sum_{j=1}^K\|\partial_jG\|_\infty
    \,.
\end{multline*}
\end{theorem}
The rate is $N^{-s^2/4+\ve}$ for $0<s<1$ and $N^{-\min(s,2)/4+\ve}$ for $s>1$. For $0<s<1$, the process $\Riesz_{s,\beta}$ inherits the Gaussian fluctuations of the circular Riesz gas~\cite{boursier2021optimal}.

\begin{corollary}[Hyperuniformity of the $\Riesz_{s,\beta}$ process]\label{theorem:hyper riesz}
Let $0<s<1$ and $\beta>0$, and let
\begin{equation*}
    \sigma^2=\frac{\cot(\pi s/2)}{2\beta\pi s}
    \,.
\end{equation*}
Under $\Riesz_{s,\beta}$,
\begin{equation*}
    R^{-s/2}\bigl(N_{(0,R]}-R\bigr)
    \ \mathrel{\underset{R\to\infty}{\Longrightarrow}}\ \mathcal N(0,\sigma^2),
    \qquad
    \Var[N_{(0,R]}]=\sigma^2R^s+o(R^s)
    \,.
\end{equation*}
Under the Palm law $\Riesz_{s,\beta}^{0}$,
\begin{equation*}
    K^{-s/2}(z_K-K)\ \mathrel{\underset{K\to\infty}{\Longrightarrow}}\ \mathcal N(0,\sigma^2),
    \qquad
    \Var[z_K]=\sigma^2K^s+o(K^s)
    \,.
\end{equation*}
\end{corollary}
In particular, the number variance is sublinear, so $\Riesz_{s,\beta}$ is
hyperuniform in the sense of~\cite{Torquato_2016}.

Our last result shows that the particles of $\Riesz_{s,\beta}$ repel each other at short distance.

\begin{proposition}\label{proposition:replusion}
Let $s\in(0,1)\cup(1,+\infty)$ and $\beta>0$. There exist $c(\beta,s),C(\beta,s)>0$ such that, for every $i\in\mathbb Z$ and $0<u\leq1$,
\begin{equation*}
    \dP_{\Riesz_{s,\beta}^0}(z_{i+1}-z_i\leq u)
    \leq C e^{-cu^{-s}}
    \,.
\end{equation*}
\end{proposition}

\subsection{Related results}
\subsubsection{The Helffer--Sjöstrand equation}
Sjöstrand~\cite{sjostrandpotential,sjostrandpot2} and Helffer and Sjöstrand~\cite{helffer1994correlation} introduced this equation to study the correlations of convex models of statistical mechanics. Its solution is tied to the Langevin dynamics, with generator $-\mc L^\mu$. Since $\nabla\mc L^\mu=A_1^\mu\nabla$, the gradient of the Langevin semigroup satisfies $\nabla e^{-t\mc L^\mu}F=e^{-tA_1^\mu}\nabla F$, and the solution is the time integral $\nabla\phi=\int_0^\infty e^{-tA_1^\mu}\nabla F\,\dd t$. In convex settings, the semigroup $e^{-tA_1^\mu}$ has a representation by a random walk in a random environment, which evolves according to the Langevin dynamics~\cite{helffer1994correlation,naddaf1997homogenization}. Decay of correlations then follows from bounds on the transition probabilities of this walk. Naddaf and Spencer combined this representation with homogenization to identify the scaling limit of gradient interface models with a Gaussian free field~\cite{naddaf1997homogenization}. The representation also gives exponential decay of correlations at low temperature~\cite{Bach2003CorrelationAL}, large deviations and concentration~\cite{Deuschel2000LargeDA}, and equilibrium fluctuations~\cite{giacomin} for these models. Other works use ideas from stochastic homogenization without this representation~\cite{armstrong2019c,dario2020massless,thoma2023thermodynamic}.

\subsubsection{Riesz and Coulomb gases}
In dimension $d\geq2$, Armstrong and Serfaty proved local laws and rigidity for Coulomb gases at every temperature~\cite{Armstrong2019LocalLA}. For confined super-Coulombic Riesz gases, Peilen and Serfaty proved such local laws down to the microscopic scale in the bulk, in every dimension. They also proved a CLT for smooth linear statistics at mesoscopic scales, under a regularity assumption on the free energy and an error bound on its relative expansion~\cite{peilen2025local}. The hyperuniformity of Corollary~\ref{theorem:hyper riesz} has a two-dimensional analogue, since Leblé proved that the one-component plasma is hyperuniform~\cite{Leble2021TheTO}.

In physics, the fluctuations of the one-dimensional Riesz gas have mostly been computed in a harmonic trap. Santra, Kethepalli, Agarwal, Dhar, Kulkarni and Kundu determined how the variance of the gaps scales with the number of particles~\cite{santra2021gap}. For the counting statistics, Kethepalli, Kulkarni, Kundu, Majumdar, Mukamel and Schehr treated the hypersingular gas~\cite{kethepalli2024full}, and Le Doussal and Schehr computed them, with the cumulants of linear statistics, for long-range and short-range interactions~\cite{ledoussal2025cumulants}. On the circle, Touzo, Le Doussal and Schehr computed the space-time correlations of the gas around the equally spaced configuration, in the regime of weak noise~\cite{touzo2025spatio}.

\subsubsection{Decay of correlations for log-gases}
For $\beta$-ensembles, Erdős and Yau proved that the gap distribution of a local log-gas does not depend on its boundary condition~\cite{erdos2014gap}. They interpolate between two boundary conditions, which reduces the question to a bound on the covariance of a gap observable with the difference of the boundary energies~\cite[Theorem~8.1]{erdos2014gap}. In the random walk representation, the walk jumps from the label $j$ to the label $k$ at rate $\beta(x_j-x_k)^{-2}$, which decays like the inverse square of $|j-k|$. The covariance with a function of the gap $x_{p+1}-x_p$ only involves the difference of the solution at the neighboring labels $p$ and $p+1$. Erdős and Yau bound this difference by a Hölder estimate for the associated parabolic equation, which they prove by a discrete De Giorgi--Nash--Moser argument.

\subsubsection{The sine-beta process}
For $\beta$-ensembles, the microscopic limit is the $\Sine_\beta$ process, and several methods describe it. Valkó and Virág constructed it from the tridiagonal matrix model of the Gaussian $\beta$-ensemble, through a system of stochastic differential equations called the Brownian carousel~\cite{valko2009continuum}. Killip and Stoiciu obtained it from the CMV matrix model of the circular $\beta$-ensemble~\cite{killip2009eigenvalue}. Valkó and Virág later identified it with the spectrum of a random differential operator~\cite{valko2017sine}. Bourgade, Erdős and Yau proved that it is the bulk limit of $\beta$-ensembles with general potentials~\cite{bourgade2012bulk,bourgade2014universality}. They fix the particles outside a window and compare the resulting local measure with a Gaussian $\beta$-ensemble with the same boundary condition. Bekerman, Figalli and Guionnet gave another proof, with transport maps between $\beta$-ensembles~\cite{bekerman2015transport}. Without matrix models, Erbar, Huesmann and Leblé characterized $\Sine_\beta$ as the unique minimizer of the log-gas free energy among stationary point processes of unit intensity~\cite[{Corollary~1.2}]{erbar2018one}. Dereudre, Hardy, Leblé and Maïda proved that it satisfies the canonical DLR equations~\cite{dereudre2021dlr}.

\subsubsection{DLR equations and number-rigidity}
The existence of limit points of the law of $\mathcal C_N$ is standard~\cite{georgiizessin,dereudre2021dlr}. In contrast with Theorem~\ref{theorem:convergence}, the one-dimensional Coulomb gas has several crystalline Gibbs states~\cite{KUNZ1974303,aizenman2001bounded}. A point process is number-rigid if the configuration outside each bounded Borel set determines the number of points inside. For $0<s<1$, Dereudre and Vasseur constructed a subsequential limit of circular Riesz gases that satisfies the canonical and grand-canonical DLR equations and is not number-rigid~\cite[Theorems~1.8, 1.11 and 1.14]{dereudre2021number}. After rescaling and forgetting labels, their finite-volume laws in dimension one at inverse temperature $2\beta$ agree with~\eqref{eq:def gibbs}, since their energy sums over unordered pairs. Theorem~\ref{theorem:convergence} therefore identifies their limit with $\Riesz_{s,\beta}$, which satisfies these DLR equations and is not number-rigid. This contrasts with the number-rigidity of $\Sine_\beta$~\cite{chhaibi2018rigidity,dereudre2021dlr}.

\subsection{Outline of the proofs}\label{subsection:outline}

Let $\mu$ be a probability measure with density proportional to $e^{-H}$ on $D_N$. As in Subsection~\ref{subsection:combes thomas}, the covariance of two smooth functions $F,G:D_N\to\dR$ is
\begin{equation*}
    \Cov_{\mu}[F,G]=\dE_{\mu}[ \nabla \phi\cdot\nabla G]
    \,,
\end{equation*}
where $\phi$ solves
\begin{equation}\label{eq:Lmuphi c}
\left\{
\begin{array}{ll}
    -\Delta \phi+\nabla H\cdot\nabla\phi=F-\dE_{\mu}[F] & \text{on }D_N \\
    \nabla \phi\cdot \vec{n}=0 & \text{on }\partial D_N.
\end{array}
\right.
\end{equation}
The operator $\mc{L}^{\mu}=\nabla H\cdot\nabla-\Delta$ of~\eqref{eq:Lmuphi c} is, up to a sign, the generator of the Langevin dynamics with energy $H$. Differentiating~\eqref{eq:Lmuphi c} gives the Helffer--Sjöstrand equation
\begin{equation}\label{eq:intro helffer equation c}
  \left\{ \begin{array}{ll}
       A_1^{\mu}\nabla \phi=\nabla F  & \text{on }D_N \\
    \nabla \phi\cdot \vec{n}=0  & \text{on }\partial D_N
   \end{array}
   \right.\,.
\end{equation}
If the Hessian of the energy is uniformly positive definite, then an integration by parts gives a weighted $L^2$ bound on $\nabla\phi$, which yields the Brascamp--Lieb inequality. A maximum principle also gives a uniform bound on $|\nabla\phi|$~\cite{helffer1994correlation}.

The energy of the Riesz gas is strictly convex in the gaps but not in the points, so we rewrite~\eqref{eq:intro helffer equation c} in gap coordinates. Let
\begin{equation*}
    \Gap_N:X_N\in D_N\mapsto (N(x_2-x_1),\ldots,N(x_{N+1}-x_N))\in \dR^N
    \,,
\end{equation*}
with increasing lifts $x_1<\cdots<x_N<x_{N+1}=x_1+1$, and let
\begin{equation*}
    \mc{M}_N:=\{(y_1,\ldots,y_N)\in (0,+\infty)^N:y_1+\ldots+y_N=N\}
    \,.
\end{equation*}
Let $H^\g$, $\chi$ and $\xi$ be smooth functions on~$\mc M_N$. Let $\mu$ have density proportional to $e^{-H^\g\circ \Gap_N}$ on~$D_N$, let $F=\chi\circ\Gap_N$ and $G=\xi\circ\Gap_N$, and let $\nu=\Gap_N\#\mu$. With
\begin{equation*}
\mc{L}^\nu=\nabla H^\g\cdot \nabla -\Delta\quad \text{and}\quad  A_1^\nu=\nabla^2 H^\g+\mc{L}^\nu\otimes I_N
    \,,
\end{equation*}
the covariance of $F$ and $G$ is
\begin{equation*}
    \Cov_\mu[F,G]=\dE_{\nu}[\nabla\psi\cdot \nabla \xi]
    \,,
\end{equation*}
where $\nabla\psi$ solves
\begin{equation}\label{eq:intro HS gaps}
    \left\{
    \begin{array}{ll}
        A_1^\nu \nabla\psi=\nabla \chi+\lambda\mathbf1 & \text{on }\mc{M}_N \\
        \nabla\psi\cdot\mathbf1=0 &  \text{on }\mc{M}_N \\
        \nabla\psi\cdot \vec{n}=0 & \text{on }\partial \mc{M}_N.
    \end{array}
    \right.
\end{equation}
The coefficient $\lambda$ is the Lagrange multiplier of the constraint $y_1+\ldots+y_N=N$. We study the decay of $\partial_j\psi$ when $\nabla\chi=e_1$.

For the Riesz gas, $H^\g=\beta\Hc_N^\g$, where $\Hc_N^\g$ is the energy~\eqref{def:Hn} in gap coordinates. The Hessian of~$\Hc_N^\g$ is bounded below by the curvature of the nearest-pair interaction, which degenerates when a gap is large. We therefore fix a window of~$n$ consecutive labels and a small~$\ve>0$, add a convex penalty on the gaps of the window larger than~$n^\ve/2$, and study the marginal law of these gaps. Since gaps larger than~$n^\ve/2$ are very unlikely, the penalty changes the law of the gaps very little. Rigidity controls the sums of consecutive gaps, and hence the interval coefficients of the Hessian in~\eqref{eq:expHg2}. On a good event, the coefficient of an interval of~$r$ labels is bounded by $Cn^{C\ve}r^{-s-2}$, so that the Hessian entries are bounded by $Cn^{C\ve}(1+d_N(i,j))^{-s}$. For~$0<s<1$ and at unit gaps, the Hessian is close to~$2\mathbb G_{N,s}$, where~$\mathbb G_{N,s}$ is the regularized Riesz matrix~\eqref{eq:def H matrix}. For $0<s<1$, the inverse of this matrix has the decay of the kernel of the fractional Laplacian~$(-\Delta)^{(1-s)/2}$.

We first present the argument in the short-range case $s>1$, which is the model computation for the long-range case. To simplify, we drop the multiplier~$\lambda$ and assume that there exist $c,C>0$ such that, uniformly in the configuration,
\begin{equation*}
    \nabla^2 \Hc_N^\g\geq cI_N\quad \text{and}\quad |\partial_{ij}\Hc_N^\g|\leq \frac{C}{(1+d_N(i,j))^{s}}\quad \text{for each $1\leq i,j\leq N$}
    \,.
\end{equation*}
Let $\alpha>0$, and let~$W_\alpha$ be the diagonal weight matrix
\begin{equation*}
    W_\alpha=\diag(((1+d_N(j,1))^\alpha)_{1\leq j\leq N})
    \,.
\end{equation*}
The vector field $v=W_\alpha\nabla\psi$ solves
\begin{equation}\label{eq:outline split}
    \beta\nabla^2\Hc_N^\g v+\mc{L}^\nu v+\beta(W_\alpha\nabla^2\Hc_N^\g W_\alpha^{-1}-\nabla^2\Hc_N^\g)v=e_1
    \,.
\end{equation}
Since $\nabla^2\Hc_N^\g$ is close to a diagonal matrix, it nearly commutes with~$W_\alpha$. For~$1/2<\alpha<s-1/2$, the commutator in~\eqref{eq:outline split} is small compared with the identity, as a quadratic form, except on the coordinates close to~$1$, where the unweighted energy estimate applies. Integration by parts and the convexity of $\Hc_N^\g$ then give an $L^2$ bound on~$v$, which is a first global decay estimate on $\nabla\psi$. This is the deformation of Combes and Thomas of Subsection~\ref{subsection:combes thomas}. Pairing the solutions with point sources at~$1$ and at~$j$ then improves this estimate, and the multiplier~$\lambda$ adds a term of order~$N^{-1}$.

For $0<s<1$, the commutator is not small. We instead factorize~\eqref{eq:intro HS gaps} around the equally spaced configuration, by multiplying the system by $\sA=\tau(2\mathbb G_{N,s}+\tau I_N)^{-1}$ for a large~$\tau>0$. Since~$\sA$ does not depend on the configuration, it commutes with~$\mc L^\nu$. The matrix $\sA\nabla^2 \Hc_N^\g$ is uniformly positive definite, but the energy estimate then no longer controls the derivatives $\nabla(\partial_i\psi)$. Let $p,q>0$ be exponents fixed below. We compare the weighted norms $\|W_p\nabla\psi\|_2$ and $\|W_q\nabla^2\psi\|_2$ in~$L^2(\nu)$, whose squares are
\begin{equation}\label{eq:alpha gamma}
    \dE_{\nu}\Bigl[\sum_{i=1}^N (1+d_N(i,1))^{2p}(\partial_i\psi)^2\Bigr]\quad \text{and}\quad \dE_{\nu}\Bigl[\sum_{i=1}^N (1+d_N(i,1))^{2q}|\nabla(\partial_i\psi)|^2\Bigr]
    \,.
\end{equation}
The comparison has three steps. First, an elliptic regularity estimate shows that, if $\partial_i\psi$ decays fast enough, then the solution has a discrete fractional primitive of order $3/2-s$ in $L^2$. A discrete Gagliardo--Nirenberg inequality then bounds its primitive of order $1-s/2$ by a power smaller than one of $\|W_p\nabla\psi\|_2$. Second, the deformation argument of the short-range case bounds $\|W_q\nabla^2\psi\|_2^2$ by a small multiple of $\|W_p\nabla\psi\|_2^2$, up to the product of $\|W_p\nabla\psi\|_2$ with this primitive. Third, the factorization bounds $\|W_p\nabla\psi\|_2$ by $\|W_{p-(1-s)/2}\nabla^2\psi\|_2$. For $p=3/2-s$ and $q=1-s/2$, we have $p-(1-s)/2=q$, so the last two bounds close and control both quantities in~\eqref{eq:alpha gamma}. This gives the optimal global decay of the solution of~\eqref{eq:intro HS gaps}, which we then localize.

It remains to pass from~\eqref{eq:alpha gamma} to a bound on a single $\partial_i\psi$. Fix a label $1\leq j\leq N$, and let
\begin{equation*}
    J=\Bigl\{i\in\{1,\ldots,N\}:d_N(i,j)\leq \frac{1}{8}d_N(1,j)\Bigr\}
    \,.
\end{equation*}
Projecting~\eqref{eq:intro HS gaps} on the window $J$ creates the exterior field
\begin{equation*}
    \dV_\ell:=-\beta \sum_{i\in J^c}\partial_{i\ell}\Hc_N^\g\partial_i\psi+\lambda\qquad\text{for }\ell\in J
    \,.
\end{equation*}
We embed~$J$ in a larger cycle, with a truncation of~$2\mathbb G_{N,s}$ on the added labels. On this cycle, we compare the solution with the solution of the same equation at the equally spaced configuration, whose resolvent is explicit. The global estimates show that~$\dV$ varies slowly in the label, and the cancellations of this resolvent then give the decay of the second solution. We apply the current pointwise estimates to the difference of the two solutions. This improves the decay exponents of~$\partial_j\psi$ and~$\nabla(\partial_j\psi)$ by~$1-s$, until they reach $2-s$ and~$(3-s)/2$. Finitely many iterations give the bound~\eqref{eq:decay corr stat} of Theorem~\ref{theorem:decay gap correlations}.

The uniqueness of the limiting point process in Theorem~\ref{theorem:convergence} follows from the decay of correlations, which we prove for slightly more general systems than~\eqref{eq:intro HS gaps}. The laws~$\dQ_{N,\beta}$ of~\eqref{eq:def Qx} have accumulation points in the local topology, so it suffices to prove that the limit is unique. Let $\dGi^\g=\Gap_N\#\dGi$ be the law of the microscopic gaps under~$\dGi$, and define~$\dGip^\g$ in the same way, with~$N'$ particles. Let $I=\{1,\ldots,n\}$, write $x=(x_i)_{i\in I}$ for the gaps with labels in~$I$, and freeze the other gaps. Given $N,N'\gg n$, we compare the conditional measures $\nu_n^y:=\dGi^\g(\cdot\mid y)$ and $\nu_n^z:=\dGip^\g(\cdot \mid z)$, where $y\in\dR^{N-n}$ and $z\in \dR^{N'-n}$ are two exterior configurations with $N-(y_1+\ldots+y_{N-n})=N'-(z_1+\ldots+z_{N'-n})$. Let $\mathcal K_y(x)=\Hc_N^\g(x,y)$ and $\mathcal K_z(x)=\Hc_{N'}^\g(x,z)$. We connect $\nu_n^y$ to $\nu_n^z$ by the path $\nu(t)$ that interpolates linearly between these two energies. Let $G:\dR^n\to\dR$ be a test function that depends only on the coordinates $x_i$ for $n/4\leq i\leq 3n/4$. Then
\begin{equation}\label{eq:diffi}
    \dE_{\nu_n^z}[G]-\dE_{\nu_n^y}[G]=\beta\int_0^1 \Cov_{\nu(t)}[G,\mathcal K_y-\mathcal K_z]\dd t
    \,.
\end{equation}
The differences of consecutive coordinates of $\nabla(\mathcal K_y-\mathcal K_z)$ decay towards the middle of~$I$, so the decay of correlations under $\nu(t)$ shows that~\eqref{eq:diffi} is small. It remains to average over the gap sum $N-(y_1+\ldots+y_{N-n})$ of the window, whose laws under~$\dGi^\g$ and~$\dGip^\g$ differ. An averaged derivative estimate shows that the conditional expectation varies little with this gap sum, and the log-concave densities of the two laws overlap.

\subsection{Organization of the paper}

\begin{itemize}
\item Section~\ref{section:preliminaries c} bounds the tails and the fluctuations of the gaps, and studies the Riesz kernel~$g_s$ and the regularized Riesz matrix~\eqref{eq:def H matrix}. In particular, it shows that the entries of the inverse of this matrix decay like~$d^{s-2}$ at distance~$d$.
\item Section~\ref{section:helffer c} states the well-posedness of the Helffer--Sjöstrand equation and rewrites it in gap coordinates. It also gives the Brascamp--Lieb inequality, a maximum principle and the Bakry--\'Emery criterion.
\item Section~\ref{section:decay short} carries out the deformation of Combes and Thomas in the short-range case $s>1$ and proves Theorem~\ref{theorem:hypersing}.
\item Section~\ref{section:large scale} treats the long-range case $0<s<1$ by the factorization, the comparison of the weighted norms and the localization of Subsection~\ref{subsection:outline}, and proves Theorem~\ref{theorem:decay gap correlations}.
\item Section~\ref{section:change measure} applies the decay of correlations along the path of measures of~\eqref{eq:diffi}. It proves Theorems~\ref{theorem:convergence} and~\ref{theorem:quantitative conv}, Corollary~\ref{theorem:hyper riesz} and Proposition~\ref{proposition:replusion}.
\item Appendix~\ref{section:existence c} verifies the collision condition~\eqref{eq:collision criterion} for the measures of the paper in Lemma~\ref{lemma:collision}, and proves the well-posedness results of Section~\ref{section:helffer c}. Appendix~\ref{section:short local laws} proves rigidity for the short-range gas, and Appendix~\ref{section:loc laws} proves local laws for the measures along the path of Section~\ref{section:change measure}.
\end{itemize}

\subsection{Notation}
\begin{itemize}
\item \emph{Labels.} We let $d_N$ be the symmetric distance on $\{1,\ldots,N\}$ defined for each $i, j\in \{1,\ldots,N\}$ by \[d_N(i,j)=\min(|j-i|,N-|j-i|).\]
We define~$d_m$ in the same way on~$\{1,\ldots,m\}$ for every~$m\geq1$, and extend it periodically to integer indices. Without a subscript,~$d$ is the cyclic distance when the indices form a cycle, and~$d(i,j)=|i-j|$ when they form an interval. The cardinality of a set of labels~$I$ is denoted by~$|I|$.
\item \emph{Vectors and matrices.} For $x\in \dR^n$, we let $|x|$ be the Euclidean norm of~$x$. For $1\leq p<\infty$, the $\ell^p$ norm of~$x$ is ${\|x\|_{\ell^p}=(\sum_{i=1}^n|x_i|^p)^{1/p}}$. We let~$e_i$ be the $i$-th vector of the standard basis and~$\mathbf1$ the vector with all entries equal to~$1$, in the dimension fixed by the context, and $\mathbf1_I=\sum_{i\in I}e_i$ for a set of labels~$I$. We let~$I_n$ be the $n\times n$ identity matrix and $\diag(a_1,\ldots,a_n)$ be the diagonal matrix with entries~$a_1,\ldots,a_n$. For symmetric matrices~$A$ and~$B$ of the same size, we write $A\geq B$ when $A-B$ is positive semidefinite. For a real $n\times n$ matrix~$M$, we let $\Vert M\Vert$ be the operator norm of $M$, given by
\begin{equation*}
    \Vert M\Vert=\sup_{v\in \dR^n\setminus \{0\}}\frac{|Mv|}{|v|}\,.
\end{equation*}
\item \emph{Functions and measures.} We write~$\nabla^2f$ for the Hessian of a real-valued function~$f$ on~$\dR^n$, and $\mathds1_A$ for the indicator function of a set~$A$. We write~$\Longrightarrow$ for convergence in law. For a measurable map~$T$ and probability measures~$\mu$ and~$\nu$, the push-forward~$T\#\mu$ and the total variation distance are given by
\begin{equation*}
    (T\#\mu)(A)=\mu(T^{-1}(A))
    \qquad\text{and}\qquad
    \mathrm{TV}(\mu,\nu)=\sup_A|\mu(A)-\nu(A)|
    \,.
\end{equation*}
\item \emph{Norms of functions.} The supremum norm of a bounded function~$f$ is $\|f\|_\infty=\sup|f|$. When a probability measure~$\nu$ is fixed by the context, we let
\begin{equation*}
    \|f\|_p=\dE_\nu[|f|^p]^{1/p}
    \qquad\text{for }1\leq p<\infty
    \,,
\end{equation*}
where~$|f|$ is the Euclidean norm of the entries when~$f$ is a vector or a matrix of functions. We write $\|f\|_{L^p(\mu)}$ for the same norm under another measure~$\mu$, and $\|f\|_{L^2(\nu;\ell^p)}=\dE_\nu[\|f\|_{\ell^p}^2]^{1/2}$ for a vector~$f$ of functions. The essential supremum of~$|f|$ under~$\nu$ is~$\|f\|_{L^\infty(\nu)}$. For a set of labels~$I$ and $1\leq p\leq\infty$, we write $L^p(\nu;\dR^I)$ and~$H^1(\nu;\dR^I)$ for the vectors $(\psi_i)_{i\in I}$ with entries in~$L^p(\nu)$ and in the space~$H^1(\nu)$ of Section~\ref{section:helffer c}, and we write~$\dR^m$ for~$\dR^I$ when $I=\{1,\ldots,m\}$. A generator~$\mc L$ acts on a vector $u=(u_i)_{i\in I}$ of functions componentwise, and we write $\mc Lu=(\mc L\otimes I_{|I|})u=(\mc Lu_i)_{i\in I}$.
\item \emph{Constants.} We write~$C(\beta)$ and~$c(\beta)$ for positive constants that may depend on~$s$ and are locally uniform in~$\beta$, that is, $C(\beta)$ is bounded above and~$c(\beta)$ below for~$\beta$ in each compact subset of~$(0,+\infty)$. Their values may change from line to line. For $A,B\geq0$, we write $A\asymp B$ when $A\leq C(\beta)B$ and $B\leq C(\beta)A$. We write~$C_n$ for a quantity of the form~$C(\beta)n^{C\ve}$, where~$C(\beta)$ may also depend on~$\ve$.
\end{itemize}

\subsection{Formalization}
Theorems~\ref{theorem:decay gap correlations}, \ref{theorem:hypersing}, \ref{theorem:convergence} and~\ref{theorem:quantitative conv}, Corollary~\ref{theorem:hyper riesz} and Proposition~\ref{proposition:replusion} have been formalized in the Lean proof assistant, using the Mathlib library and the formalization of~\cite{boursier2021optimal}. The formalization is available at \url{https://github.com/jboursier/circular-riesz-gas-correlations-lean}.

\subsection*{Acknowledgments}
The author would like to thank Sylvia Serfaty, Thomas Leblé, Angel Alastuey and Ahmed Bou-Rabee for many helpful remarks on this work.

The author thanks Scott Armstrong for providing the skill files used for the Lean formalization. During the last round of revisions, the author used GPT-6.1 Sol, GPT-6 Astra and Claude Opus 5.5 to formalize the paper in Lean, to make corrections, to propose simplifications of the proofs and to rewrite the manuscript.

{This work was supported by a grant from the ``Fondation CFM pour la Recherche'' and by the ERC Advanced Grant LDRAM (ERC-2019-ADG 884584).}

%This project has received partial funding from the European Research Council (ERC) under the European Union Horizon
%2020 research and innovation program (grant agreement No. 884584).

\section{Preliminaries}\label{section:preliminaries c}
This section collects the estimates on the Riesz gas and on the Riesz kernel that we use in the proofs. The first subsection bounds the tails and the fluctuations of the gaps. The other subsections study the kernel~$g_s$ and the circulant matrix obtained by sampling it. We show that the inverse of this matrix has entries that decay like~$d^{s-2}$ at distance~$d$.

\subsection{Gap tails and rigidity}
We first bound the tails of a single gap.

\begin{lemma}[Uniform one-gap tails]\label{lemma:single gap tails}
Let $s\in(0,1)\cup(1,+\infty)$ and $\beta>0$. There exist $C(\beta,s)>0$ and $c(\beta,s)>0$, locally uniform in~$\beta$, such that, for every $N\geq2$ and every $1\leq i\leq N$, the microscopic gap $X=N(x_{i+1}-x_i)$ under $\dGi$ satisfies
\begin{equation*}
    \mathbb P(X\leq u)\leq Ce^{-cu^{-s}}\quad(0<u\leq1),
    \qquad
    \mathbb P(X\geq u)\leq Ce^{-cu}\quad(u\geq1)
    \,.
\end{equation*}
Consequently, for every $r>0$, there exists $C(r)>0$ such that, for every $p\geq1$,
\begin{equation*}
    \|X^r\|_p+\|X^{-r}\|_p\leq C(r)p^{C(r)}
    \,.
\end{equation*}
\end{lemma}

\begin{proof}
Cyclic symmetry gives $\mathbb EX=1$. The gap energy is convex on the full simplex, so Pr\'ekopa's theorem makes the density~$f$ of~$X$ log-concave. Removing the nearest-pair term $2X^{-s}$ leaves a convex energy, since the remaining periodic images have positive second derivative and the long-range renormalization subtracts only constants. Thus
\begin{equation*}
    {f(u)=e^{-2\beta u^{-s}}h(u)}
    \,,
\end{equation*}
with $h$ log-concave.

Let $\bar F(u)=\mathbb P(X\geq u)$. This survival function is log-concave. The tangent of~$\log\bar F$ at $b=1/2$ gives
\begin{equation*}
    1=\int_0^N \bar F(u)\dd u
    \leq b+\frac{\bar F(b)^2}{f(b)}
    \,.
\end{equation*}
Since $\bar F(b)\leq1$, it follows that $f(b)\leq2$.
If $(\log h)'_-(b)<-2\beta sb^{-s-1}-2$, then $(\log f)'<-2$ on $(b,N)$. Hence $\bar F(b+t)\leq e^{-2t}\bar F(b)$ and $\mathbb EX<b+1/2=1$, a contradiction. Therefore $(\log h)'_-(b)\geq-C(\beta,s)$. Concavity yields $h(u)\leq h(b)e^{C(b-u)}$ for $u\leq b$, and consequently $f(u)\leq Ce^{-2\beta u^{-s}}$. Integration proves the lower tail. Finally, $\bar F(2)\leq1/2$ by Markov's inequality. Since $\log\bar F$ is concave and $\bar F(0)=1$, we have, for $u\geq2$,
\begin{equation*}
    {\log \bar F(u)\leq\frac u2\log \bar F(2)\leq-\frac u2\log2}
    \,.
\end{equation*}
Integrating the two tails proves the moment bounds.
\end{proof}

By the first item of~\cite[Theorem~1]{boursier2021optimal} and its proof, the sums of consecutive gaps satisfy the following estimate.

\begin{lemma}[Near-optimal rigidity]\label{theorem:almost optimal rigidity c}
Let $s\in(0,1)$ and $\beta>0$. Then there exists $\ve_0(s)>0$ such that, for every $0<\ve<\ve_0$, there exist $C(\beta,s,\ve)>0$ and $c(\beta,s,\ve)>0$ such that, for every $N\geq2$, every $1\leq i\leq N$ and every integer $1\leq k\leq N/2$,
\begin{equation*}
    \dGi\bigl(|N(x_{i+k}-x_i)-k|\geq k^{s/2+\ve}\bigr)
    \leq {Ce^{-ck^{\ve/(4(s+2))}}}\,.
\end{equation*}
\end{lemma}

\subsection{Fractional Laplacian on the circle}
In this subsection, let $s\in(0,1)$. We justify the expression of the fundamental solution of the fractional Laplace equation on the circle
\begin{equation}\label{eq:lapl}
    (-\Delta)^{\frac{1-s}{2}}g_s=c_s(\delta_0-1)
    \,,
\end{equation}
where the Fourier multiplier of $(-\Delta)^{(1-s)/2}$ is
$(2\pi|k|)^{1-s}$ and
\begin{equation}\label{eq:riesz Fourier constant}
    c_s=2\Gamma(1-s)\sin(\pi s/2)
    =2^{1-s}\sqrt\pi\,
      \frac{\Gamma((1-s)/2)}{\Gamma(s/2)}
    \,.
\end{equation}
In the next lemma, $\zeta(s,x)$ denotes the analytic continuation to~${0<s<1}$ of the Hurwitz zeta function~\cite{apostol1997modular}.

\begin{lemma}[Fundamental solution]\label{lemma:fund}
Let $s\in(0,1)$ and let $g_s$ be the mean-zero solution of~\eqref{eq:lapl}, with $c_s$ given by~\eqref{eq:riesz Fourier constant}. For $x\in(0,1)$, we have
\begin{equation}\label{eq:expression diff}
 g_s(x)=\zeta(s,x)+\zeta(s,1-x)=\lim_{n\to\infty}\Bigl(\sum_{k=-n}^n \frac{1}{|k+x|^s}-\frac{2}{1-s}n^{1-s}\Bigr)
    \,.
\end{equation}
Moreover for every integer $p\geq1$ and $x\in\dT\setminus\{0\}$, we have
\begin{equation*}
    g_s^{(p)}(x)=(-1)^p s\cdots (s+p-1)\sum_{k\in\mathbb{Z}}\frac{\operatorname{sgn}(x+k)^p}{|x+k|^{s+p}}
    \,.
\end{equation*}
\end{lemma}

\begin{proof}
We follow the proof of~\cite[Lemma~2.1]{boursier2021optimal}. For every $\lambda>0$, we have
\begin{equation*}
    \lambda^{-\frac{1-s}{2}}=\frac{1}{\Gamma(\frac{1-s}{2})}\int_0^{\infty}e^{-\lambda t}\frac{\dd t}{t^{1-\frac{1-s}{2}}}
    \,.
\end{equation*}
We apply this formula with $\lambda=4\pi^2k^2$ to the Fourier coefficient of~$g_s$ of index $k\neq0$. This expresses $g_s$ as
\begin{equation*}
 g_s(x)=\frac{c_s}{\Gamma(\frac{1-s}{2})}\int_0^{\infty}(h_t(x)-1)\frac{\dd t}{t^{\frac{1+s}{2}}}
    \,,
\end{equation*}
where $h_t$ is the heat kernel on $\dT$, namely
\begin{equation*}
    h_t(x)=\sum_{k\in\mathbb Z}e^{-4\pi^2tk^2}e^{2\pi ikx}=\frac{1}{\sqrt{4\pi t}}\sum_{k\in \mathbb{Z}}e^{-\frac{|x- k|^2}{4t} }
    \,.
\end{equation*}
We first center each image as $e^{-|x-k|^2/(4t)}-\int_0^1e^{-|y-k|^2/(4t)}\dd y$. The resulting series is absolutely integrable with weight $t^{-1-s/2}\dd t$, so Fubini's theorem gives~\eqref{eq:expression diff}. For $p\geq1$, the series of the $p$-th derivatives of the terms in~\eqref{eq:expression diff} converges absolutely and locally uniformly on $\dR\setminus\mathbb Z$, since $s+p>1$. Differentiating termwise therefore gives the formula for~$g_s^{(p)}$.
\end{proof}

By the formula for~$g_s^{(p)}$ with $p=2$, the kernel~$g_s$ is convex on~$(0,1)$. Hence the energy of the Riesz gas is a convex function of the gaps, which gives the functional inequalities of Section~\ref{section:helffer c}.

\subsection{Discrete and semi-discrete Fourier transforms}
We bound the entries of circulant matrices by the derivatives of their Fourier symbols. These bounds give the decay of the inverse Riesz matrix in Lemma~\ref{lemma:inverse}, and the kernel estimates of Section~\ref{section:large scale}.

Let $\theta_k=2\pi k/N$ for $0\leq k\leq N-1$. The Fourier and inverse Fourier transforms on $\mathbb{Z}/N\mathbb{Z}$ are
\begin{equation}\label{eq:Fd}
    \mathcal{F}_d(f)(\theta_k)=\sum_{n=0}^{N-1}f(n)e^{in\theta_k}
    \,,
\end{equation}
\begin{equation}\label{eq:Fdinv}
    \mathcal{F}_d^{-1}(m)(n)=\frac{1}{N}\sum_{k=0}^{N-1}m(\theta_k)e^{-in\theta_k}
    \,.
\end{equation}
For $f\in\ell^2(\mathbb{Z})$ and $m\in L^2([0,2\pi])$, the semi-discrete Fourier transform and its inverse are given by
\begin{equation*}
    \mathcal{F}_{\mathbb{Z}}(f)(\theta)=\sum_{n\in\mathbb{Z}}f(n)e^{in\theta}
    \qquad\text{and}\qquad
    \mathcal{F}_{\mathbb{Z}}^{-1}(m)(n)=\frac{1}{2\pi}\int_0^{2\pi}m(\theta)e^{-in\theta}\,\dd\theta
    \,.
\end{equation*}

For $\phi\in\ell^1(\mathbb{Z})$ and $0\leq k\leq N-1$, absolute convergence gives
\begin{equation}\label{eq:aliasing}
    \sum_{n=0}^{N-1}\sum_{\ell\in\mathbb{Z}}\phi(n+\ell N)e^{in\theta_k}
    =\sum_{n\in\mathbb{Z}}\phi(n)e^{in\theta_k}=\mathcal{F}_{\mathbb{Z}}(\phi)(\theta_k)
    \,.
\end{equation}
Let $\rho(\theta)$ be the distance from~$\theta$ to~$2\pi\mathbb{Z}$, so that $\rho(\theta)=\min(\theta,2\pi-\theta)$ for $0\leq\theta\leq2\pi$. The next lemma bounds the kernel of a discrete Fourier multiplier by the derivatives of its symbol.

\begin{lemma}[Kernels of Fourier multipliers]\label{lemma:multiplier kernels}
Let $N\geq2$ and $B>0$.
\begin{enumerate}
\item Let $r\in\{2,3\}$ and $0<\alpha<r-1$. Let $m:\dR\to\mathbb C$ be continuous and $2\pi$-periodic, with $m(0)=0$, of class~$\mathcal C^r$ on $\dR\setminus2\pi\mathbb Z$, and such that $|m^{(j)}(\theta)|\leq B\rho(\theta)^{\alpha-j}$ for $0\leq j\leq r$ and $\theta\notin2\pi\mathbb Z$. Let $\phi=\mathcal F_{\mathbb Z}^{-1}(m)$, let $z_0\in\mathbb C$, and let~$K$ be the inverse transform~\eqref{eq:Fdinv} of the symbol equal to~$z_0$ at~$\theta_0$ and to~$m(\theta_k)$ at~$\theta_k$ for $1\leq k\leq N-1$, that is, $K(h)=N^{-1}\bigl(z_0+\sum_{k=1}^{N-1}m(\theta_k)e^{-ih\theta_k}\bigr)$ for $h\in\mathbb Z$. Then there exists $C(\alpha,r)>0$ such that $|\phi(h)|\leq C(\alpha,r)B(1+|h|)^{-1-\alpha}$ for every $h\in\mathbb Z$, and
\begin{equation}\label{eq:multiplier kernel}
    |K(h)|\leq\frac{C(\alpha,r)B}{(1+d_N(h,0))^{1+\alpha}}+\frac{|z_0|}{N},
    \qquad
    \sum_{h=0}^{N-1}K(h)=z_0
    \,.
\end{equation}
\item Let $t>0$, and let $m:\dR\to\mathbb C$ be $2\pi$-periodic and of class~$\mathcal C^2$, supported where $t\leq\rho\leq4t$, with $|m^{(j)}|\leq Bt^{-j}$ for $0\leq j\leq2$. Let $K(h)=N^{-1}\sum_{k=1}^{N-1}m(\theta_k)e^{-ih\theta_k}$ for $h\in\mathbb Z$. Then there exists an absolute constant $C>0$ such that $\sum_{h=0}^{N-1}|K(h)|\leq CB$.
\end{enumerate}
\end{lemma}

\begin{proof}\
\paragraph{\emph{Step 1: the coefficients}}
Let $m$ be as in the first part, and let $h\neq0$. We split $m=m_1+m_2$, where $m_1$ is the product of~$m$ with a smooth periodic cutoff equal to one where $\rho\leq|h|^{-1}$ and to zero where $\rho\geq2|h|^{-1}$, whose $j$-th derivative is bounded by $C|h|^j$. The function~$m_2$ vanishes near~$2\pi\mathbb Z$, and we integrate its Fourier integral $r$ times by parts. A derivative of order $j\geq1$ of the cutoff is supported where $\rho\leq2|h|^{-1}$, and there it costs $|h|^j\leq2^j\rho^{-j}$. Since $\alpha-r<-1$, this gives
\begin{equation*}
    |\phi(h)|
    \leq CB\Bigl(\int_0^{2|h|^{-1}}u^{\alpha}\,\dd u
    +|h|^{-r}\int_{|h|^{-1}}^{\pi}u^{\alpha-r}\,\dd u\Bigr)
    \leq C(\alpha,r)B|h|^{-1-\alpha}
    \,.
\end{equation*}
Moreover, $|\phi(0)|\leq B\pi^\alpha$. Since $\alpha>0$, the function~$\phi$ is summable, and the continuity of~$m$ gives $\mathcal F_{\mathbb Z}(\phi)=m$.

\paragraph{\emph{Step 2: periodization}}
By~\eqref{eq:aliasing}, the discrete Fourier transform of $h\mapsto\sum_{\ell\in\mathbb Z}\phi(h+\ell N)$ equals $m(\theta_k)$ at~$\theta_k$, and vanishes at $\theta_0=0$. By~\eqref{eq:Fd}, the constant function $z_0/N$ has transform $z_0$ at~$\theta_0$ and zero elsewhere. Fourier inversion therefore gives
\begin{equation*}
    K(h)=\sum_{\ell\in\mathbb Z}\phi(h+\ell N)+\frac{z_0}N
    \,.
\end{equation*}
Since $K$ is $N$-periodic, we may assume that $|h|\leq N/2$. Then $d_N(h,0)=|h|$, and the sum of ${(1+|h+\ell N|)^{-1-\alpha}}$ over $\ell\neq0$ is at most $C(\alpha)(1+|h|)^{-1-\alpha}$. The bound of Step~1 therefore proves the first bound of~\eqref{eq:multiplier kernel}. Summing the definition of~$K$ over~$h$ proves the second.

\paragraph{\emph{Step 3: the localized multiplier}}
Let $m$ be as in the second part, and let $\phi=\mathcal F_{\mathbb Z}^{-1}(m)$. The support of~$m$ has measure at most $\min(6t,2\pi)$. Hence $|\phi(h)|\leq CB\min(1,t)$ for every~$h$, and two integrations by parts give
\begin{equation*}
    |\phi(h)|\leq\frac{CB}{th^2}
    \quad\text{for }h\neq0
    \,.
\end{equation*}
Summing over $h\in\mathbb Z$ gives $\sum_h|\phi(h)|\leq CB$. Since $m(0)=0$, Step~2 with $z_0=0$ expresses~$K(h)$ as the sum of $\phi(h+\ell N)$ over~$\ell$. Therefore $\sum_{h=0}^{N-1}|K(h)|\leq\sum_{h\in\mathbb Z}|\phi(h)|\leq CB$.
\end{proof}

\subsection{Decay of the inverse Riesz matrix}
In gap coordinates, the Hessian of the energy at the equally spaced configuration is close to twice the matrix~$\mathbb G_{N,s}$ defined below (see~\eqref{eq:unit gap comparison}). In this subsection, we show that the entries of the inverse of~$\mathbb G_{N,s}$ decay like~$d^{s-2}$ at distance~$d$.

Fix $s\in(0,1)$ and $N\geq2$. Set $m_N=\lfloor{N/2}\rfloor$ and $c_{N,s}=N^{-s}g_s({m_N/N})$. We regularize the diagonal and subtract a constant from the sampled Riesz kernel:
\begin{equation*}
    g_{N,s}(k)=
    \begin{cases}
        N^{-s}g_s(k/N)-c_{N,s} & \text{if }1\leq k\leq N-1,\\
        2+2N^{-s}\zeta(s)-c_{N,s} & \text{if }k=0.
    \end{cases}
\end{equation*}
Here $\zeta(s)=\lim_{n\to\infty}(\sum_{\ell=1}^n\ell^{-s}-n^{1-s}/(1-s))$ is the value of the Riemann zeta function at~$s$. Extend $g_{N,s}$ periodically to $\mathbb{Z}$. This normalization gives $g_{N,s}(m_N)=0$. Let~$a(0)=2$ and~$a(r)=|r|^{-s}$ for~$r\neq0$, and write~$\Delta a(r)=a(r+1)+a(r-1)-2a(r)$. By~\eqref{eq:def gs} if $k\notin N\mathbb Z$, and by the definition of~$\zeta(s)$ if $k\in N\mathbb Z$, the number $g_{N,s}(k)+c_{N,s}$ is the limit of $\sum_{|\ell|\leq n}a(k+\ell N)-2N^{-s}n^{1-s}/(1-s)$ as $n\to\infty$. Hence the cyclic second difference of~$g_{N,s}$ is the periodization of~$\Delta a$:
\begin{equation}\label{eq:defgss}
    g_{N,s}''(k)=g_{N,s}(k+1)+g_{N,s}(k-1)-2g_{N,s}(k)=\sum_{\ell\in\mathbb Z}\Delta a(k+\ell N)
    \,.
\end{equation}
Define
\begin{equation}\label{eq:def H matrix}
    \mathbb G_{N,s}=(g_{N,s}(j-i))_{1\leq i,j\leq N}
    \,.
\end{equation}
When the inverse exists, write ${\mathbb G}_{N,s}^{-1}=(g_{N,s}^{-1}(j-i))_{1\leq i,j\leq N}$.

\begin{lemma}[Decay of the inverse Riesz matrix]\label{lemma:inverse}
Let $s\in(0,1)$. There exists $C(s)>0$ such that, for every $N\geq2$, the matrix ${\mathbb G}_{N,s}$ is positive-definite and, for every $1\leq i,j\leq N$,
\begin{equation}\label{eq:inverse HN}
    |(\mathbb G_{N,s}^{-1})_{i,j}|\leq\frac{C(s)}{1+d_N(i,j)^{2-s}}
    \,.
\end{equation}
In addition,
\begin{equation}\label{eq:inv sum}
    \left|\sum_{i=1}^N(\mathbb G_{N,s}^{-1})_{i,1}\right|\leq\frac{C(s)}{N^{1-s}}
    \,.
\end{equation}
\end{lemma}

\begin{proof}\
\paragraph{\emph{Step 1: the Fourier eigenvalues}}
The sequence~$\Delta a$ is absolutely summable. For~$1\leq k\leq N-1$, \eqref{eq:defgss} and~\eqref{eq:aliasing} applied to~$\Delta a$ give the first equality below. The cosine series converges by Dirichlet's test, and a summation by parts gives the second equality.
\begin{equation*}
    {-2(1-\cos\theta_k)\mathcal F_d(g_{N,s})(\theta_k)
    =\sum_{r\in\mathbb Z}\Delta a(r)e^{ir\theta_k}
    =-2(1-\cos\theta_k)\Bigl(2+2\sum_{r=1}^\infty\frac{\cos(r\theta_k)}{r^s}\Bigr)}
    \,.
\end{equation*}
Dividing by~$-2(1-\cos\theta_k)\neq0$ gives
\begin{equation}\label{eq:Fdgs}
    \mathcal{F}_d(g_{N,s})(\theta_k)=S(\theta_k)
    \,,\qquad
    S(\theta)=2+2\sum_{r=1}^{\infty}\frac{\cos(r\theta)}{r^s}
    \,.
\end{equation}
Convexity gives $\Delta a(r)>0$ for $r\geq1$, with $\Delta a(r)\asymp r^{-s-2}$. Summation by parts yields
\begin{equation*}
    S(\theta)=\sum_{r=1}^{\infty}\Delta a(r)\frac{1-\cos(r\theta)}{1-\cos\theta}
    \geq\Delta a(1)=2^{-s}
    \,.
\end{equation*}
For~$1\leq r\leq m_N$, \eqref{eq:defgss} gives~$g_{N,s}''(r)\asymp r^{-s-2}$. Since $g_{N,s}(m_N)=0$, two telescoping sums give, for $0\leq d\leq m_N$,
\begin{equation}\label{eq:kernel intervals}
    g_{N,s}(d)=\sum_{r=d+1}^{m_N}(r-d)
    \left(1-\frac{1}{2}\mathds{1}_{2r=N}\right)g_{N,s}''(r)
    \,.
\end{equation}
Summing over the cycle therefore gives, for $\vartheta_N=\mathcal F_d(g_{N,s})(0)$,
\begin{equation*}
    {\vartheta_N}=\sum_{r=1}^{m_N}
    \left(1-\frac{1}{2}\mathds{1}_{2r=N}\right)r^2g_{N,s}''(r)
    \asymp N^{1-s}
    \,.
\end{equation*}
All comparison constants depend only on~$s$. The periodic-zeta identity~\cite[Eq.~(25.13.3)]{dlmf-periodic-zeta} gives, for $0<\theta<2\pi$,
\begin{equation}\label{eq:defS}
    S(\theta)=2+\frac{(2\pi)^s}{2\Gamma(s)\cos(\pi s/2)}
    \left(\zeta\left(1-s,\frac{\theta}{2\pi}\right)
    +\zeta\left(1-s,1-\frac{\theta}{2\pi}\right)\right)
    \,.
\end{equation}
Using $\zeta(1-s,x)=x^{s-1}+\zeta(1-s,1+x)$ near zero and the positivity of~$S$, we obtain
\begin{equation}\label{eq:lowerS}
    \frac{c(s)}{\rho(\theta)^{1-s}}\leq S(\theta)\leq\frac{C(s)}{\rho(\theta)^{1-s}}
    \,.
\end{equation}

\paragraph{\emph{Step 2: the inverse}}
The matrix~$\mathbb G_{N,s}$ is circulant, with eigenvalues~$\vartheta_N>0$ and~$S(\theta_k)\geq2^{-s}$ for $1\leq k\leq N-1$. It is therefore positive definite, and Fourier inversion gives
\begin{equation*}
    g_{N,s}^{-1}(h)=\frac1N\Bigl(\vartheta_N^{-1}+\sum_{k=1}^{N-1}S(\theta_k)^{-1}e^{-ih\theta_k}\Bigr)
    \,.
\end{equation*}
Let $m=1/S$ on $(0,2\pi)$ and $m(0)=0$, and extend~$m$ periodically. Differentiating~\eqref{eq:defS} and using~\eqref{eq:lowerS} gives, for $j\in\{0,1,2\}$,
\begin{equation*}
    |m^{(j)}(\theta)|\leq C(s)\rho(\theta)^{1-s-j}
    \,.
\end{equation*}
We apply the first part of Lemma~\ref{lemma:multiplier kernels} with $r=2$, $\alpha=1-s$ and $z_0=\vartheta_N^{-1}\leq C(s)N^{s-1}$. It gives~\eqref{eq:inv sum} and, since $d_N(h,0)\leq N/2$, the bound~\eqref{eq:inverse HN}.
\end{proof}

\begin{remark}[Discrete fractional primitive]\label{remark:discrete}
By~\eqref{eq:Fdgs} and~\eqref{eq:lowerS}, convolution with $g_{N,s}$ on zero-sum vectors has Fourier multiplier comparable to that of a fractional primitive of order~$1-s$.
\end{remark}

\subsection{The Riesz matrix as a sum over intervals}
The quadratic form of~$\mathbb G_{N,s}$ is a sum of squares of interval sums, with coefficients given by the second differences of~$g_{N,s}$. We use this representation to compare~$\mathbb G_{N,s}$ with the Hessian of the energy, and to bound truncations of~$\mathbb G_{N,s}$ from below.

For $k\neq0$, \eqref{eq:defgss} gives
\begin{equation}\label{eq:gss bounds}
    \frac{c(s)}{d_N(k,0)^{s+2}}\leq g_{N,s}''(k)\leq\frac{C(s)}{d_N(k,0)^{s+2}}
    \,.
\end{equation}
Let $U_N=(u_1,\ldots,u_N)\in\dR^N$. Counting the cyclic intervals containing two prescribed indices in~\eqref{eq:kernel intervals} gives, with indices taken modulo~$N$,
\begin{equation}\label{eq:riesz quadratic form}
    U_N\cdot{\mathbb G}_{N,s}U_N
    =\sum_{i=1}^N\sum_{r=1}^{m_N}
    \left(1-\frac{1}{2}\mathds{1}_{2r=N}\right)
    g_{N,s}''(r)(u_i+\ldots+u_{i+r-1})^2
    \,.
\end{equation}
The coefficients of the squares in this sum are positive. Let $1\leq K_0\leq N/4$ be an integer. Keeping only $r\geq K_0$ in~\eqref{eq:riesz quadratic form} defines a nonnegative truncated matrix ${\mathbb G}_{N,s}^{[K_0]}$. Let $K_0\leq r<2K_0$. Since $u_i=(u_i+\cdots+u_{i+r})-(u_{i+1}+\cdots+u_{i+r})$, we have
\begin{equation*}
    {\sum_{i=1}^Nu_i^2
    \leq2\sum_{i=1}^N\bigl[(u_i+\cdots+u_{i+r})^2+(u_{i+1}+\cdots+u_{i+r})^2\bigr]}
    \,.
\end{equation*}
Summing over $K_0\leq r<2K_0$ bounds $K_0|U_N|^2/4$ by the sum of the squares $(u_i+\cdots+u_{i+r-1})^2$ over $1\leq i\leq N$ and $K_0\leq r\leq2K_0$. The coefficients of these squares in~\eqref{eq:riesz quadratic form} are at least $c(s)K_0^{-s-2}$. Hence
\begin{equation}\label{eq:truncated coercivity}
    {\mathbb G}_{N,s}^{[K_0]}\geq c(s)K_0^{-1-s}I_N
    \,.
\end{equation}

\section{The Helffer--Sjöstrand equation}\label{section:helffer c}
This section collects standard results on the Helffer--Sjöstrand equation. We first recall basic properties valid for convex Gibbs measures on~$D_N$. We then rewrite the Helffer--Sjöstrand equation in gap coordinates. For the Gibbs measures of this paper, the energy is a convex function of the gaps. We use this convexity to derive the Brascamp--Lieb inequality, a maximum principle for solutions and concentration inequalities.

\subsection{Well-posedness}
We start by explaining the principle of the Helffer--Sjöstrand representation and give existence and uniqueness results. We follow the formulation of~\cite[Subsection~3.1]{boursier2021optimal} and the presentation of~\cite{armstrong2019c}. Let $\mu$ be a probability measure on $D_N$ of the form \[\dd\mu=e^{-H(X_N)}\mathds{1}_{D_N}(X_N)\dd X_N,\] where $H:D_N\to \dR$ is a smooth and convex function. Given a smooth test function $F:D_N\to \dR$, we rewrite its variance through the solution of a Poisson equation. Let $\mc{L}^{\mu}$ be the operator acting on $\mathcal{C}^{\infty}(D_N,\dR)$ given by
\begin{equation*}
    \mc{L}^{\mu}=\nabla H\cdot \nabla-\Delta
    \,,
\end{equation*}
where $\nabla$ and $\Delta$ are the standard gradient and Laplace operators on $\dT^N$. The Neumann realization of $-\mc{L}^\mu$ is the generator of the Langevin dynamics associated with the energy $H$ of which $\mu$ is the unique invariant measure. For smooth functions $f,h$ supported away from collisions and periodic in the common rotation, an integration by parts under the Lebesgue measure on~$D_N$ gives
\begin{equation}\label{eq:informal IPP c}
    \dE_{\mu}[h\,\mc{L}^{\mu}f]=\dE_{\mu}[\nabla f\cdot \nabla h]
    \,.
\end{equation}

If the Poisson equation
\begin{equation}\label{eq:Poisson c}
    \left\{
    \begin{array}{ll}
       \mc{L}^{\mu}\phi=F-\dE_{\mu}[F]  & \text{on }D_N \\
       \nabla \phi\cdot \vec{n}=0  & \text{on }\partial D_N
    \end{array}
    \right.
\end{equation}
has a weak solution~$\phi$, then~\eqref{eq:informal IPP c} with $f=\phi$ and $h=F$ expresses the variance of $F$ under $\mu$ as
\begin{equation*}
    \Var_{\mu}[F]=\dE_{\mu}[\nabla F\cdot \nabla \phi]
    \,.
\end{equation*}
The above identity is called the Helffer--Sjöstrand representation formula. Let
\begin{equation*}
    A_1^{\mu}=\nabla^2 H+\mc{L}^{\mu}\otimes I_N
    \,,
\end{equation*}
with $\mc{L}^{\mu}\otimes I_N$ acting diagonally on $\mathcal C^\infty(D_N;\dR^N)$. For every $\phi\in \mathcal{C}^{\infty}(D_N,\dR)$, we have $\nabla \mc{L}^{\mu}\phi= A_1^{\mu}\nabla \phi$. Therefore the solution $\nabla\phi$ of~\eqref{eq:Poisson c} formally satisfies
\begin{equation*}
\left\{\begin{array}{ll}
   A_1^{\mu}\nabla\phi=\nabla F &\text{on }D_N\\
   \nabla \phi\cdot \vec{n}=0 &\text{on }\partial D_N.
   \end{array}
   \right.
\end{equation*}
This partial differential equation is called the Helffer--Sjöstrand equation. To make these derivations rigorous, we define the norm
\begin{equation*}
    \Vert F\Vert_{H^{1}(\mu) }=\dE_{\mu}[F^2]^{\frac{1}{2}}+\dE_{\mu}[|\nabla F|^2]^{\frac{1}{2}}
    \,.
\end{equation*}
Let $H^1(\mu)$ be the completion of $\mathcal{C}^{\infty}(D_N)$ with respect to the norm $\Vert \cdot \Vert_{H^1(\mu)}$. Equation~\eqref{eq:informal IPP c} extends to $f$ in the Neumann generator domain and $h\in H^1(\mu)$. We also define the norm
\begin{equation*}
    \Vert F \Vert_{H^{-1}(\mu) }=\sup \{ |\dE_{\mu}[F G]|:G\in H^1(\mu), \Vert G\Vert_{H^1(\mu)}\leq 1\}
    \,.
\end{equation*}
We denote by $H^{-1}(\mu)$ the dual of $H^1(\mu)$, that is the completion of $\mathcal{C}^{\infty}(D_N)$ with respect to the norm $\Vert \cdot \Vert_{H^{-1}(\mu)}$. We make the following assumptions on~$\mu$.

\begin{assumptions}\label{assumptions:gibbs measure c}
Let $N\geq2$, write configurations in cyclic real lifts $x_1<\cdots<x_N<x_1+1$, and let $d_\partial(X)=\operatorname{dist}(X,\partial D_N)$. Assume that:
\begin{enumerate}
    \item The measure~$\mu$ has density proportional to~$e^{-H}$, where $H=2\sum_{i<j}g_{ij}(x_j-x_i)$ and each~$g_{ij}$ belongs to~$\mathcal C^\infty(0,1)$.
    \item There exists $c>0$ such that $g_{ij}''\geq c$ for every $i<j$, and each~$g_{ij}$ tends to~$+\infty$ at both endpoints of~$(0,1)$.
    \item For every $a>0$, there exists $C(a,N)<\infty$ such that
    \begin{equation}\label{eq:collision criterion}
        a|\nabla H|^2-\Delta H\geq d_\partial^{-2}-C(a,N)
        \,.
    \end{equation}
\end{enumerate}
\end{assumptions}

Condition~\eqref{eq:collision criterion} holds for the Riesz potentials and convex penalties used here, by Lemma~\ref{lemma:collision}.

Define the vector quadratic form
\begin{equation}\label{eq:HS closed form}
    \mathfrak q_\mu(u,v)=\dE_\mu\left[\sum_i\nabla u_i\cdot\nabla v_i
        +u\cdot\nabla^2H v\right]
    \,.
\end{equation}
Let $\mathcal V_\mu$ be the completion of smooth vector fields supported away from collisions, periodic in the common rotation, for the norm $(\|u\|_2^2+\mathfrak q_\mu(u,u))^{1/2}$. We interpret $A_1^\mu u=v$ by $\mathfrak q_\mu(u,h)=\langle v,h\rangle$ for every $h\in\mathcal V_\mu$, where $\langle v,h\rangle$ is the duality pairing of~$\mathcal V_\mu$, equal to~$\dE_\mu[v\cdot h]$ when $v\in L^2(\mu)$. The Neumann boundary conditions in the equations below refer to this weak formulation.

We now state the well-posedness result.

\begin{proposition}[Existence and representation]\label{proposition: existence c}
Let $\mu$ satisfy Assumptions~\ref{assumptions:gibbs measure c} and let $F\in H^1(\mu)$. There exists a unique gradient $\nabla\phi\in\mathcal V_\mu$, with $\dE_\mu\sum_i\partial_i\phi=0$, such that
\begin{equation}\label{eq:the HS c}
    \left\{
    \begin{array}{ll}
        A_1^{\mu}\nabla \phi=\nabla F & \text{ on }D_N \\
        \nabla \phi\cdot \vec{n}=0 & \text{ on }\partial D_N,
    \end{array}
    \right.
\end{equation}
with the first identity understood in the dual of $\mathcal V_\mu$. Moreover the solution of~\eqref{eq:the HS c} is the unique minimizer of the functional
\begin{equation*}
u\mapsto\mathfrak q_\mu(u,u)-2\dE_\mu[\nabla F\cdot u]
    \,,
\end{equation*}
over $u\in\mathcal V_\mu$ with $\dE_\mu\sum_i u_i=0$. The variance of $F$ may be represented as
\begin{equation*}
    \Var_{\mu }[F]=\dE_{\mu}[\nabla \phi\cdot \nabla F ]
\end{equation*}
and the covariance between $F$ and a function $G\in H^1(\mu)$ as
\begin{equation*}
    \Cov_{\mu}[F,G]=\dE_{\mu}[\nabla \phi\cdot \nabla G]
    \,.
\end{equation*}
\end{proposition}

Appendix~\ref{section:existence c} proves Proposition~\ref{proposition: existence c}.

\subsection{Change of coordinates}
From now on, we study the decay of correlations in gap coordinates. Define the map
\begin{equation*}
   \Gap_N:X_N\in D_N\mapsto (N(x_2-x_1),\ldots,N(x_{N+1}-x_N))\in\mc{M}_N
    \,,
\end{equation*}
where
\begin{equation*}
    \mc{M}_N=\Gap_N(D_N)=\{Y_N\in(0,+\infty)^N:y_1+\ldots+y_N=N\}
    \,.
\end{equation*}
Since $\mc{M}_N$ is not an open subset of~$\dR^N$, we adapt Proposition~\ref{proposition: existence c} as follows. Let $\mu$ satisfy Assumptions~\ref{assumptions:gibbs measure c}, and let $H^\g:\mc{M}_N\to\dR$ be such that
\begin{equation*}
    H=H^\g\circ \Gap_N
    \,.
\end{equation*}
Let $\nu=\Gap_N\#\mu$, let $\dGi^\g=\Gap_N\#\dGi$ be the law of the microscopic gaps under~$\dGi$, and equip~$\mc M_N$ with the Euclidean metric of its tangent space~$\mathbf1^\perp$. We define~$H^1(\nu)$ as~$H^1(\mu)$, with the gradient of this metric. Let $\mc L^\nu$ be the operator acting on $\mathcal{C}^{\infty}(\mc{M}_N,\dR)$ given by
\begin{equation*}
    \mc{L}^{\nu}=\nabla H^{\mathrm{g}}\cdot \nabla -\Delta
    \,,
\end{equation*}
with $\nabla$ and $\Delta$ the standard gradient and Laplace operator on $\mc{M}_N$. Also let $A_1^\nu$ be the operator on $\mathcal C^\infty(\mc M_N;\dR^N)$ given by
\begin{equation*}
    A_1^{\nu}=\nabla^2 H^{\mathrm{g}}+\mc{L}^{\nu}\otimes I_N
    \,,
\end{equation*}
where~$\nabla^2H^\g$ is the Hessian of an extension of~$H^\g$ to a neighborhood of~$\mc M_N$ in~$\dR^N$.

For a function~${F=\chi\circ\Gap_N}$, we apply the Helffer--Sjöstrand representation formula to~$\chi$ under~$\nu$. Let~$\psi$ solve $\mc L^\nu\psi=\chi-\dE_\nu[\chi]$. The gradient~$\nabla\psi$ is tangent to~$\mc M_N$, and differentiating this equation in the tangent directions gives
\begin{equation*}
    A_1^\nu\nabla\psi=\nabla\chi+\lambda\mathbf1,
    \qquad \mathbf1\cdot\nabla\psi=0
    \,,
\end{equation*}
where~$\lambda$ is the Lagrange multiplier of the constraint $y_1+\cdots+y_N=N$. In this equation, $\nabla\chi$ is computed from an extension of~$\chi$. The tangent components of the equation do not depend on the extensions, which only change~$\lambda$. Let~$\mathfrak q_\nu$ and~$\mathcal V_\nu$ be defined as~$\mathfrak q_\mu$ and~$\mathcal V_\mu$, with tangent fields and the tangent Hessian of~$H^\g$ in~\eqref{eq:HS closed form}.

\begin{proposition}\label{proposition: HS gaps c}
Let $\mu$ satisfy Assumptions~\ref{assumptions:gibbs measure c}, and let $\nu=\Gap_N\#\mu$. Let $F=\chi\circ\Gap_N$ with $\chi\in H^1(\nu)$. There exists a unique gradient $\nabla\psi\in\mathcal V_\nu$ solving
\begin{equation}\label{eq:HS gaps form c}
    \left\{
    \begin{array}{ll}
        A_1^{\nu} \nabla\psi= \nabla\chi+\lambda\mathbf1 & \text{on } \mc{M}_N \\
        \nabla \psi\cdot\mathbf1=0 & \text{on }\mc{M}_N\\
     \nabla\psi\cdot \vec{n}=0 &\text{on }\partial \mc{M}_N,
    \end{array}
    \right.
\end{equation}
with the distribution~$\lambda$ given by
\begin{equation}\label{eq:exp la}
    \lambda=\frac{1}{N}\mathbf1\cdot (\nabla^2 H^\g \nabla\psi-\nabla\chi)
    \,.
\end{equation}
The variance of $F$ can be represented as
\begin{equation*}
    \Var_{\mu}[F]=\dE_{\nu}[\nabla\chi\cdot \nabla \psi]
    \,.
\end{equation*}
Furthermore, $\nabla\psi$ is the unique minimizer of
\begin{equation*}
    u\mapsto\mathfrak q_\nu(u,u)-2\dE_\nu[\nabla\chi\cdot u]
\end{equation*}
over tangent gradient fields~$u$ in~$\mathcal V_\nu$.
\end{proposition}

The proof of Proposition~\ref{proposition: HS gaps c} is given in Appendix~\ref{section:existence c}.

We now fix the ambient extension of the energy~\eqref{def:Hn} off~$\mc M_N$. For $U_N=(u_1,\ldots,u_N)\in\dR^N$, let~$L_{i,r}(U_N)$ be the sum~${u_i+\ldots+u_{i+r-1}}$, with cyclic indices, and write $L_J=L_{i,r}(Y_N)$ for the interval $J=\{i,\ldots,i+r-1\}$. The extension is
\begin{equation}\label{eq:ambient extension}
    \Hc_N^\g(Y_N)=N^{-s}\sum_{i=1}^N\sum_{r=1}^{m_N}
    (2-\mathds{1}_{2r=N})g_s\left(\frac{L_{i,r}(Y_N)}{N}\right)
    \,.
\end{equation}
On~$\mc M_N$, the number $L_{i,r}(Y_N)/N$ is the distance between the particles~$i$ and~$i+r$, so~\eqref{eq:ambient extension} agrees with~\eqref{def:Hn}. The Hessian of this extension is given by
\begin{equation}\label{eq:expHg2}
    U_N\cdot\nabla^2\Hc_N^\g(Y_N)U_N=N^{-s-2}\sum_{i=1}^N\sum_{r=1}^{m_N}
    (2-\mathds{1}_{2r=N})g_s''\left(\frac{L_{i,r}(Y_N)}{N}\right)L_{i,r}(U_N)^2
    \,.
\end{equation}
For $s\in(0,1)$, Taylor's formula gives $|N^{-s-2}g_s''(r/N)-g_{N,s}''(r)|\leq C(s)r^{-s-4}$ for $2\leq r\leq m_N$, and a bounded error for $r=1$. Summing over intervals containing $i$ and $j$ in~\eqref{eq:expHg2} and~\eqref{eq:riesz quadratic form} gives, at the equally spaced configuration~$\mathbf1$,
\begin{equation}\label{eq:unit gap comparison}
    \left|\left(\nabla^2\Hc_N^\g(\mathbf{1})-2\mathbb G_{N,s}\right)_{i,j}\right|
    \leq\frac{C(s)}{(1+d_N(i,j))^{s+2}}
    \,.
\end{equation}

\subsection{The Brascamp--Lieb inequality}
We now recall the Brascamp--Lieb inequality~\cite[Theorem~4.1]{brascamplieb}, which bounds a variance under a strictly log-concave measure by the inverse of the Hessian of its energy. The measure $\mu$ is not strictly log-concave, but its push-forward $\nu$ is. Since~$\nabla^2H$ vanishes in the direction~$\mathbf1$, we write $(\nabla^2H)^{-1}$ for its inverse on~$\mathbf1^\perp$. This gives the following bound on the variance of functions of the gaps.

\begin{lemma}\label{lemma:brascamp Lieb}
Let $\mu$ satisfy Assumptions~\ref{assumptions:gibbs measure c}, and let $\nu=\Gap_N\#\mu$. Let $\Omega\subset D_N$ be a domain with a piecewise smooth boundary and $\mu(\Omega)>0$, whose cyclic lift with $0<x_1<1$ is convex. Let $F=\chi\circ \Gap_N$ with $\chi\in H^1(\nu)$. Then
\begin{equation*}
    \Var_{\mu}[F\mid \Omega]\leq \dE_{\mu}[\nabla F\cdot (\nabla^2 H)^{-1}\nabla F\mid\Omega]
    \,.
\end{equation*}
\end{lemma}

\subsection{Convexity of marginals}
In Sections~\ref{section:decay short} and~\ref{section:large scale}, we study the marginal of the Gibbs measure on a window of gaps. By the following convexity lemma of Brascamp and Lieb~\cite{brascamplieb}, its energy remains convex. Like Lemma~\ref{lemma:brascamp Lieb}, it is based on the Brascamp--Lieb inequality.

\begin{lemma}\label{lemma:the convexity lemma}
Let $\mu$ be a measure on $D_N$ of the form $\dd\mu=e^{-H}\dd X_N$, with $H\in\mc C^2(D_N)$. Write $X_N=(x,y)$, with $x=(x_1,\ldots,x_n)$ and $y=(x_{n+1},\ldots,x_N)$. Assume that $H$ can be written as $H(x,y)=H_1(x)+H_2(x,y)$ with $\nabla^2 H_2$ nonnegative. Let $\widetilde{\mu}$ be the push-forward of $\mu$ by the map $X_N\mapsto (x_1,\ldots,x_n)$. Then $\dd\widetilde{\mu}(x)=e^{-\widetilde{H}(x)}\dd x$, with
\begin{equation*}
    \widetilde{H}(x)=-\log\int e^{-H(x,y)} \dd y
\end{equation*}
and $\widetilde{H}-H_1$ is convex. In particular, in the sense of distributions,
\begin{equation*}
    \nabla^2 \widetilde{H}\geq \nabla^2 H_1
    \,.
\end{equation*}
If the domain of~$y$ does not depend on~$x$, and if, locally uniformly in~$x$, the functions $e^{-H_2}$, $|\nabla_xH_2|^2e^{-H_2}$ and $|\nabla_x^2H_2|e^{-H_2}$ are bounded by an integrable function of~$y$, then for $1\leq i,j\leq n$,
\begin{equation}\label{eq:exp gradient ext}
   \partial_i\widetilde H=\partial_iH_1+\dE_{\mu(\cdot\mid x)}[\partial_iH_2]
   \,,
\end{equation}
\begin{equation}\label{eq:exp hess ext}
   \partial_{ij}\widetilde H=\partial_{ij}H_1
   +\dE_{\mu(\cdot\mid x)}[\partial_{ij}H_2]
   -\Cov_{\mu(\cdot\mid x)}[\partial_iH_2,\partial_jH_2]
   \,.
\end{equation}
\end{lemma}

\subsection{Maximum principle}
In this subsection, the unknown~$\psi$ is a vector field, which plays the role of~$\nabla\psi$ in Proposition~\ref{proposition: HS gaps c}. We derive a maximum principle, which bounds the essential supremum of~$|\psi|$ when the energy is uniformly convex. This maximum principle is standard on~$\dR^N$~\cite[Section~10]{helffer1994correlation}, and we state it for a general measure~$\nu$.

\begin{proposition}\label{proposition:maximum principle}
Let $\nu$ be a probability measure with closed scalar Dirichlet form $\dE_\nu[|\nabla f|^2]$ and nonnegative generator $\mc L^\nu$, such that $H^1(\nu)$ satisfies the chain rule for Lipschitz radial truncations of vector fields. Let $c>0$ and $m\geq1$, and let $M$ be a measurable field of symmetric $m\times m$ matrices such that, for every $u\in\dR^m$,
\begin{equation*}
    u\cdot M u\geq c|u|^2
    \,.
\end{equation*}
Let $v\in L^2(\nu;\dR^m)$, and let $\psi\in H^1(\nu;\dR^m)$ with $\dE_\nu[\psi\cdot M\psi]<\infty$ solve
\begin{equation*}
    (M+\mc L^\nu\otimes I_m)\psi=v
\end{equation*}
in the weak sense. Then
\begin{equation}\label{eq:max pr}
    \|\psi\|_{L^\infty(\nu)}\leq c^{-1}\|v\|_{L^\infty(\nu)}
    \,.
\end{equation}
\end{proposition}

\begin{proof}
We may assume that $v$ is essentially bounded. Let $R=c^{-1}\|v\|_{L^\infty(\nu)}$. We test the weak equation with ${h=\psi(1-R/|\psi|)_+}$, defined as zero when $\psi=0$. By the chain rule for Lipschitz radial truncations, $h\in H^1(\nu;\dR^m)$. The Dirichlet term of the tested equation is nonnegative, since
\begin{equation*}
    {\sum_i\nabla\psi_i\cdot\nabla h_i
    =\mathds1_{|\psi|>R}\Bigl[\Bigl(1-\frac R{|\psi|}\Bigr)\sum_i|\nabla\psi_i|^2
    +\frac R{|\psi|}|\nabla|\psi||^2\Bigr]\geq0}
    \,.
\end{equation*}
Since $|v|\leq cR$ and $\psi\cdot M\psi\geq c|\psi|^2$, we have $(M\psi-v)\cdot h\geq c(|\psi|-R)_+^2$ pointwise. Hence
\begin{equation*}
    0\geq\dE_\nu[(M\psi-v)\cdot h]
      \geq c\,\dE_\nu[(|\psi|-R)_+^2]\,.
\end{equation*}
This proves~\eqref{eq:max pr}.
\end{proof}

\subsection{The Bakry--\'Emery criterion}
We use the Bakry--\'Emery criterion~\cite{bakryemery} in the following form.

\begin{lemma}\label{lemma:bakry}
Let $\mu$ be a uniformly log-concave measure on a convex domain of $\dR^N$, with a convexity constant larger than $c>0$. Then $\mu$ satisfies the logarithmic Sobolev inequality with constant $2c^{-1}$.
\end{lemma}

\section{Decay of correlations for the short-range Riesz gas}\label{section:decay short}
Throughout this section, $s>1$. We prove Theorem~\ref{theorem:hypersing} on the decay of correlations for the short-range gas. We first treat uniformly coercive matrices with summable entries, then compare them with the Hessian of the marginal energy.

\subsection{\texorpdfstring{{The deformation of Combes and Thomas}}{The deformation of Combes and Thomas}}
This subsection carries out the deformation of Combes and Thomas outlined in Subsection~\ref{subsection:outline}. The measure~$\nu$ lives on a gap simplex or on a projection of it, such as the marginals~\eqref{eq:def nu hypersing}. The equations below have~$m$ components, on which~$\mc L^\nu$ acts componentwise, and~$m$ may differ from the dimension of the domain of~$\nu$. The unknown~$\psi$ is a vector, whose coordinates~$\psi_i$ play the role of the derivatives~$\partial_i\psi$ in the outline.

\begin{assumptions}\label{assumptions:mu}
Let~$\nu$ be a probability measure on an open convex subset of a Euclidean affine space. Assume that:
\begin{enumerate}
\item The scalar Dirichlet form of~$\nu$ is closed and conservative, its generator~$\mc L^\nu$ is nonnegative, and, for every~$f$ in the domain of~$\mc L^\nu$ and every $h\in H^1(\nu)$,
    \begin{equation*}
        \dE_\nu[h\,\mc L^\nu f]=\dE_\nu[\nabla f\cdot\nabla h]
        \,.
    \end{equation*}
\item The space~$H^1(\nu)$ satisfies the chain rule for Lipschitz radial truncations of vector fields.
\end{enumerate}
\end{assumptions}

\begin{assumptions}\label{assumptions:sM1}
Let $s>1$ and $\ve>0$, and let $m\leq n$ be positive integers. Let~$M$ be a measurable field of symmetric $m\times m$ matrices such that, uniformly in the configuration,
\begin{equation*}
    M\geq C_n^{-1}I_m,
    \qquad |M_{ij}|\leq C_n(1+d(i,j))^{-s}
    \,.
\end{equation*}
\end{assumptions}

For $\alpha>0$ and an index~$i_0$, let $W_\alpha=\operatorname{diag}((1+d(i,i_0))^\alpha)_{1\leq i\leq m}$. Multiplying the solution by~$W_\alpha$ adds to the equation the commutator~$W_\alpha MW_\alpha^{-1}-M$, as in~\eqref{eq:outline split}. We first bound this commutator.

\begin{lemma}[Commutator estimate]\label{lemma:perturbation lemma}
Let $s>1$ and $m\geq1$, let the indices lie in an interval or a cycle with distance~$d$, and let~$i_0$ be an index. Let $B>0$ and let $M\in\dR^{m\times m}$ satisfy $|M_{ij}|\leq B(1+d(i,j))^{-s}$. Let $1/2<\alpha<s-1/2$. Then, for every $0<\gamma<\min(1,s-1/2-\alpha)$, there exists $C=C(s,\alpha,\gamma)$ such that, for every $K\geq1$ and every $u\in\dR^m$,
\begin{equation*}
    |u\cdot(W_\alpha MW_\alpha^{-1}-M)u|
    \leq CBK^{-\gamma}|u|^2
       +CBK^{1/2}|u|\,|u\mathds{1}_{d(\cdot,i_0)\leq2K}|
    \,.
\end{equation*}
\end{lemma}

\begin{proof}
We split the quadratic form into the pairs~$(i,j)$ with $d(i,i_0)>2K$ and $d(j,i_0)>2K$, and the remaining pairs. Write $r_i=1+d(i,i_0)$. The entries of the commutator are
\begin{equation*}
    (W_\alpha MW_\alpha^{-1}-M)_{ij}=((r_i/r_j)^\alpha-1)M_{ij}
    \,.
\end{equation*}
The triangle inequality gives $|r_i-r_j|\leq d(i,j)$, hence $r_i\leq(1+d(i,j))r_j$ and
\begin{equation}\label{eq:commutator entry}
    |(W_\alpha MW_\alpha^{-1}-M)_{ij}|\leq 2B(1+d(i,j))^{\alpha-s}
    \,.
\end{equation}

We first bound the pairs with $d(i,i_0)>2K$ and $d(j,i_0)>2K$, for which $r_i\geq K$ and $r_j\geq K$. If $r_j/2\leq r_i\leq2r_j$, then $|r_i/r_j-1|\leq1$. Since $x\mapsto x^\alpha$ is Lipschitz on~$[1/2,2]$ and $\gamma<1$, it follows that
\begin{equation*}
    \Bigl|\Bigl(\frac{r_i}{r_j}\Bigr)^\alpha-1\Bigr|
    \leq C\Bigl|\frac{r_i}{r_j}-1\Bigr|^\gamma
    \leq C\Bigl(\frac{d(i,j)}{r_j}\Bigr)^\gamma
    \leq CK^{-\gamma}(1+d(i,j))^\gamma
    \,,
\end{equation*}
so that $|(W_\alpha MW_\alpha^{-1}-M)_{ij}|\leq CBK^{-\gamma}(1+d(i,j))^{\gamma-s}$. If $r_i\geq2r_j$, then $d(i,j)\geq r_i-r_j\geq r_i/2$ and $|(r_i/r_j)^\alpha-1|\leq r_i^\alpha r_j^{-\alpha}$. Since $r_i\geq K$, this gives
\begin{equation*}
    |(W_\alpha MW_\alpha^{-1}-M)_{ij}|
    \leq 2^sBr_i^{\alpha-s}r_j^{-\alpha}
    \leq 2^sBK^{-\gamma}r_i^{\alpha+\gamma-s}r_j^{-\alpha}
    \,.
\end{equation*}
If $r_j\geq2r_i$, then $|(r_i/r_j)^\alpha-1|\leq|(r_j/r_i)^\alpha-1|$, and the previous case with $i$ and~$j$ exchanged gives $|(W_\alpha MW_\alpha^{-1}-M)_{ij}|\leq2^sBK^{-\gamma}r_j^{\alpha+\gamma-s}r_i^{-\alpha}$.

We now sum these bounds. Since $s-\gamma>1$, the kernel $(1+d(i,j))^{\gamma-s}$ has bounded row and column sums, and the Schur test gives $\sum_{i,j}|u_i|(1+d(i,j))^{\gamma-s}|u_j|\leq C|u|^2$. Since $2\alpha>1$ and $2(s-\alpha-\gamma)>1$, the Cauchy--Schwarz inequality gives
\begin{equation*}
    \sum_{i,j}|u_i|r_i^{\alpha+\gamma-s}r_j^{-\alpha}|u_j|
    \leq\Bigl(\sum_ir_i^{2(\alpha+\gamma-s)}\Bigr)^{1/2}\Bigl(\sum_jr_j^{-2\alpha}\Bigr)^{1/2}|u|^2
    \leq C|u|^2
    \,.
\end{equation*}
The kernel $r_j^{\alpha+\gamma-s}r_i^{-\alpha}$ is bounded in the same way. Therefore the pairs with $d(i,i_0)>2K$ and $d(j,i_0)>2K$ contribute at most $CBK^{-\gamma}|u|^2$.

We now bound the pairs with $d(j,i_0)\leq2K$. Since $2(s-\alpha)>1$, the bound~\eqref{eq:commutator entry} shows that every column of $W_\alpha MW_\alpha^{-1}-M$ has Euclidean norm at most~$CB$. At most $4K+1\leq5K$ indices satisfy $d(j,i_0)\leq2K$. The Cauchy--Schwarz inequality in~$i$, then in~$j$, gives
\begin{equation*}
    \sum_{j:d(j,i_0)\leq2K}\sum_i|u_i(W_\alpha MW_\alpha^{-1}-M)_{ij}u_j|
    \leq CB|u|\sum_{j:d(j,i_0)\leq2K}|u_j|
    \leq CBK^{1/2}|u|\,|u\mathds{1}_{d(\cdot,i_0)\leq2K}|
    \,.
\end{equation*}
The pairs with $d(i,i_0)\leq2K$ are bounded in the same way, with rows instead of columns. Adding the three bounds proves the lemma.
\end{proof}

We now deduce a first decay estimate for the equation with a point source at~$i_0$.

\begin{lemma}[{First decay estimate}]\label{lemma:illustrative}
Under Assumptions~\ref{assumptions:mu}--\ref{assumptions:sM1}, let $1\leq i_0\leq m$, let~$\chi\in L^2(\nu)$ and let~$\psi\in H^1(\nu;\dR^m)$ be the weak solution of
\begin{equation}\label{eq:illus}
    (\beta M+\mc L^\nu)\psi=\chi e_{i_0}
    {\,.}
\end{equation}
Then, for every~$1/2<\alpha<s-1/2$,
\begin{equation}\label{eq:first illus}
    \|W_\alpha\psi\|_2+\|W_\alpha\nabla\psi\|_2
    \leq C_n\|\chi\|_2
    \,.
\end{equation}
\end{lemma}

\begin{proof}
Testing~\eqref{eq:illus} against~$\psi$ and using $M\geq C_n^{-1}I_m$ gives the unweighted energy estimate $\|\psi\|_2+\|\nabla\psi\|_2\leq C_n\|\chi\|_2$. Put $u=W_\alpha\psi$. Its equation is
\begin{equation*}
    (\beta M+\mc L^\nu)u+\beta(W_\alpha MW_\alpha^{-1}-M)u
    =\chi e_{i_0}
    \,.
\end{equation*}
We test this equation against~$u$. Since~$u_{i_0}=\psi_{i_0}$, this gives
\begin{equation*}
    \|\nabla u\|_2^2+\beta\dE_\nu[u\cdot Mu]
    =\dE_\nu[\chi\psi_{i_0}]-\beta\dE_\nu[u\cdot(W_\alpha MW_\alpha^{-1}-M)u]
    \,.
\end{equation*}
By Assumptions~\ref{assumptions:sM1}, the left-hand side is at least~$\|\nabla u\|_2^2+\beta C_n^{-1}\|u\|_2^2$. We apply Lemma~\ref{lemma:perturbation lemma} with~$B=C_n$ and~$\gamma=\min(1,s-1/2-\alpha)/2$. Since~$|u\mathds{1}_{d(\cdot,i_0)\leq2K}|\leq(1+2K)^\alpha|\psi|$, the Cauchy--Schwarz inequality bounds the last term by
\begin{equation*}
    C_nK^{-\gamma}\|u\|_2^2+C_nK^{1/2+\alpha}\|u\|_2\|\psi\|_2
    \,.
\end{equation*}
We choose~$K$ of the form~$C_n$, large enough that the first term is at most~$\beta C_n^{-1}\|u\|_2^2/2$. Young's inequality and the unweighted estimate then give~$\|\nabla u\|_2^2+\|u\|_2^2\leq C_n\|\chi\|_2^2$. Since~$\nabla u_i=(W_\alpha)_{ii}\nabla\psi_i$, this proves~\eqref{eq:first illus}.
\end{proof}

\subsection{\texorpdfstring{{Optimal decay by duality}}{Optimal decay by duality}}
We now improve Lemma~\ref{lemma:illustrative} to the decay rate~$(1+d(j,i_0))^{-s}$ of the entries of~$M$.

\begin{proposition}\label{proposition:optimal illus}
Under Assumptions~\ref{assumptions:mu}--\ref{assumptions:sM1}, let $1\leq i_0\leq m$, let~$\chi\in L^2(\nu)$ and let~$\psi\in H^1(\nu;\dR^m)$ be the weak solution of~\eqref{eq:illus}. Then, for every index~$j$,
\begin{equation}\label{eq:opt illus}
    \|\psi_j\|_2\leq C_n(1+d(j,i_0))^{-s}\|\chi\|_2
    \,.
\end{equation}
\end{proposition}

\begin{proof}
The unweighted energy estimate proves~\eqref{eq:opt illus} for~$j=i_0$. Fix~$j\neq i_0$ and let~$R=d(j,i_0)$. We compare~$\psi$ with the solution~$z$ of the same equation with a source at~$j$,
\begin{equation*}
    (\beta M+\mc L^\nu)\psi=\chi e_{i_0}
    \qquad\text{and}\qquad
    (\beta M+\mc L^\nu)z=\psi_je_j
    \,.
\end{equation*}
Let~$c_i=\mathds1_{d(i,i_0)>R/2}$, so that~$c_{i_0}=0$ and~$c_j=1$. We test the first equation against~$(c_iz_i)_i$ and the second against~$(c_i\psi_i)_i$. This gives
\begin{align*}
    \sum_ic_i\dE_\nu[\nabla z_i\cdot\nabla\psi_i]+\beta\sum_{i,k}c_i\dE_\nu[z_iM_{ik}\psi_k]&=c_{i_0}\dE_\nu[z_{i_0}\chi]=0\,,\\
    \sum_ic_i\dE_\nu[\nabla\psi_i\cdot\nabla z_i]+\beta\sum_{i,k}c_i\dE_\nu[\psi_iM_{ik}z_k]&=c_j\|\psi_j\|_2^2=\|\psi_j\|_2^2\,.
\end{align*}
The Dirichlet form terms are equal. We subtract the first identity from the second and use the symmetry of~$M$. Only the pairs with~$c_i\neq c_k$ remain, and~$|M_{ik}|\leq C_n(1+d(i,k))^{-s}$ gives
\begin{equation*}
    \|\psi_j\|_2^2=\beta\sum_{i,k}(c_i-c_k)\dE_\nu[\psi_iM_{ik}z_k]
    \leq C_n\sum_{i,k:c_i\neq c_k}(1+d(i,k))^{-s}\|\psi_i\|_2\|z_k\|_2
    \,.
\end{equation*}

We now bound the right-hand side. Since~$1/2<s/2<s-1/2$, Lemma~\ref{lemma:illustrative} applies to~$\psi$ and~$z$ with~$\alpha=s/2$. It gives
\begin{equation*}
    \sum_i(1+d(i,i_0))^s\|\psi_i\|_2^2\leq C_n\|\chi\|_2^2
    \quad\text{and}\quad
    \sum_k(1+d(k,j))^s\|z_k\|_2^2\leq C_n\|\psi_j\|_2^2
    \,.
\end{equation*}
As~$\sum_i(1+d(i,i_0))^{-s}\leq C(s)$, the Cauchy--Schwarz inequality gives~$\sum_i\|\psi_i\|_2\leq C_n\|\chi\|_2$ and~$\sum_k\|z_k\|_2\leq C_n\|\psi_j\|_2$. Hence the pairs with~$d(i,k)\geq R/6$ contribute at most~$C_nR^{-s}\|\chi\|_2\|\psi_j\|_2$. If~$c_i\neq c_k$ and~$d(i,k)<R/6$, then~$d(i,i_0)$ and~$d(k,i_0)$ lie between~$R/3$ and~$2R/3$, hence~$d(k,j)\geq R/3$. On these indices, the weighted bounds above give
\begin{equation*}
    \sum_{3d(i,i_0)\geq R}\|\psi_i\|_2^2\leq C_nR^{-s}\|\chi\|_2^2
    \quad\text{and}\quad
    \sum_{3d(k,j)\geq R}\|z_k\|_2^2\leq C_nR^{-s}\|\psi_j\|_2^2
    \,.
\end{equation*}
Since the kernel~$(1+d(i,k))^{-s}$ has bounded row sums, the Schur test shows that these pairs also contribute at most~$C_nR^{-s}\|\chi\|_2\|\psi_j\|_2$. Therefore~$\|\psi_j\|_2^2\leq C_nR^{-s}\|\chi\|_2\|\psi_j\|_2$. Since~$R\geq1$, this gives~\eqref{eq:opt illus}.
\end{proof}

\subsection{\texorpdfstring{{Decay under the mass constraint}}{Decay under the mass constraint}}
The equation~\eqref{eq:intro HS gaps} in gap coordinates contains the Lagrange multiplier~$\lambda$ of the constraint that the gaps sum to~$N$. We now add this multiplier to~\eqref{eq:illus}.

\begin{lemma}\label{lemma:control lambda}
Under Assumptions~\ref{assumptions:mu}--\ref{assumptions:sM1}, let $\chi\in L^2(\nu)$ and $1\leq i_0\leq m$, and let~$\psi\in H^1(\nu;\dR^m)$ and~$\lambda$ be the weak solution of the constrained equation
\begin{equation}\label{eq:billus 2}
    (\beta M+\mc L^\nu)\psi=\chi e_{i_0}+\lambda\mathbf1,
    \qquad \mathbf1\cdot\psi=0
    \,.
\end{equation}
Then, for every $1\leq j\leq m$,
\begin{equation}\label{eq:sec case}
    \|\lambda\|_2\leq C_n m^{-1}\|\chi\|_2,
    \qquad
    \|\psi_j\|_2\leq C_n\left((1+d(j,i_0))^{-s}+m^{-1}\right)\|\chi\|_2
    \,.
\end{equation}
\end{lemma}

\begin{proof}\
\paragraph{\emph{Step 1: the multiplier}}
Let $v$ solve $(\beta M+\mc L^\nu)v=\chi e_{i_0}$. By~\eqref{eq:opt illus} and~$s>1$, we have~$\sum_i\|v_i\|_2\leq C_n\|\chi\|_2$. Let~$w=\psi-v+m^{-1}(\mathbf1\cdot v)\mathbf1$. Since~$\mathbf1\cdot w=0$, testing the difference of the equations of~$\psi$ and~$v$ against~$w$ gives
\begin{equation*}
    \sum_i\|\nabla w_i\|_2^2+\beta\dE_\nu[w\cdot Mw]=\frac{\beta}{m}\dE_\nu[(\mathbf1\cdot v)(\mathbf1\cdot Mw)]
    \,.
\end{equation*}
The left-hand side is at least $\beta C_n^{-1}\sum_i\|w_i\|_2^2$. Since $\sup_k\sum_i|M_{ik}|\leq C_n$, the Cauchy--Schwarz inequality bounds the right-hand side by $C_nm^{-1}\|\mathbf1\cdot v\|_2\sum_i\|w_i\|_2$. Since the sum of the~$m$ norms~$\|w_i\|_2$ is at most~$\sqrt m$ times their Euclidean norm, it follows that
\begin{equation*}
    \sum_i\|w_i\|_2\leq C_n\|\mathbf1\cdot v\|_2\leq C_n\|\chi\|_2
    \,.
\end{equation*}
Hence~$\sum_i\|\psi_i\|_2\leq C_n\|\chi\|_2$. Summing~\eqref{eq:billus 2} over the components gives~$m\lambda=\beta\mathbf1\cdot M\psi-\chi$, which proves the first bound of~\eqref{eq:sec case}.

\paragraph{\emph{Step 2: decay by duality}}
Let~$z$ solve~\eqref{eq:illus} with source~$\psi_je_j$, so that
\begin{equation*}
    (\beta M+\mc L^\nu)\psi=\chi e_{i_0}+\lambda\mathbf1
    \qquad\text{and}\qquad
    (\beta M+\mc L^\nu)z=\psi_je_j
    \,.
\end{equation*}
We test the first equation against~$z$ and the second against~$\psi$. This gives
\begin{align*}
    \sum_i\dE_\nu[\nabla z_i\cdot\nabla\psi_i]+\beta\sum_{i,k}\dE_\nu[z_iM_{ik}\psi_k]&=\dE_\nu[\chi z_{i_0}]+\dE_\nu[\lambda\mathbf1\cdot z]\,,\\
    \sum_i\dE_\nu[\nabla\psi_i\cdot\nabla z_i]+\beta\sum_{i,k}\dE_\nu[\psi_iM_{ik}z_k]&=\|\psi_j\|_2^2\,.
\end{align*}
The Dirichlet form terms are equal, and so are the other two terms on the left, by the symmetry of~$M$. Hence
\begin{equation*}
    \|\psi_j\|_2^2=\dE_\nu[\chi z_{i_0}]+\dE_\nu[\lambda\mathbf1\cdot z]
    \,.
\end{equation*}
By~\eqref{eq:opt illus} for~$z$, we have~$\|z_{i_0}\|_2\leq C_n(1+d(j,i_0))^{-s}\|\psi_j\|_2$ and~$\|\mathbf1\cdot z\|_2\leq C_n\|\psi_j\|_2$. The Cauchy--Schwarz inequality and the bound on~$\|\lambda\|_2$ in~\eqref{eq:sec case} then prove the bound on~$\|\psi_j\|_2$.
\end{proof}

\subsection{The marginal on a window of gaps}\label{sub:marginal}
The Hessian of~$\Hc_N^\g$ is bounded below by the nearest-pair curvature, of order~$u^{-s-2}$ at a gap of length~$u$, which degenerates when a gap is large. We therefore add a convex penalty on the gaps larger than~$n^\ve/2$ in a window of~$n$ consecutive labels, and study the marginal law of these gaps. We show that the Hessian of the marginal energy is close to a matrix satisfying Assumptions~\ref{assumptions:sM1}, and deduce the decay of the marginal solution in Proposition~\ref{proposition:ap to ex}. In~\eqref{eq:def nu hypersing}, in the definitions of~$q$, $\Xi_i$, $P_I$, $Q$, $P_\partial$ and~$\delta_i$, and in Lemma~\ref{lemma:marginal hessian}, we also allow $0<s<1$.

Fix $\ve>0$ and $n\leq N/2$, and let $I=\{1,\ldots,n\}$. For the marginal construction, write the full gap configuration as~$(x,y)$, with $x=(x_i)_{i\in I}$ and $y=(y_a)_{a\notin I}$. Let~$\pi$ be the projection on~$x$, and let~$\theta$ be a smooth convex function on~$[0,+\infty)$ such that $0\leq\theta''\leq C$, with $\theta=0$ on~$[0,1/2]$ and $\theta''\geq1$ on~$[1,+\infty)$. For~$\iota\in\{0,1\}$, let
\begin{equation}\label{eq:def nu hypersing}
    \FF^\g(x)=\sum_{i\in I}{\bigl(}\theta(n^{-\ve}x_i)+\iota\theta(n^{-\ve}x_i^{-1})\bigr),\qquad
    \dd\dGiQ^\g\propto e^{-\beta\FF^\g\circ\pi}\dd\dGi^\g,
    \qquad \nu=\pi\#\dGiQ^\g
    \,.
\end{equation}
The choice $\iota=1$ is used in Section~\ref{section:change measure}. By Lemma~\ref{lemma:single gap tails}, the penalty changes the law of the gaps by at most~$Ce^{-cn^\delta}$ in total variation. Let~$\widetilde\Hc_n^\g$ be the marginal energy, so that $\dd\nu\propto e^{-\beta\widetilde\Hc_n^\g(x)}\dd x$, and let~$\mu_x$ be the conditional law of the exterior gaps~$y$ given~$x$. It is supported on $\sum_{a\notin I}y_a=N-\sum_{i\in I}x_i$. All derivatives of~$\Hc_N^\g$ use the ambient extension~\eqref{eq:ambient extension}. The nearest-pair terms in~\eqref{eq:expHg2} and the penalty give, on the gap simplex,
\begin{equation}\label{eq:penalized curvature}
    \nabla^2(\Hc_N^\g+\FF^\g)\geq c(s)n^{-(s+2)\ve}\operatorname{diag}(\mathbf1_I)
    \,.
\end{equation}
Indeed, the nearest-pair curvature gives this bound when $x_i\leq n^\ve$, and the penalty gives it when $x_i\geq n^\ve$. The inverse-gap penalty is convex, so this holds for both values of~$\iota$. For $i,j\in I$, let
\begin{gather*}
    q=\frac{\mathbf1_{I^c}}{N-n},\qquad
    {\Xi_i=\partial_i\Hc_N^\g-q\cdot\nabla\Hc_N^\g},\\
    P_I=\dE_{\mu_x}[(\nabla^2\Hc_N^\g)_{I,I}],\qquad
    Q_{ij}=\dE_{\mu_x}\bigl[q\cdot\nabla^2\Hc_N^\g q-(\nabla^2\Hc_N^\g q)_i-(\nabla^2\Hc_N^\g q)_j\bigr]
    \,.
\end{gather*}

\begin{lemma}[{Hessian of the marginal energy}]\label{lemma:marginal hessian}
Let $s\in(0,1)\cup(1,+\infty)$ and $n\leq N/2$. Then $\nabla^2\widetilde\Hc_n^\g\geq C_n^{-1}I$ and, for every $i,j\in I$,
\begin{equation}\label{eq:short marginal Hessian}
    \partial_{ij}\widetilde\Hc_n^\g=\partial_{ij}\FF^\g+(P_I)_{ij}+Q_{ij}
        -\beta\Cov_{\mu_x}[\Xi_i,\Xi_j]
    \,.
\end{equation}
\end{lemma}

\begin{proof}
We parametrize the exterior gaps by $y_a=z_a+(N-\sum_{k\in I}x_k)/(N-n)$ for $a\notin I$, where~$z$ ranges over the hyperplane $\sum_{a\notin I}z_a=0$. This change of variables has a constant Jacobian. Up to an additive constant, it gives
\begin{equation*}
    {\widetilde\Hc_n^\g(x)=\FF^\g(x)-\frac1\beta\log\int e^{-\beta\Hc_N^\g(x,y)}\dd z}
    \,.
\end{equation*}
We extend the integrand by zero when an exterior gap is nonpositive. When an exterior gap tends to zero, the integrand and its first two derivatives in~$x$ vanish faster than every power of this gap. We may therefore differentiate under the integral. When~$x_i$ increases, the exterior gaps move in the direction~$-q$. Hence the derivatives of~$\Hc_N^\g(x,y)$ in~$x_i$ and in~$x_i,x_j$ are
\begin{equation*}
    {\Xi_i
    \qquad\text{and}\qquad
    \partial_{ij}\Hc_N^\g-(\nabla^2\Hc_N^\g q)_i-(\nabla^2\Hc_N^\g q)_j+q\cdot\nabla^2\Hc_N^\g q}
    \,.
\end{equation*}
As in~\eqref{eq:exp gradient ext} and~\eqref{eq:exp hess ext}, differentiating the logarithm of the integral twice and taking expectations under~$\mu_x$ gives~\eqref{eq:short marginal Hessian}.

By~\eqref{eq:penalized curvature}, $(x,z)\mapsto\Hc_N^\g(x,y)+\FF^\g(x)-|x|^2/(2C_n)$ is convex. By Pr\'ekopa's theorem, $\widetilde\Hc_n^\g(x)-|x|^2/(2C_n)$ is convex, which proves $\nabla^2\widetilde\Hc_n^\g\geq C_n^{-1}I$.
\end{proof}

We now define the event on which~$\nabla^2\widetilde\Hc_n^\g$ is close to a matrix satisfying Assumptions~\ref{assumptions:sM1}. Let~$P_\partial$ be the part of~$P_I$ coming from the intervals in~\eqref{eq:expHg2} that meet~$I$ without being contained in~$I$. For $i\in I$, let $\delta_i=\min(i,n+1-i)$ be the distance from~$i$ to the complement of~$I$. Define the good event~$\mc B$ as the set of $x\in\pi(\mc M_N)$ such that, for every $i,j,k\in I$ and every interval $J\subset I$,
\begin{equation}\label{eq:short good coefficients}
\begin{gathered}
    n^{-\ve}\leq x_i\leq n^\ve
    \qquad\text{and}\qquad
    \Bigl|\sum_{\ell\in J}x_\ell-|J|\Bigr|\leq n^\ve|J|^{1/2},
    \\
    (P_I)_{ij}\leq C_n(1+d_N(i,j))^{-s}
    \qquad\text{and}\qquad
    (P_\partial)_{jk}\leq C_n\left(\max(j,k)^{-s}+(n+1-\min(j,k))^{-s}\right),
    \\
    |Q_{ij}|\leq C_nN^{-1}
    \qquad\text{and}\qquad
    \Var_{\mu_x}[\Xi_i]\leq C_n\left(\delta_i^{1-2s}+N^{-1}\right)
    \,.
\end{gathered}
\end{equation}
In~\eqref{eq:short good coefficients}, $C_n=C(\beta)n^{C\ve}$ is fixed, with~$C(\beta)$ and~$C$ large enough for Lemma~\ref{lemma:good event}.

When the two gaps of Theorem~\ref{theorem:hypersing} are far apart, we use the full system. For $n=N$, we take $I=\{1,\ldots,N\}$ in~\eqref{eq:def nu hypersing}, so that $\nu=\dGiQ^\g$, and let~$\mc L^\nu$ be its intrinsic generator on~$\mc M_N$. We let $\widetilde\Hc_N^\g=\Hc_N^\g+\FF^\g$, with the ambient extension~\eqref{eq:ambient extension}, and $P_I=\nabla^2\Hc_N^\g$. In~\eqref{eq:short good coefficients}, the intervals~$J$ are then the cyclic intervals of length at most~$N/2$, and the conditions involving~$Q$ and~$\Xi_i$ are omitted. We take $Q=0$ and $P_\partial=0$, and~\eqref{eq:short marginal Hessian} holds without the covariance term.

\begin{lemma}\label{lemma:good event}
Let $s>1$, let $n\leq N/2$ or $n=N$, and let~$\mc B$ be the good event~\eqref{eq:short good coefficients}. Then there exist $C,c,\delta>0$ such that
\begin{equation*}
    {\nu(\mc B^c)\leq Ce^{-cn^\delta}
    \qquad\text{and}\qquad
    \|\mathds1_{\mc B^c}\nabla^2\widetilde\Hc_n^\g\|_2\leq Ce^{-cn^\delta}}
    \,.
\end{equation*}
\end{lemma}

\begin{proof}
The event~$\mc B$ of~\eqref{eq:short good coefficients} is defined by $O(n^2)$ conditions. By a union bound, it suffices to show that each condition fails with probability at most~$Ce^{-cn^\delta}$. For the last two lines of~\eqref{eq:short good coefficients}, we bound moments. If a nonnegative random variable~$X$ satisfies $\|X\|_p\leq Cp^Ca$ for every $p\geq1$, then Markov's inequality with $p=n^{c\ve}$ gives $\nu(X>C_na)\leq Ce^{-cn^{c\ve}}$.

\paragraph{\emph{Step 1: the gap conditions}}
Lemma~\ref{lemma:local hyper} with $\eta=\ve$ bounds the probability that the second condition fails for an interval of at least~$n^\ve$ gaps, since $|J|^{1/2+\ve}\leq n^\ve|J|^{1/2}$. A shorter interval~$J$ with $|\sum_{i\in J}x_i-|J||>n^\ve|J|^{1/2}$ contains a gap larger than~$n^{\ve/2}$. Lemma~\ref{lemma:single gap tails} bounds the probability of this event and the probability that the first condition fails.

\paragraph{\emph{Step 2: moments of the coefficients}}
By Lemma~\ref{lemma:single gap tails} and Jensen's inequality, applied to the gap sums of~$J$ and of its complement, the coefficient of an interval~$J$ in~\eqref{eq:expHg2} satisfies, for every $p\geq1$, under~$\dGi^\g$ and under~$\dGiQ^\g$,
\begin{equation}\label{eq:coefficient moments}
    \bigl\|N^{-s-2}g_s''(L_J/N)\bigr\|_p\leq C(\beta)p^{C(s)}|J|^{-s-2}
    \,.
\end{equation}
The entry~$(P_I)_{ij}$ is the conditional expectation of the sum of these coefficients over the intervals that contain~$i$ and~$j$. Each of these intervals has length at least $1+d_N(i,j)$, and at most~$r$ of them have length~$r$. The conditional Jensen inequality and a summation over~$r$ therefore give the first bound below. An interval that contributes to~$(P_\partial)_{jk}$ contains in addition~$[1,\max(j,k)]$ and the gap preceding~$I$, or~$[\min(j,k),n]$ and the gap following~$I$. Its length is then at least~$\max(j,k)$ or $n+1-\min(j,k)$, and the same summation gives the second bound:
\begin{equation*}
    {\|(P_I)_{ij}\|_p\leq Cp^C(1+d_N(i,j))^{-s},
    \qquad
    \|(P_\partial)_{jk}\|_p\leq Cp^C\left(\max(j,k)^{-s}+(n+1-\min(j,k))^{-s}\right)}
    \,.
\end{equation*}
Each interval~$J$ enters the entries of~$\nabla^2\Hc_N^\g q$ and~$q\cdot\nabla^2\Hc_N^\g q$ with a factor at most $|J\cap I^c|/(N-n)\leq2|J|/N$. Since $s>1$, the sum over the lengths converges, and
\begin{equation*}
    {\|Q_{ij}\|_p+\|\dE_{\mu_x}[q\cdot\nabla^2\Hc_N^\g q]\|_p\leq Cp^CN^{-1}}
    \,.
\end{equation*}

\paragraph{\emph{Step 3: moments of the conditional variance}}
The nearest-pair terms give $\nabla_y^2\Hc_N^\g\geq2s(s+1)\operatorname{diag}(y_a^{-s-2})$ on the tangent space of the exterior simplex. The Brascamp--Lieb inequality for~$\mu_x$ therefore gives
\begin{equation*}
    {\Var_{\mu_x}[\partial_i\Hc_N^\g]\leq C\,\dE_{\mu_x}\Bigl[\sum_{a\notin I}y_a^{s+2}(\partial_{ia}\Hc_N^\g)^2\Bigr]}
    \,.
\end{equation*}
The derivative of~$q\cdot\nabla\Hc_N^\g$ in an exterior direction~$v$ is $q\cdot\nabla_y^2\Hc_N^\g v$, and $(q\cdot\nabla_y^2\Hc_N^\g v)^2\leq(q\cdot\nabla_y^2\Hc_N^\g q)(v\cdot\nabla_y^2\Hc_N^\g v)$. The Brascamp--Lieb inequality then gives $\Var_{\mu_x}[q\cdot\nabla\Hc_N^\g]\leq\beta^{-1}\dE_{\mu_x}[q\cdot\nabla^2\Hc_N^\g q]$, which Step~2 bounds. H\"older's inequality and Lemma~\ref{lemma:single gap tails} give $\|y_a^{s+2}(\partial_{ia}\Hc_N^\g)^2\|_p\leq Cp^C(1+d_N(i,a))^{-2s}$. Since $\sum_{a\notin I}(1+d_N(i,a))^{-2s}\leq C\delta_i^{1-2s}$, the conditional Jensen inequality gives
\begin{equation*}
    {\|\Var_{\mu_x}[\Xi_i]\|_p\leq Cp^C\left(\delta_i^{1-2s}+N^{-1}\right)}
    \,.
\end{equation*}

\paragraph{\emph{Step 4: the Hessian outside the good event}}
The covariance term in~\eqref{eq:short marginal Hessian} is bounded by $\Var_{\mu_x}[\Xi_i]+\Var_{\mu_x}[\Xi_j]$. Steps~2 and~3 and Lemma~\ref{lemma:single gap tails} for~$\nabla^2\FF^\g$ give $\|\nabla^2\widetilde\Hc_n^\g\|_4\leq Cn^C$. By the Cauchy--Schwarz inequality,
\begin{equation*}
    {\|\mathds1_{\mc B^c}\nabla^2\widetilde\Hc_n^\g\|_2\leq\|\nabla^2\widetilde\Hc_n^\g\|_4\,\nu(\mc B^c)^{1/4}}
    \,,
\end{equation*}
which proves the second bound after decreasing~$c$.

For $n=N$, Steps~1, 2 and~4 apply without conditional expectations, with cyclic intervals and $P_I=\nabla^2\Hc_N^\g$.
\end{proof}

\subsection{Decay for the marginal}
We compare~$\nabla^2\widetilde\Hc_n^\g$ with the matrix built from the intervals contained in~$I$, and deduce the decay of the marginal solution. By the end of Appendix~\ref{section:existence c}, the measures~\eqref{eq:def nu hypersing}, for $n\leq N/2$ and for $n=N$, and their conditional laws given~$\mathbf1_I\cdot x$ satisfy Assumptions~\ref{assumptions:mu}. On~$\mc B$, the variance bound in~\eqref{eq:short good coefficients} and the Cauchy--Schwarz inequality give
\begin{equation}\label{eq:short score bounds}
    |\Cov_{\mu_x}[\Xi_i,\Xi_j]|\leq C_n\left(\delta_i^{-(s-1/2)}+N^{-1/2}\right)\left(\delta_j^{-(s-1/2)}+N^{-1/2}\right)
    \,.
\end{equation}
With~$\mc B$ the good event~\eqref{eq:short good coefficients}, we set
\begin{equation}\label{eq:short frozen matrices}
    M_{\mc B}=\nabla^2\widetilde\Hc_n^\g\mathds{1}_{\mc B}+I\mathds{1}_{\mc B^c},
    \qquad M_0=(\nabla^2\FF^\g+{P_I-P_\partial})\mathds{1}_{\mc B}
        +I\mathds{1}_{\mc B^c}
    \,.
\end{equation}
We check that~$M_0$ satisfies Assumptions~\ref{assumptions:sM1}, and that $C_n^{-1}I\leq M_{\mc B}\leq C_nI$. On~$\mc B$, the Hessian of~$\FF^\g$ is diagonal, nonnegative and bounded by~$C_n$, and the entries of~$P_I-P_\partial$ are bounded by those of~$P_I$. By~\eqref{eq:expHg2}, the nearest-pair terms in~$P_I-P_\partial$ contribute a diagonal matrix with entries at least $c(s)x_i^{-s-2}$, and $x_i\leq n^\ve$ on~$\mc B$. The other intervals and the penalty contribute nonnegative matrices, so $M_0\geq C_n^{-1}I$. Hence~$M_0$ satisfies Assumptions~\ref{assumptions:sM1}. Since the covariance in~\eqref{eq:short marginal Hessian} is positive semidefinite, we have $\nabla^2\widetilde\Hc_n^\g\leq\nabla^2\FF^\g+P_I+Q$, and the row sums of~$P_I$ and~$Q$ give $M_{\mc B}\leq C_nI$. Lemma~\ref{lemma:marginal hessian} gives $M_{\mc B}\geq C_n^{-1}I$. Proposition~\ref{proposition:ap to ex} uses the next lemma with $\widetilde\nu=\nu$, and Lemma~\ref{lemma:length averaging} uses it with $\widetilde\nu=\nu^L$.

\begin{lemma}[{Comparison with the interior matrix}]\label{lemma:interior comparison}
Let $s>1$ and $n\leq N/2$, and let~$M_0$ and~$M_{\mc B}$ be given by~\eqref{eq:short frozen matrices}, with~$\mc B$ the good event~\eqref{eq:short good coefficients}. Let~$\widetilde\nu$ be~$\nu$, or the conditional law $\nu^L=\nu(\cdot\mid\mathbf1_I\cdot x=L)$ for some $0<L<N$, with its scalar Dirichlet form and generator~$\mc L^{\widetilde\nu}$. Let $i_0\in I$ with $n/4\leq i_0\leq3n/4$, let $\chi\in L^2(\widetilde\nu)$, and let~$v$ and~$u$ solve
\begin{equation*}
    (\beta M_0+\mc L^{\widetilde\nu})v=\chi e_{i_0}
    \qquad\text{and}\qquad
    (\beta M_{\mc B}+\mc L^{\widetilde\nu})u=\chi e_{i_0}
    \,.
\end{equation*}
Then
\begin{equation}\label{eq:short marginal frozen comparison}
    \|u-v\|_2+\|\nabla(u-v)\|_2
    \leq C_n(n^{-(s-1/2)}+N^{-1/2})\|\chi\|_2
\end{equation}
and
\begin{equation}\label{eq:short marginal sums}
    \sum_i\|u_i\|_2+\|\nabla(\mathbf1_I\cdot u)\|_2
    \leq C_n\|\chi\|_2
    \,.
\end{equation}
\end{lemma}

\begin{proof}
Proposition~\ref{proposition:optimal illus} gives $\|v_k\|_2\leq C_n(1+d(k,i_0))^{-s}\|\chi\|_2$, hence $\sum_k\|v_k\|_2\leq C_n\|\chi\|_2$. We first bound~$(M_{\mc B}-M_0)v$. The two matrices agree outside the good event~$\mc B$ of~\eqref{eq:short good coefficients}, and on~$\mc B$, by~\eqref{eq:short marginal Hessian},
\begin{equation*}
    M_{\mc B}-M_0=P_\partial+Q-\beta\Cov_{\mu_x}[\Xi,\Xi]
    \,,
\end{equation*}
where $\Cov_{\mu_x}[\Xi,\Xi]$ is the matrix $(\Cov_{\mu_x}[\Xi_i,\Xi_j])_{i,j\in I}$.
On~$\mc B$, the bound on~$P_\partial$ in~\eqref{eq:short good coefficients} gives bounded column sums for~$P_\partial$, and $(P_\partial)_{jk}\leq C_nn^{-s}$ if $\delta_k\geq n/8$. If $\delta_k<n/8$, then $|k-i_0|\geq n/8$, so that $\|v_k\|_2\leq C_nn^{-s}\|\chi\|_2$. Splitting the indices~$k$ into these two sets gives
\begin{equation*}
    \|\mathds{1}_{\mc B}P_\partial v\|_2\leq C_nn^{-(s-1/2)}\|\chi\|_2,
    \qquad \sum_k\delta_k^{-(s-1/2)}\|v_k\|_2\leq C_nn^{-(s-1/2)}\|\chi\|_2
    \,.
\end{equation*}
By~\eqref{eq:short score bounds} and $\sum_i\delta_i^{1-2s}\leq C(s)$, the covariance term applied to~$v$ has norm at most
\begin{equation*}
    C_n\left(1+\sqrt{n/N}\right)\sum_k\left(\delta_k^{-(s-1/2)}+N^{-1/2}\right)\|v_k\|_2
    \,.
\end{equation*}
Since $|Q_{ij}|\leq C_n/N$ on~$\mc B$, the norm of~$Qv$ is at most $C_n\sqrt nN^{-1}\sum_k\|v_k\|_2$. As $n\leq N$, we obtain
\begin{equation*}
    \|(M_{\mc B}-M_0)v\|_2\leq C_n(n^{-(s-1/2)}+N^{-1/2})\|\chi\|_2
    \,.
\end{equation*}
The difference $u-v$ solves $(\beta M_{\mc B}+\mc L^{\widetilde\nu})(u-v)=-\beta(M_{\mc B}-M_0)v$, and the energy estimate with $M_{\mc B}\geq C_n^{-1}I$ proves~\eqref{eq:short marginal frozen comparison}. Lemma~\ref{lemma:illustrative} with $1/2<\alpha<s-1/2$ and the Cauchy--Schwarz inequality give $\sum_i\|v_i\|_2+\|\nabla(\mathbf1_I\cdot v)\|_2\leq C_n\|\chi\|_2$. Since $\sqrt n(n^{-(s-1/2)}+N^{-1/2})\leq2$, the Cauchy--Schwarz inequality in the components transfers this bound to~$u$, which proves~\eqref{eq:short marginal sums}.
\end{proof}

We now deduce the decay for the marginal from Lemma~\ref{lemma:interior comparison}.

\begin{proposition}[{Decay for the marginal}]\label{proposition:ap to ex}
Let $s>1$, $\beta>0$ and $\ve>0$. Let~$\nu$ be given by~\eqref{eq:def nu hypersing}, or by its variant for $n=N$, with generator~$\mc L^\nu$ and marginal energy~$\widetilde\Hc_n^\g$. Let $\chi\in L^\infty(\nu)$.
\begin{enumerate}
\item Let $n\leq N/2$, let $i_0\in I$ with $n/4\leq i_0\leq3n/4$, and let~$\psi$ solve $(\beta {\nabla^2\widetilde\Hc_n^\g}+\mc L^\nu)\psi=\chi e_{i_0}$. Then there exist $c,\delta>0$ such that, for every~$j\in I$,
\begin{equation}\label{eq:main illus e}
    \|\psi_j\|_2\leq C_n\left((1+d(j,i_0))^{-s}
        +n^{-(s-1/2)}+N^{-1/2}\right)
        \left(\|\chi\|_2+e^{-cn^\delta}\|\chi\|_\infty\right)
    \,.
\end{equation}
\item Let $n=N$, let $1\leq i_0\leq N$, and let~$\psi$ and~$\lambda$ solve~\eqref{eq:billus 2} with $M={\nabla^2\widetilde\Hc_N^\g}$. Then there exist $c,\delta>0$ such that, for every $1\leq j\leq N$,
\begin{equation}\label{eq:main illus e2}
    \|\psi_j\|_2\leq C_N\left((1+d_N(j,i_0))^{-s}+N^{-1}\right)
        \left(\|\chi\|_2+e^{-cN^\delta}\|\chi\|_\infty\right)
    \,.
\end{equation}
\end{enumerate}
\end{proposition}

\begin{proof}
We first treat $n\leq N/2$. Let~$\mc B$ be the good event~\eqref{eq:short good coefficients}, and let~$u$ solve $(\beta M_{\mc B}+\mc L^\nu)u=\chi e_{i_0}$, with~$M_{\mc B}$ given by~\eqref{eq:short frozen matrices}. The maximum principle of Proposition~\ref{proposition:maximum principle} gives $\|u\|_{L^\infty(\nu)}\leq C_n\|\chi\|_\infty$. The difference $\psi-u$ solves
\begin{equation*}
    (\beta {\nabla^2\widetilde\Hc_n^\g}+\mc L^\nu)(\psi-u)=-\beta\mathds1_{\mc B^c}({\nabla^2\widetilde\Hc_n^\g}-I)u
    \,.
\end{equation*}
By Lemma~\ref{lemma:good event}, the right-hand side has norm at most $Ce^{-cn^\delta}\|\chi\|_\infty$ in~$L^2(\nu)$. Testing against~$\psi-u$ and using ${\nabla^2\widetilde\Hc_n^\g}\geq C_n^{-1}I$ give $\|\psi-u\|_2\leq Ce^{-cn^\delta}\|\chi\|_\infty$, after decreasing~$c$. Let~$v$ be as in Lemma~\ref{lemma:interior comparison}, with $\widetilde\nu=\nu$. By~\eqref{eq:short marginal frozen comparison} and Proposition~\ref{proposition:optimal illus} applied to~$v$,
\begin{equation*}
    \|u_j\|_2\leq\|v_j\|_2+\|u-v\|_2
    \leq C_n\left((1+d(j,i_0))^{-s}+n^{-(s-1/2)}+N^{-1/2}\right)\|\chi\|_2
    \,.
\end{equation*}
Since the first factor in~\eqref{eq:main illus e} is at least~$n^{-(s-1/2)}$, decreasing~$c$ absorbs the factor~$n^{s-1/2}$ in the exponential remainder.

For $n=N$, let~$M_{\mc B}$ be given by~\eqref{eq:short frozen matrices}, where~$\mc B$ is the good event~\eqref{eq:short good coefficients} for $n=N$. The coefficient bounds in~\eqref{eq:short good coefficients} give the entry bounds of Assumptions~\ref{assumptions:sM1} for~$M_{\mc B}$. By~\eqref{eq:penalized curvature} with $n=N$, we have $\nabla^2(\Hc_N^\g+\FF^\g)\geq c(s)N^{-(s+2)\ve}I_N$. Hence~$M_{\mc B}$ satisfies Assumptions~\ref{assumptions:sM1}. Let~$u$ and its multiplier solve~\eqref{eq:billus 2} with $M=M_{\mc B}$. Lemma~\ref{lemma:control lambda} bounds~$u$, and the maximum principle of Proposition~\ref{proposition:maximum principle}, applied in an orthonormal basis of~$\mathbf1^\perp$, gives $\|u\|_{L^\infty(\nu)}\leq C_N\|\chi\|_\infty$. We subtract the equations of~$\psi$ and~$u$ and test on the tangent space, where both multipliers disappear. As above, this gives $\|\psi-u\|_2\leq Ce^{-cN^\delta}\|\chi\|_\infty$, which proves~\eqref{eq:main illus e2}.
\end{proof}

\subsection{Decay of gap correlations}\label{sub:hyper decay}
We now conclude the proof of Theorem~\ref{theorem:hypersing}. When the gaps~$i$ and~$j$ are at distance larger than a small power of~$N$, we take $n=N$ in Proposition~\ref{proposition:ap to ex}. Otherwise, we take~$n$ to be a power of~$d_N(i,j)$.

\begin{proof}[Proof of Theorem~\ref{theorem:hypersing}]
Let $d=d_N(i,j)$, and fix $p>s/(s-1/2)$ and $\delta_0>0$ such that $p\delta_0<1$ and $2s\delta_0<1$. By approximation, we may assume that~$\xi$ and~$\chi$ are smooth. If $d\leq N^{\delta_0}$, let~$n$ be an integer with $n\asymp(1+d)^p$. Since $p\delta_0<1$, we have $n\leq N/2$ for~$N$ large. By invariance of the gap law under cyclic shifts of the labels, we may assume that~$i$ and~$j$ lie in the middle half of the window $I=\{1,\ldots,n\}$. If $d>N^{\delta_0}$, let $n=N$. In both cases, let~$\nu$ be given by~\eqref{eq:def nu hypersing} with $\iota=0$.

In gap coordinates, the covariance in~\eqref{eq:decay hyper} is $\Cov_{\dGi^\g}[\xi(x_i),\chi(x_j)]$. We first replace~$\dGi^\g$ by~$\nu$. By Lemma~\ref{lemma:single gap tails}, the penalty changes the law of the gaps by at most~$Ce^{-cn^\delta}$ in total variation. The functions $\xi-\xi(1)$ and $\chi-\chi(1)$ grow at most linearly, with slopes~$\|\xi'\|_\infty$ and~$\|\chi'\|_\infty$, and the gaps have uniformly bounded moments. Hence
\begin{equation}\label{eq:tilt transfer}
    \Cov_{\dGi^\g}[\xi(x_i),\chi(x_j)]
    =\Cov_\nu[\xi(x_i),\chi(x_j)]+O\left(\|\xi'\|_\infty\|\chi'\|_\infty e^{-cn^\delta}\right)
    \,.
\end{equation}
The same argument compares the second moments of~$\xi'$ and~$\chi'$ under~$\dGi^\g$ and~$\nu$. By the Helffer--Sj\"ostrand representation of Proposition~\ref{proposition: HS gaps c}, which the end of Appendix~\ref{section:existence c} extends to~$\nu$, and by polarization,
\begin{equation*}
    \Cov_\nu[\xi(x_i),\chi(x_j)]=\dE_\nu[\xi'(x_i)\psi_i]
    \,,
\end{equation*}
where~$\psi$ solves $(\beta\nabla^2\widetilde\Hc_n^\g+\mc L^\nu)\psi=\chi'(x_j)e_j$ if $n\leq N/2$, and~\eqref{eq:billus 2} with $M=\nabla^2\widetilde\Hc_N^\g$ and source~$\chi'(x_j)e_j$ if $n=N$. By the Cauchy--Schwarz inequality, it remains to bound~$\|\psi_i\|_2$.

If $n\leq N/2$, Proposition~\ref{proposition:ap to ex} gives
\begin{equation*}
    \|\psi_i\|_2\leq C_n\left((1+d)^{-s}+n^{-(s-1/2)}+N^{-1/2}\right)\left(\|\chi'(x_j)\|_2+e^{-cn^\delta}\|\chi'\|_\infty\right)
    \,.
\end{equation*}
Since $p(s-1/2)>s$ and $2s\delta_0<1$, both $n^{-(s-1/2)}$ and~$N^{-1/2}$ are at most~$C(1+d)^{-s}$. Since ${C_n\leq C(1+d)^{Cp\ve}}$, replacing~$\ve$ by~$\ve/(Cp)$ in the penalty turns~$C_n$ into a factor~$(1+d)^{\ve}$. If $n=N$, Proposition~\ref{proposition:ap to ex} gives the same bound with~$C_N$ and~$N^{-1}$ in place of~$C_n$ and $n^{-(s-1/2)}+N^{-1/2}$. Since $N<d^{1/\delta_0}$, after decreasing~$\ve$, the factor~$C_N$ is absorbed into~$(1+d)^\ve$ and~$N^\ve$. In both cases $n\geq d$, so the remainders are at most~$e^{-cd^\delta}$. Inserting these bounds into~\eqref{eq:tilt transfer} proves~\eqref{eq:decay hyper}.
\end{proof}

\section{Decay of correlations for the long-range Riesz gas}\label{section:large scale}
We prove Theorem~\ref{theorem:decay gap correlations} on the decay of correlations in the long-range case $0<s<1$. The Hessian of the energy in gap coordinates concentrates around twice the matrix~\eqref{eq:def H matrix}, whose entries decay like~$(1+d)^{-s}$ at distance~$d$. Since this decay is not summable, the commutator estimate of Lemma~\ref{lemma:perturbation lemma} does not apply to this matrix. The Hessian, however, controls a fractional primitive of the solution, and not only the solution itself. This control comes from the long-range interaction, which also gives the near-optimal concentration estimates of~\cite[Section~4]{boursier2021optimal}, adapted from~\cite[Section~3]{bourgade2014universality}.

We follow the outline of Subsection~\ref{subsection:outline}. Subsection~\ref{sub:re} embeds the equation of a window of~$n$ gaps in a system on a cycle of~$\bn$ labels. On the good event, the matrix of this system keeps the penalty~$\FF^\g$ and the coefficients~\eqref{eq:expHg2} of the intervals contained in the window, and takes the coefficients of twice the matrix~\eqref{eq:def H matrix} for the other intervals. Lemmas~\ref{lemma:BDC}, \ref{lemma:boundary intervals}, \ref{lemma:compensation} and~\ref{lemma:cond variance} bound the resulting error operator. On the cycle, Lemmas~\ref{lemma:elliptic}, \ref{lemma:diff} and~\ref{lemma:global} are the three steps of the outline. Proposition~\ref{proposition:loc} turns these global bounds into the decay rates~$2-s$ and~$(3-s)/2$, and Proposition~\ref{prop:main2} returns to the window.

\subsection{Periodization}\label{sub:re}
As in Section~\ref{section:decay short}, we localize the system on a window of gaps. Let $\ve>0$, let $n=N$ or $n\leq N/2$, and let $I=\{1,\ldots,n\}$. Let~$\FF^\g$, $\dGiQ^\g$ and~$\nu$ be given by~\eqref{eq:def nu hypersing} with $\iota=0$. Let~$\widetilde\Hc_n^\g$ be the marginal energy of Section~\ref{section:decay short}, normalized so that $e^{-\beta\widetilde\Hc_n^\g}$ is the density of~$\nu$, and let $A_1^\nu=\beta\nabla^2\widetilde\Hc_n^\g+\mc L^\nu\otimes I_n$. For $i_0\in I$ and a bounded function~$\chi_n$ on~$\pi(\mc M_N)$, we study the solution $\psi\in H^1(\nu;\dR^I)$ of
\begin{equation}\label{eq:eqqq}
    A_1^{\nu}\psi=\chi_ne_{i_0}\quad\text{on }\pi(\mc{M}_N)
    \,.
\end{equation}

We embed~\eqref{eq:eqqq} in a periodic system of size $\bar n=\min\{N,\max\{2n,\lceil n^{1/2+1/s}\rceil\}\}$ if $n\leq N/2$, and set $\bar n=N$ if $n=N$. On the cycle, the reference matrix is circulant, so that the next subsections treat it and its inverses as Fourier multipliers. Since $\bn\leq2n^{1/2+1/s}$, the factors $\log(2\bn)$ in these estimates are at most~$C_n$. Let $\bar I=\{1,\ldots,\bar n\}$, and let $K_0=\lfloor n^\ve\rfloor^{\mathfrak m}\leq\bar n/4$, where~$\mathfrak m$ is the smallest integer with $\mathfrak m(1-s/2)>2$. For $k\geq2$ and $1\leq K\leq k/4$, let $\sR_k=2\mathbb G_{k,s}$ and $\sR_k^{[K]}=2\mathbb G_{k,s}^{[K]}$. The rows of the circulant matrix~$\mathbb G_{k,s}$ sum to a number~$\vartheta_k$, so that $\sR_k\mathbf1=2\vartheta_k\mathbf1$, and~\eqref{eq:inv sum} gives $\vartheta_k\geq c(s)k^{1-s}$. We decompose the truncated matrix~$\sR_{\bn}^{[K_0]}$ into blocks on $\dR^{n}\times \dR^{\bn-n}$,
\begin{equation}\label{eq:blockHN}
    \sR_{\bn}^{[K_0]}=\begin{pmatrix}
A & B\\
B^T & D
\end{pmatrix}
    \,.
\end{equation}
Let~$A_\partial$ be the part of~$A$ coming from the intervals in the representation~\eqref{eq:riesz quadratic form} of~$\sR_{\bn}^{[K_0]}$ that meet~$I$ without being contained in~$I$. The intervals not contained in~$I$ form a nonnegative quadratic form with blocks~$A_\partial$, $B$, $B^T$ and~$D$. Minimizing it over the exterior coordinates gives $A_\partial-BD^{-1}B^T\geq0$. When $n=N$, the exterior blocks are absent.

For $n\leq N/2$, let $\mu_x$, $q$, $\Xi_i$, $P_I$ and~$Q$ be as before Lemma~\ref{lemma:marginal hessian}. By Lemma~\ref{lemma:marginal hessian}, \eqref{eq:short marginal Hessian} holds. For $n=N$, we use the conventions stated after~\eqref{eq:short good coefficients}.

Let~$P_\partial$ and~$\delta_i$ be defined as before~\eqref{eq:short good coefficients}. The matrix $P_I-P_\partial$ comes from the intervals contained in~$I$, so its coefficients do not involve the exterior gaps. Define the good event~$\mc A$ as the set of $x\in\pi(\mc M_N)$ such that, for every $i,j\in I$ and every interval $J\subset I$,
\begin{equation}\label{eq:good5 gap}
    n^{-\ve}\leq x_i\leq n^\ve,
    \qquad
    \Bigl|\sum_{\ell\in J}x_\ell-|J|\Bigr|\leq n^\ve|J|^{s/2}
    \qquad\text{and}\qquad
    (P_I)_{ij}\leq C_n(1+d_N(i,j))^{-s}
    \,.
\end{equation}
In~\eqref{eq:good5 gap}, $C_n=C(\beta)n^{C\ve}$ is fixed, with~$C(\beta)$ and~$C$ large enough for Lemma~\ref{lemma:good event long}. For $n=N$, the intervals~$J$ are cyclic intervals of length at most~$N/2$, as for~\eqref{eq:short good coefficients}.

As in Section~\ref{section:decay short}, the complement of the good event has exponentially small probability.

\begin{lemma}\label{lemma:good event long}
Let $0<s<1$, let $n\leq N/2$ or $n=N$, let~$\nu$ be given by~\eqref{eq:def nu hypersing} with $\iota=0$, and let~$\mc A$ be the good event~\eqref{eq:good5 gap}. Then there exist $C,c,\delta>0$ such that $\nu(\mc A^c)\leq Ce^{-cn^\delta}$ and, for every $p\geq1$ and $i,j\in I$,
\begin{equation*}
    {\|(P_I)_{ij}\|_p\leq\frac{Cp^C}{(1+d_N(i,j))^s}}
    \,.
\end{equation*}
\end{lemma}

\begin{proof}
The function~$\theta$ in~\eqref{eq:def nu hypersing} is nonnegative, vanishes on~$[0,1/2]$ and grows at most quadratically. Jensen's inequality and Lemma~\ref{lemma:single gap tails} therefore bound the normalization $\dE_{\dGi^\g}[e^{-\beta\FF^\g}]$ below by a positive constant. Hence it suffices to prove both bounds under~$\dGi^\g$. The event~$\mc A$ is defined by $O(n^2)$ conditions. By a union bound, it suffices to show that each of them fails with probability at most~$Ce^{-cn^\delta}$.

\paragraph{\emph{Step 1: the gap conditions}}
By Lemma~\ref{lemma:single gap tails}, the first condition fails for a given~$i$ with probability at most~$Ce^{-cn^{s\ve}}$. For the second condition, let $\eta=\min(\ve,1/4)$. If $|J|\geq\lceil n^{\eta/2}\rceil$, then Lemma~\ref{theorem:almost optimal rigidity c}, applied with $\min(\eta/8,\ve_0/2)$ in place of~$\ve$, bounds the probability that the condition fails by~$Ce^{-cn^\delta}$. For shorter intervals~$J$, the lower bound in the condition holds trivially. The upper bound follows from the same rigidity estimate for a block of $\lceil n^{\eta/2}\rceil$ gaps that contains~$J$, since the gap sum of this block is at most $2\lceil n^{\eta/2}\rceil\leq n^\eta$ for large~$n$.

\paragraph{\emph{Step 2: the coefficient condition}}
An interval that contains~$i$ and~$j$ has length at least $1+d_N(i,j)$, and at most~$r$ such intervals have length~$r$. The conditional Jensen inequality and~\eqref{eq:coefficient moments} therefore give
\begin{equation*}
    \|(P_I)_{ij}\|_p\leq Cp^C\sum_{r\geq1+d_N(i,j)}r^{-s-1}\leq\frac{Cp^C}{(1+d_N(i,j))^s}
    \,.
\end{equation*}
Markov's inequality with $p=n^{c\ve}$, for $c>0$ sufficiently small, then shows that the last condition of~\eqref{eq:good5 gap} fails with probability at most~$Ce^{-cn^{c\ve}}$.
\end{proof}

On~$\mc A$, we keep the penalty~$\FF^\g$ and the physical coefficients of the intervals contained in~$I$, which are their coefficients in~\eqref{eq:expHg2}. Elsewhere, we use the reference coefficients, which are those of~$\sR_{\bn}^{[K_0]}$ in~\eqref{eq:riesz quadratic form}. Let
\begin{equation}\label{eq:A0 e}
    A_0=(\nabla^2\FF^\g+{P_I-P_\partial})\mathds{1}_{\mc A}
    +(A-A_\partial)\mathds{1}_{\mc A^c}+{A_\partial}
    \,.
\end{equation}
Let
\begin{equation}\label{eq:schur term}
    \sS^\nu=\beta B(D+\beta^{-1}\mc L^\nu\otimes I_{\bar n-n})^{-1}B^T
\end{equation}
be the Schur term. It appears when we eliminate the exterior coordinates of the periodic system in Lemma~\ref{lemma:ident}. We split
\begin{equation}\label{eq:dec A1mu}
   A_1^{\nu}=\bar{A}_1^{\nu}+\tM,
   \qquad
   {\bar A_1^\nu=\beta A_0+\mc L^\nu\otimes I_n-\sS^\nu}
   \,,
\end{equation}
so that
\begin{equation}\label{eq:tH5}
    \tM=\beta(P_\partial-A_\partial)\mathds{1}_{\mc A}
    +\beta(\nabla^2\FF^\g+P_I-A)\mathds{1}_{\mc A^c}
    +\beta Q-\beta^2\Cov_{\mu_x}[\Xi,\Xi]
    +{\sS^\nu}
    \,.
\end{equation}

We work with general measures $\nu$ on convex domains in order to make our argument flexible enough to be applied to the interpolating measure~\eqref{eq:def mun(t)} of Section~\ref{section:change measure}.

\begin{assumptions}\label{assumptions:mu2}
Let $n\geq2$ be an integer, let~$\Omega$ be a bounded open convex subset of a finite-dimensional Euclidean affine space, let $\mc A\subset\Omega$ be measurable, and let~$\nu$ be a probability measure on~$\Omega$. Assume that:
\begin{enumerate}
    \item The measure~$\nu$ has a density $e^{-H^\g}$ with $H^\g\in\mathcal C^2(\Omega)$.
    \item The Dirichlet form $\dE_\nu[|\nabla f|^2]$, with domain~$H^1(\nu)$, is closed and conservative, with nonnegative generator~$\mc L^\nu$.
    \item There exist $C,c,\delta>0$ such that $\nu(\mc A^c)\leq Ce^{-cn^{\delta}}$.
\end{enumerate}
\end{assumptions}

For the marginal~$\nu$ of~\eqref{eq:eqqq}, we take $\Omega=\pi(\mc M_N)$, $H^\g=\beta\widetilde\Hc_n^\g$ and the good event~$\mc A$ of~\eqref{eq:good5 gap}. The end of Appendix~\ref{section:existence c} then verifies the second assumption, and Lemma~\ref{lemma:good event long} gives the third one. For a general~$\nu$, the Schur term~$\sS^\nu$ is defined by the same formula. Every vector equation below holds componentwise in~$H^{-1}(\nu)$, with the Neumann condition on the relative boundary of~$\Omega$.

For the bootstrap argument of Subsection~\ref{sub:loc2}, we also work with a slightly more general system, in which~$A_0$ is a measurable map from~$\Omega$ to the symmetric $n\times n$ matrices. The operator~$\bar A_1^\nu$ is then defined by the same formula, and we impose the following assumptions on the matrix
\begin{equation}\label{eq:defsM'}
    \sM=\begin{pmatrix}
A_0 & B\\
B^T & D
\end{pmatrix}
    \,.
\end{equation}

\begin{assumptions}\label{assumptions:sM}
Let $\ve>0$, let $n\leq N/2$ or $n=N$, let~$\bn$ and~$K_0$ be as before~\eqref{eq:blockHN}, and let $A_0$ be a measurable map from~$\Omega$ to the symmetric $n\times n$ matrices. Let~$A$, $B$ and~$D$ be the blocks of~\eqref{eq:blockHN}, and let~$\sM$ be given by~\eqref{eq:defsM'}. Assume that there exist $c,C,\kappa>0$ such that, almost everywhere on~$\Omega$:
\begin{enumerate}
    \item The matrix~$\sM$ satisfies
    \begin{equation}\label{eq:hypsM}
    \sM\geq cn^{-\kappa\ve}I_{\bar{n}}
    \,.
    \end{equation}
    \item There exist coefficients~$a_J$, indexed by the cyclic intervals~$J$ of length at most~$\bn/2$, such that $u\cdot\sM u=\sum_Ja_J(\sum_{i\in J}u_i)^2$ for every $u\in\dR^{\bn}$ and
    \begin{equation}\label{eq:alphaik}
    0\leq a_J\leq\frac{Cn^{\kappa\ve}}{|J|^{s+2}},
    \qquad
    a_J\geq\frac{cn^{-\kappa\ve}}{|J|^{s+2}}
    \quad\text{if }|J|\geq n^{\kappa\ve}
    \,.
    \end{equation}
    \item For every $i,j\in I$,
    \begin{equation}\label{eq:tildeA}
        |(A_0)_{ij}-A_{ij}|\leq \frac{Cn^{\kappa\ve}}{(1+d_{\bn}(i,j))^{1+s/2}}
        \,.
    \end{equation}
\end{enumerate}
\end{assumptions}

In the results of this section, the constants~$C(\beta)$ and~$c(\beta)$ may also depend on~$\ve$, on the convex function~$\theta$, on~$\mathfrak m$ and on the constants of Assumptions~\ref{assumptions:mu2} and~\ref{assumptions:sM}. The exponents~$\kappa$ in the conclusions may be larger than the exponent~$\kappa$ of Assumptions~\ref{assumptions:sM}. They depend only on~$s$, on the exponent~$\kappa$ of Assumptions~\ref{assumptions:sM} and on~$\mathfrak m$, and the exponents~$\delta$ also depend on~$\ve$ and on the tail exponent in Assumptions~\ref{assumptions:mu2}.

We now check that the construction~\eqref{eq:A0 e} satisfies these assumptions, and that eliminating the exterior components of the periodic system gives the equation with the Schur term.

\begin{lemma}\label{lemma:ident}
\begin{enumerate}
\item Let $\mathfrak m$ be the smallest integer with $\mathfrak m(1-s/2)>2$, let $K_0=\lfloor n^\ve\rfloor^{\mathfrak m}\leq\bar n/4$ and let $A_0$ be given by~\eqref{eq:A0 e}. For~$\ve$ sufficiently small, the matrix $\sM$ in~\eqref{eq:defsM'} satisfies Assumptions~\ref{assumptions:sM}.
\item Let~$\nu$ satisfy Assumptions~\ref{assumptions:mu2}, let~$\sM$ satisfy Assumptions~\ref{assumptions:sM}, and let $i_0\in I$ and $\chi_n\in L^2(\nu)$. Then there is a unique solution $\Psi\in H^1(\nu;\dR^{\bar I})$ of
\begin{equation}\label{eq:equ 54}
    \beta\sM\Psi+\mc{L}^{\nu}\Psi=\chi_n e_{i_0}
    \,,
\end{equation}
and a unique solution $\psi^{(1)}\in H^1(\nu;\dR^I)$ of
\begin{equation}\label{eq:equ psi1}
    (\beta A_0+\mc L^\nu\otimes I_n-\sS^\nu)\psi^{(1)}=\chi_ne_{i_0}
    \,.
\end{equation}
Moreover, $\Psi_j=\psi_j^{(1)}$ for every $j\in I$.
\end{enumerate}
\end{lemma}

\begin{proof}
We treat sufficiently large~$n$, since bounded sizes are absorbed in the constants. We assume that $\ve\leq1/[4(\mathfrak m+1)]$. On~$\mc A$, an interval of length $r<K_0$ then has gap sum at most $K_0n^\ve\leq n^{(\mathfrak m+1)\ve}\leq n^{1/4}$, which is small compared with~$n$. For $r\geq K_0$, since $\mathfrak m(1-s/2)>2$, we have $n^\ve r^{s/2-1}\leq C(\mathfrak m,s)n^{\ve(1-\mathfrak m(1-s/2))}\leq Cn^{-\ve}$, so the second condition in~\eqref{eq:good5 gap} keeps the gap sum between $r/2$ and~$2r$. For an interval $J$ of length $r$ contained in~$I$, its physical coefficient in~\eqref{eq:expHg2} is comparable to $r^{-s-2}$ up to powers of~$n^\ve$ on~$\mc A$. The reference coefficients have the same bounds for $r\geq K_0$. The penalty~$\FF^\g$ contributes nonnegative singleton coefficients bounded by a power of~$n^\ve$. Thus~\eqref{eq:alphaik} holds and, by comparison of the coefficients,
\begin{equation*}
    \sM\geq n^{-C\ve}{\sR_{\bn}^{[K_0]}}
    \geq c(s)n^{-C\ve}K_0^{-1-s}I_{\bar n}
    \,.
\end{equation*}
To prove~\eqref{eq:tildeA}, we compare the physical and reference coefficients of each interval~$J\subset I$ of length~$r$. For $r\geq K_0$, rigidity, Taylor's formula and the comparison~\eqref{eq:unit gap comparison} at unit gaps show that they differ by at most
\begin{equation*}
    {Cn^{C\ve}r^{-3-s/2}
    +Cr^{-s-4}+C\mathds{1}_{\bar n<N}\bar n^{-s-2}}
    \,.
\end{equation*}
At most $r$ intervals of length~$r$ contain two fixed indices, so the sum of the first two terms over these intervals is at most ${Cn^{C\ve}(1+d_{\bar n}(i,j))^{-1-s/2}}$. The last term vanishes if $\bar n=N$. Otherwise, its sum is at most $Cn^2\bar n^{-s-2}\leq Cn^{-1-s/2}$. The intervals of length less than~$K_0$ contribute a matrix of bandwidth~$K_0$ and entries bounded by a power of~$n^\ve$ on~$\mc A$. Increasing the power of~$n^\ve$ proves~\eqref{eq:tildeA} on~$\mc A$. On~$\mc A^c$, we have $A_0=A$.

We now prove the second item. Since~$\mc L^\nu$ is nonnegative, \eqref{eq:hypsM} gives, for every $v\in H^1(\nu;\dR^{\bar I})$,
\begin{equation*}
    \dE_{\nu}[v\cdot (\beta\sM+\mc{L}^\nu\otimes I_{\bar{n}}) v] \geq {\beta c} n^{-\kappa\ve}\dE_{\nu}[|v|^2]
    \,.
\end{equation*}
The Lax--Milgram theorem therefore gives the existence and uniqueness of~$\Psi$. For~$\psi^{(1)}$, it remains to prove the coercivity of~$\bar A_1^\nu$ on~$H^1(\nu;\dR^I)$. Minimizing the quadratic form of~$\sM$ over the exterior coordinates gives
\begin{equation}\label{eq:511}
    A_0-BD^{-1}{B^T}\geq {c}n^{-\kappa\ve}I_n
    \,.
\end{equation}
Let $w\in H^1(\nu;\dR^I)$. Since $\mc L^\nu$ is nonnegative, we have $0\leq(D+\beta^{-1}\mc L^\nu\otimes I_{\bn-n})^{-1}\leq D^{-1}$, and therefore
\begin{equation*}
    0\leq\dE_\nu[w\cdot\sS^\nu w]\leq\beta\dE_\nu[w\cdot BD^{-1}B^Tw]
    \,.
\end{equation*}
Inserting~\eqref{eq:511}, we find
\begin{equation*}
    \dE_{\nu}[w\cdot {\bar{A}_1^\nu} w]\geq {\beta c} n^{-\kappa\ve}\dE_{\nu}[|w|^2]
    \,.
\end{equation*}
It remains to prove that $\Psi_j=\psi_j^{(1)}$ for $j\in I$. Write~$\Psi=(u,v)$, where~$u$ collects the components in~$I$. The equations of~\eqref{eq:equ 54} on the exterior components give~$v=-(D+\beta^{-1}\mc L^\nu\otimes I_{\bn-n})^{-1}B^Tu$. Substituting~$v$ into the equations on~$I$ shows that~$u$ solves~\eqref{eq:equ psi1}. Hence~$u=\psi^{(1)}$ by uniqueness.
\end{proof}

To compare~$\bar A_1^\nu$ with~$A_1^\nu$, let~$\psi$ be the solution of~\eqref{eq:eqqq}, and let~$\psi^{(1)}$ be the solution of~\eqref{eq:equ psi1}. By~\eqref{eq:dec A1mu}, the difference $\psi^{(2)}=\psi^{(1)}-\psi$ solves
\begin{equation*}
    A_1^{\nu}\psi^{(2)}=\tM\psi^{(1)}
    \,.
\end{equation*}
To test the equation of~$\psi^{(2)}$ against~$\psi^{(2)}$, we first check that it is a valid test function. Lemma~\ref{lemma:collision} gives~\eqref{eq:collision criterion} for~$\beta\widetilde{\Hc}_n^\g$ in place of~$H$. For smooth~$f$ supported away from collisions, integration by parts against~$\nu$ and~\eqref{eq:collision criterion} with~$a=1/4$ give
\begin{equation*}
    \dE_\nu[f^2\Delta\widetilde{\Hc}_n^\g]\leq C(N)\dE_\nu[f^2+|\nabla f|^2]
    \,.
\end{equation*}
This bound extends to~$H^1(\nu)$ by density. The matrix~$\nabla^2\widetilde{\Hc}_n^\g$ is nonnegative, and its trace is~$\Delta\widetilde{\Hc}_n^\g$ for~${n<N}$ and at most~$2\Delta\widetilde{\Hc}_n^\g$ for~$n=N$, by~\eqref{eq:expHg2}. Thus ${\dE_\nu[u\cdot\nabla^2\widetilde{\Hc}_n^\g u]}$ is finite for every~$u\in H^1(\nu;\dR^I)$, and~$\psi^{(2)}$ is a valid test function. Testing the equation of~$\psi^{(2)}$ against~$\psi^{(2)}$ and integrating by parts with respect to~$\nu$ give
\begin{equation}\label{eq:ifp}
   C_n^{-1}\dE_{\nu}[|\psi^{(2)}|^2]\leq \beta \dE_{\nu}[\psi^{(2)}\cdot \nabla^2\widetilde{\Hc}_n^\g \psi^{(2)}]\leq \dE_{\nu}[\psi^{(2)}\cdot \tM\psi^{(1)}]
   \,.
\end{equation}
Lemma~\ref{lemma:cond variance} below bounds the entries of~$\tM$, and the proof of Proposition~\ref{prop:main2} bounds the right-hand side.

We estimate the Schur term~$\sS^\nu$ through the covariance of a Gaussian vector with covariance matrix~$D^{-1}$.

\begin{lemma}\label{lemma:BDC}
Let $s\in(0,1)$, let $\nu$ satisfy Assumptions~\ref{assumptions:mu2}, let~$\sS^\nu$ be the Schur term~\eqref{eq:schur term}, and let $\delta_i=\min(i,n+1-i)$ for $i\in I$. Then there exist $C(\beta)>0$ and $\kappa>0$ such that, for every $\phi_1,\phi_2\in L^2(\nu)$ and $i,j,\ell\in I$,
\begin{gather}
    {|\dE_\nu[\phi_1 e_i\cdot\sS^\nu(\phi_2 e_j)]|
    \leq \frac{C(\beta)n^{\kappa\ve}}{\delta_i^{s/2}\delta_j^{s/2}}
    \|\phi_1\|_2\|\phi_2\|_2}
    \,,\label{eq:BDC}\\
    |\dE_\nu[\phi_1 e_i\cdot\sS^\nu(\phi_2(e_j-e_\ell))]|
    \leq \frac{C(\beta)n^{\kappa\ve}|j-\ell|}
    {\delta_i^{s/2}\min(\delta_j,\delta_\ell)^{1+s/2}}
    \|\phi_1\|_2\|\phi_2\|_2
    \,.\label{eq:BDC b}
\end{gather}
\end{lemma}

\begin{proof}
For $n=N$, the Schur term vanishes. Assume $n\leq N/2$. Let~$D$ be the exterior block of~\eqref{eq:blockHN}, and let $Z\sim\mathcal N(0,D^{-1})$ be a Gaussian vector indexed by $\bar I\setminus I$, extended by zero on~$I$. All variances below refer to this Gaussian vector.

\paragraph{\emph{Step 1: sums of the Gaussian vector over intervals}}
The truncation at the cutoff~$K_0$ of~\eqref{eq:blockHN} removes from~\eqref{eq:riesz quadratic form} short intervals with total row sum at most $C(s)\sum_{r<K_0}r^{-s}\leq C(s)K_0^{1-s}$. This gives the first bound below. The second is the coercivity bound~\eqref{eq:truncated coercivity}, and the third follows from the first two.
\begin{equation*}
    0\leq {\sR_{\bn}}-\sR_{\bn}^{[K_0]}\leq C(s)K_0^{1-s}I_{\bn},
    \qquad \sR_{\bn}^{[K_0]}\geq c(s)K_0^{-1-s}I_{\bn},
    \qquad \sR_{\bn}\leq C(s)K_0^2\sR_{\bn}^{[K_0]}
    \,.
\end{equation*}
Let $J\subset\bar I\setminus I$ be an interval of $r$ indices, and let $\mathbf 1_J\in\dR^{\bn}$ be its indicator. The variational formula for the inverse gives
\begin{equation*}
    \Var\Bigl[\sum_{a\in J}Z_a\Bigr]
    =\mathbf1_J\cdot D^{-1}\mathbf1_J
    \leq\mathbf1_J\cdot(\sR_{\bn}^{[K_0]})^{-1}\mathbf1_J
    \leq C(s)K_0^2\mathbf1_J\cdot\sR_{\bn}^{-1}\mathbf1_J
    \,.
\end{equation*}
By Lemma~\ref{lemma:inverse}, the row sum of~${\sR_{\bn}^{-1}}$ is at most $C(s)\bn^{s-1}$. At most $C\min(r,d)$ pairs in $J\times J^c$ have cyclic distance~$d$. Subtracting the sum over these pairs from $r$ times the row sum yields
\begin{equation*}
    \mathbf1_J\cdot\sR_{\bn}^{-1}\mathbf1_J
    \leq C(s)\Bigl(r\bn^{s-1}
        +\sum_{d=1}^{\lfloor \bn/2\rfloor}\min(r,d)d^{s-2}\Bigr)
    \leq C(s)r^s
    \,.
\end{equation*}
Consequently,
\begin{equation}\label{eq:VarS}
    \Var\Bigl[\sum_{a\in J}Z_a\Bigr]\leq C(s)K_0^2|J|^s
    \,.
\end{equation}

\paragraph{\emph{Step 2: the exterior forces and their increments}}
For $1\leq r\leq\bn/2$, let $c_r=(2-\mathds{1}_{2r=\bn})g_{\bn,s}''(r)$ be the coefficient of the intervals of length~$r$ in the representation~\eqref{eq:riesz quadratic form} of~$\sR_{\bn}$. The bounds~\eqref{eq:gss bounds} give~$0\leq c_r\leq C(s)r^{-s-2}$. Since~$Z$ vanishes on~$I$, this representation, restricted to $r\geq K_0$, gives, for $i\in I$,
\begin{equation*}
    (BZ)_i=\sum_{r=K_0}^{\lfloor\bn/2\rfloor}c_r\sum_{J}\,\sum_{k\in J}Z_k
    \,,
\end{equation*}
where~$J$ runs over the~$r$ cyclic intervals of length~$r$ that contain~$i$. For each such~$J$, the set~$J\setminus I$ has at most two interval components, so~\eqref{eq:VarS} and the triangle inequality give $\Var[\sum_{k\in J}Z_k]^{1/2}\leq C(s)K_0r^{s/2}$. If~$r<\delta_i$, then every interval of length~$r$ that contains~$i$ lies in~$I$, and its sum vanishes. When~$i$ increases by one, the interval~$[i-r+1,i]$ leaves this family and the interval~$[i+1,i+r]$ enters it. Minkowski's inequality therefore gives, for~$i\in I$ and for~$i<n$ respectively,
\begin{align*}
    \Var[(BZ)_i]^{1/2}
    &\leq C(s)K_0\sum_{r\geq\delta_i}r^{-1-s/2}
    \leq C(s)K_0\delta_i^{-s/2},\\
    \Var[(BZ)_{i+1}-(BZ)_i]^{1/2}
    &\leq C(s)K_0\sum_{r\geq\min(\delta_i,\delta_{i+1})}r^{-2-s/2}
    \leq C(s)K_0\min(\delta_i,\delta_{i+1})^{-1-s/2}
    \,.
\end{align*}
For~$j<\ell$, every~$i$ with~$j\leq i<\ell$ satisfies~$\min(\delta_i,\delta_{i+1})\geq\min(\delta_j,\delta_\ell)$. Summing the second bound over these~$i$ yields
\begin{equation*}
    \Var[(BZ)_j-(BZ)_\ell]^{1/2}
    \leq\frac{C(s)K_0|j-\ell|}{\min(\delta_j,\delta_\ell)^{1+s/2}}
    \,.
\end{equation*}

\paragraph{\emph{Step 3: passage to the Schur term}}
Since $\mc L^\nu$ is nonnegative, we have $0\leq(D+\beta^{-1}\mc L^\nu\otimes I_{\bn-n})^{-1}\leq D^{-1}$. The Cauchy--Schwarz inequality for this nonnegative operator gives, for~$u,v\in\dR^n$,
\begin{equation*}
    |\dE_\nu[\phi_1 u\cdot\sS^\nu(\phi_2 v)]|
    \leq\beta\|\phi_1\|_2\|\phi_2\|_2
        (u\cdot BD^{-1}B^Tu)^{1/2}(v\cdot BD^{-1}B^Tv)^{1/2}
    \,.
\end{equation*}
The quadratic form $u\cdot BD^{-1}B^Tu$ is the variance of~$u\cdot BZ$. We take $u=e_i$, and $v=e_j$ or $v=e_j-e_\ell$. Since~$K_0$ is a fixed power of~$\lfloor n^\ve\rfloor$, we have $K_0^2\leq n^{\kappa\ve}$ for a fixed~$\kappa$, and Step~2 proves~\eqref{eq:BDC} and~\eqref{eq:BDC b}.
\end{proof}

We next bound the entries of the boundary term $P_\partial-A_\partial$ of~\eqref{eq:tH5}, which comes from the intervals that meet~$I$ without being contained in~$I$.

\begin{lemma}\label{lemma:boundary intervals}
Let $0<s<1$ and $n\leq N/2$, and let $\delta_i=\min(i,n+1-i)$ for $i\in I$. Then there exist $C,\kappa>0$ such that, on the good event~$\mc A$ of~\eqref{eq:good5 gap}, for every $i,j,\ell\in I$,
\begin{gather}
    {|(P_\partial-A_\partial)_{ij}|}
    \leq \frac{Cn^{\kappa\ve}}{\delta_i^{s/2}\delta_j^{s/2}}
    \,,\label{eq:boundary entries}\\
    |(P_\partial-A_\partial)_{ij}-(P_\partial-A_\partial)_{i\ell}|
    \leq \frac{Cn^{\kappa\ve}|j-\ell|}
    {\delta_i^{s/2}\min(\delta_j,\delta_\ell)^{1+s/2}}
    \,.\label{eq:boundary increments}
\end{gather}
Moreover, for every $i\in I$ and $1\leq j<n$,
\begin{equation}\label{eq:boundary physical}
    (P_\partial)_{ii}\leq\frac{Cn^{\kappa\ve}}{\delta_i^{s}},
    \qquad
    (e_j-e_{j+1})\cdot P_\partial(e_j-e_{j+1})\leq\frac{Cn^{\kappa\ve}}{\delta_j^{1+s}}
    \,.
\end{equation}
\end{lemma}

\begin{proof}
A short cyclic interval not contained in~$I$ meets~$I$ in a prefix~$[1,k]$, in a suffix~$[k,n]$, or not at all, where the prefix~$[1,n]$ is~$I$ itself. Thus each of the matrices $V=P_\partial$ and $V=A_\partial$ has the form
\begin{equation*}
    {V=\sum_{k=1}^{n}\mathfrak a_k\mathbf1_{[1,k]}\mathbf1_{[1,k]}^T
    +\sum_{k=2}^{n}\mathfrak b_k\mathbf1_{[k,n]}\mathbf1_{[k,n]}^T}
    \,,
\end{equation*}
with nonnegative coefficients. For the reference matrix, counting the intervals of each length gives ${\mathfrak a_k\leq Ck^{-1-s}}$ and $\mathfrak b_{k+1}\leq C(n-k)^{-1-s}$ for $k<n$, and $\mathfrak a_n\leq Cn^{-s}$.

For~$V=P_\partial$, fix $j<n$ and let $h=\lfloor\min(j-1,n-j)/4\rfloor$. Moving the right endpoint of an interval from~$j$ to $k\in[j-h,j]$ removes at most $hn^\ve$ from its gap sum~$L_J$, whereas $N-L_J\geq(n-j)n^{-\ve}$. By the positive series
\begin{equation*}
    N^{-s-2}g_s''(L_J/N)=s(s+1)\sum_{\ell\in\mathbb Z}|L_J+\ell N|^{-s-2}
    \,,
\end{equation*}
the coefficient of each such interval changes by at most a factor~$Cn^{C\ve}$. Hence $\mathfrak a_k\geq C^{-1}n^{-C\ve}\mathfrak a_j$ for these~$k$, and the bound on~$P_I$ in~\eqref{eq:good5 gap} yields
\begin{equation*}
    (h+1)\mathfrak a_j\leq Cn^{C\ve}(P_I)_{1,j-h}
    \leq Cn^{C\ve}j^{-s}
    \,.
\end{equation*}
Consequently, ${\mathfrak a_j}\leq Cn^{C\ve}j^{-s}/\delta_j$. The reflected bound holds for~${\mathfrak b_k}$, and ${\mathfrak a_n}\leq(P_I)_{1n}\leq Cn^{C\ve}n^{-s}$. Summing gives ${V_{ii}}\leq Cn^{C\ve}\log(2n)\delta_i^{-s}$, which proves the first bound of~\eqref{eq:boundary physical}. Since ${V}\geq0$, the Cauchy--Schwarz inequality proves~\eqref{eq:boundary entries}.

For $i\leq j{<n}$, the column increment $V_{i,j+1}-V_{ij}$ equals $-\mathfrak a_j$, and for $i\geq j+1$ it equals~$\mathfrak b_{j+1}$. The bounds above give ${|V_{i,j+1}-V_{ij}|\leq Cn^{C\ve}\delta_i^{-s/2}\delta_j^{-1-s/2}}$. Telescoping proves~\eqref{eq:boundary increments}. Finally, for the coefficients of~$P_\partial$, we have ${(e_j-e_{j+1})\cdot P_\partial(e_j-e_{j+1})=\mathfrak a_j+\mathfrak b_{j+1}}$. Since $\delta_j\leq j$ and ${\delta_j\leq2\delta_{j+1}\leq2(n-j)}$, the bounds above prove the second bound of~\eqref{eq:boundary physical}.
\end{proof}

We now control the operator~$\tM$ of~\eqref{eq:tH5}. For $n\leq N/2$, its term~$Q$ involves the coefficients
\begin{equation}\label{eq:ijH}
    b_i=\dE_{\mu_x}[(\nabla^2\Hc_N^\g q)_i],
    \qquad
    \bar b=\dE_{\mu_x}[q\cdot\nabla^2\Hc_N^\g q]
    \,,
\end{equation}
with~$\mu_x$ and~$q$ as before Lemma~\ref{lemma:marginal hessian}. Since~$q$ is the uniform probability vector on~$I^c$, the coefficient~$\bar b$ is the average over the exterior labels of the quantity that defines~$b_i$. The definition of~$Q$ gives $Q_{ij}=\bar b-b_i-b_j$. We first bound the moments of these coefficients and the conditional variance of~$\Xi_i$.

\begin{lemma}\label{lemma:compensation}
Let $0<s<1$ and $n\leq N/2$, let $\nu=\pi\#\dGiQ^\g$, and let $\delta_i=\min(i,n+1-i)$ for $i\in I$. Then there exists $C(\beta)>0$ such that, for every $p\geq1$ and $i\in I$, with $i<n$ in the second bound,
\begin{equation}\label{eq:compensation moments}
    \|b_i\|_p+\|\bar b\|_p
    \leq C(\beta)p^{C(s)}N^{-s},\qquad
    \|b_{i+1}-b_i\|_p
    \leq C(\beta)p^{C(s)}N^{-1}\delta_i^{-s}
    \,.
\end{equation}
Moreover, for every $i\in I$,
\begin{equation}\label{eq:condQ}
    \Var_{\mu_x}[\Xi_i]
    \leq\beta^{-1}\bigl((P_\partial)_{ii}+\bar b\bigr)
    \,.
\end{equation}
\end{lemma}

\begin{proof}
{Among the} intervals of length~$r$, there are $r$ containing a fixed index, $N$ in total, and at most two contributing to $b_{i+1}-b_i$. In the latter case, an interval meeting~$I^c$ has $r\geq\delta_i$. Since $\sum_{a\in J}q_a\leq2r/N$, the conditional Jensen inequality and~\eqref{eq:coefficient moments} give, for $p\geq1$ and $i<n$,
\begin{gather*}
    \|b_i\|_p\leq C(\beta)p^{C(s)}\sum_{r=1}^{\lfloor N/2\rfloor}r\cdot\frac{2r}{N}\cdot r^{-s-2},
    \qquad
    \|\bar b\|_p\leq C(\beta)p^{C(s)}\sum_{r=1}^{\lfloor N/2\rfloor}N\Bigl(\frac{2r}{N}\Bigr)^2r^{-s-2},\\
    \|b_{i+1}-b_i\|_p\leq C(\beta)p^{C(s)}\sum_{r\geq\delta_i}2\cdot\frac{2r}{N}\cdot r^{-s-2}
    \,.
\end{gather*}
Since $s<1$, the first two sums are at most $C(s)N^{-s}$. Since $s>0$, the last sum is at most $C(s)N^{-1}\delta_i^{-s}$. This proves~\eqref{eq:compensation moments}.

We delete from $\nabla^2\Hc_N^\g$ all intervals contained in~$I$. This gives a positive semidefinite matrix~${M_\partial}$, and changes neither the exterior Hessian nor the exterior gradient of~$\Xi_i$. The conditional Brascamp--Lieb inequality on the exterior tangent space bounds $\Var_{\mu_x}[\Xi_i]$ by~$\beta^{-1}$ times the inverse of the exterior block of~$M_\partial$, tested on the exterior gradient of~$\Xi_i$. This gradient is the exterior part of $M_\partial(e_i-q)$, so the Schur complement inequality for $M_\partial\geq0$ gives
\begin{equation*}
    \Var_{\mu_x}[\Xi_i]
    \leq\beta^{-1}\dE_{\mu_x}[(e_i-q)\cdot {M_\partial}(e_i-q)]
    =\beta^{-1}(({P_\partial})_{ii}-2b_i+\bar b)
    \,.
\end{equation*}
The coefficients of the intervals in~\eqref{eq:expHg2} are positive, so $b_i\geq0$. This proves~\eqref{eq:condQ}.
\end{proof}

We bound the coefficients of~$\tM$ by truncation. Let~$\nu$ be a probability measure, let~$X$ be a random variable, let $\kappa,a>0$, and let~$\mc G$ be an event on which $|X|\leq n^{\kappa\ve}a$. For all bounded~$\phi_1$ and~$\phi_2$, the Cauchy--Schwarz inequality gives
\begin{equation}\label{eq:cond truncation}
\begin{aligned}
    |\dE_\nu[\phi_1 X\phi_2]|
    &\leq n^{\kappa\ve}a\|\phi_1\|_2\|\phi_2\|_2
    +\dE_\nu[|X|\mathds1_{\mc G^c}]\|\phi_1\|_\infty\|\phi_2\|_\infty\\
    &\leq n^{\kappa\ve}a\|\phi_1\|_2\|\phi_2\|_2
    +\|X\|_2\,\nu(\mc G^c)^{1/2}\|\phi_1\|_\infty\|\phi_2\|_\infty
    \,.
\end{aligned}
\end{equation}

\begin{lemma}\label{lemma:cond variance}
Let $0<s<1$, let $n\leq N/2$ or $n=N$, let $\nu=\pi\#\dGiQ^\g$, let $\tM$ be given by~\eqref{eq:tH5}, and let $\delta_i=\min(i,n+1-i)$ for $i\in I$. Then there exist $\kappa,\delta>0$ and ${C(\beta),c(\beta)>0}$ such that, up to an additive error $C(\beta)e^{-c(\beta)n^\delta}\|\phi_1\|_{L^\infty(\nu)}\|\phi_2\|_{L^\infty(\nu)}$, for all $\phi_1,\phi_2\in L^\infty(\nu)$ and $i,j,\ell\in I$,
\begin{gather}
    {|\dE_\nu[\phi_1 e_i\cdot\tM(\phi_2 e_j)]|
    \leq \frac{C(\beta)n^{\kappa\ve}}{\delta_i^{s/2}\delta_j^{s/2}}
    \|\phi_1\|_2\|\phi_2\|_2}
    \,,\label{eq:Mg1}\\
    |\dE_\nu[\phi_1 e_i\cdot\tM(\phi_2(e_j-e_\ell))]|
    \leq \frac{C(\beta)n^{\kappa\ve}|j-\ell|}
    {\delta_i^{s/2}\min(\delta_j,\delta_\ell)^{1+s/2}}
    \|\phi_1\|_2\|\phi_2\|_2
    \,.\label{eq:Mg2}
\end{gather}
\end{lemma}
\begin{proof}
We bound each coefficient below with~\eqref{eq:cond truncation}. If $\|X\|_p\leq C(\beta)p^{C(s)}a$ for every $p\geq1$, then Markov's inequality with $p=\lfloor n^{c(s)\ve}\rfloor$, for a sufficiently small $c(s)>0$, shows that $\mc G=\{|X|\leq n^{\kappa\ve}a\}$ satisfies
\begin{equation*}
    \nu(\mc G^c)
    \leq\Bigl(\frac{C(\beta)p^{C(s)}}{n^{\kappa\ve}}\Bigr)^p
    \leq C(\beta)e^{-c(\beta)n^{c(s)\ve}}
    \,.
\end{equation*}
In each step below, $\|X\|_2$ is bounded by a power of~$n$, which the exponential factor absorbs.

\paragraph{\emph{Step 1: the boundary and Schur terms}}
The boundary term in~\eqref{eq:tH5} satisfies~\eqref{eq:Mg1} and~\eqref{eq:Mg2} by Lemma~\ref{lemma:boundary intervals}{, and the Schur term~$\sS^\nu$ satisfies them by Lemma~\ref{lemma:BDC}}. By Lemma~\ref{lemma:single gap tails} and the moment bound of Lemma~\ref{lemma:good event long}, the entries of $\nabla^2\FF^\g+P_I-A$ have bounded $L^2(\nu)$ norms, where~$A$ is the block of~\eqref{eq:blockHN}. The Cauchy--Schwarz inequality and the bound on~$\nu(\mc A^c)$ in that lemma, for the good event~$\mc A$ of~\eqref{eq:good5 gap}, therefore give, for every $i,j\in I$,
\begin{equation*}
    |\dE_\nu[\phi_1\phi_2(\nabla^2\FF^\g+P_I-A)_{ij}\mathds1_{\mc A^c}]|
    \leq Ce^{-cn^\delta}\|\phi_1\|_\infty\|\phi_2\|_\infty
    \,.
\end{equation*}
For $n=N$, this completes the proof, since $Q=0$ and there is no covariance term. We assume from now on that $n\leq N/2$.

\paragraph{\emph{Step 2: the term~$Q$}}
We use the moment bounds~\eqref{eq:compensation moments} of Lemma~\ref{lemma:compensation}. Since $N^{-s}\leq\delta_i^{-s/2}\delta_j^{-s/2}$, the first bound gives $\|Q_{ij}\|_p\leq C(\beta)p^{C(s)}\delta_i^{-s/2}\delta_j^{-s/2}$. Telescoping the second bound gives
\begin{equation*}
    \|Q_{ij}-Q_{i\ell}\|_p
    \leq C(\beta)p^{C(s)}N^{-1}|j-\ell|\min(\delta_j,\delta_\ell)^{-s}
    \leq\frac{C(\beta)p^{C(s)}|j-\ell|}
    {\delta_i^{s/2}{\min(\delta_j,\delta_\ell)}^{1+s/2}}
    \,.
\end{equation*}
The second inequality holds because $\delta_i^{s/2}\min(\delta_j,\delta_\ell)^{1-s/2}\leq n\leq N$. The truncation bound~\eqref{eq:cond truncation}, applied to the entries and increments of~$Q$ with these scales, gives~\eqref{eq:Mg1} and~\eqref{eq:Mg2} for~$\beta Q$.

\paragraph{\emph{Step 3: the covariance entries}}
We use the bound~\eqref{eq:condQ} of Lemma~\ref{lemma:compensation}. On~$\mc A$, \eqref{eq:boundary physical} bounds $({P_\partial})_{ii}$ by $Cn^{\kappa\ve}\delta_i^{-s}$. We truncate~$\bar b$ at $n^{\kappa\ve}N^{-s}$ and apply the conditional Cauchy--Schwarz inequality to obtain the coefficient bound of~\eqref{eq:Mg1} for the covariance term, since ${N^{-s}\leq\delta_i^{-s}}$. By~\eqref{eq:condQ}, the moments of~$P_I$ and~\eqref{eq:compensation moments}, its entries have $L^2(\nu)$ norms bounded uniformly in~$N$. Lemma~\ref{lemma:good event long} bounds~$\nu(\mc A^c)$, and the Markov bound above, with $X=\bar b$ and $a=N^{-s}$, bounds the probability that $\bar b>n^{\kappa\ve}N^{-s}$. Hence~\eqref{eq:cond truncation}, with $\mc G=\mc A\cap\{\bar b\leq n^{\kappa\ve}N^{-s}\}$, gives the exponential remainder.

\paragraph{\emph{Step 4: the covariance increments}}
To prove~\eqref{eq:Mg2} for the covariance term, we estimate the conditional variance ${\Var_{\mu_x}[\Xi_j-\Xi_{j+1}]}$. Fix $x\in\mc A$ and $j<n$. In microscopic particle coordinates, the fixed gaps pin $t_1=0$ and ${t_{k+1}=x_1+\ldots+x_k}$ for $1\leq k\leq n$. Let~$z_a$, for $n+2\leq a\leq N$, be the exterior positions. This change of coordinates has a constant Jacobian. For $1\leq k\leq n+1$ and $n+2\leq a\leq N$, let
\begin{equation*}
    {w_{ka}=2N^{-s-2}g_s''\Bigl(\frac{z_a-t_k}{N}\Bigr),
    \qquad w_a=\sum_{k=1}^{n+1}w_{ka}}
    \,.
\end{equation*}
The conditional energy has Hessian at least $\operatorname{diag}({w_a})$, since the remaining part is the nonnegative graph Laplacian of the exterior pair interactions. The direction $e_j-e_{j+1}$ moves only the particle~${t_{j+1}}$, so $\partial_{z_a}(\Xi_j-\Xi_{j+1})=-{w_{j+1,a}}$.

The gap bounds on~$\mc A$ put each exterior particle at distance at least $\delta_j n^{-\ve}/2$ from~${t_{j+1}}$ along either arc. There are at least $c\delta_j$ pinned particles~$t_k$ with $|t_k-t_{j+1}|\leq\delta_j n^\ve$. Comparing the positive periodic series for~$g_s''$ gives $w_{ka}\geq C_n^{-1}w_{j+1,a}$ for these particles, hence $w_a\geq C_n^{-1}\delta_jw_{j+1,a}$. By the second bound of~\eqref{eq:boundary physical},
\begin{equation*}
    \dE_{\mu_x}\sum_{a=n+2}^N w_{j+1,a}
    =(e_j-e_{j+1})\cdot {P_\partial}(e_j-e_{j+1})
    \leq C_n\delta_j^{-1-s}
    \,.
\end{equation*}
The conditional Brascamp--Lieb inequality therefore yields
\begin{equation*}
    \Var_{\mu_x}[\Xi_j-\Xi_{j+1}]
    \leq\beta^{-1}\dE_{\mu_x}\sum_{a=n+2}^N\frac{{w_{j+1,a}^2}}{{w_a}}
    \leq C_n\delta_j^{-2-s}
    \,.
\end{equation*}
Telescoping in the space of centered $L^2(\mu_x)$ functions gives, for $j,\ell\in I$,
\begin{equation}\label{eq:condQ diff}
    \Var_{\mu_x}[\Xi_j-\Xi_\ell]^{1/2}
    \leq\frac{C_n|j-\ell|}{\min(\delta_j,\delta_\ell)^{1+s/2}}
    \,.
\end{equation}
We combine~\eqref{eq:condQ diff} with~\eqref{eq:condQ} and the conditional Cauchy--Schwarz inequality. On $\mc A\cap\{\bar b\leq n^{\kappa\ve}N^{-s}\}$, this gives the coefficient bound in~\eqref{eq:Mg2}, since $N^{-s}\leq\delta_i^{-s}$. The moment bounds of Step~3 and~\eqref{eq:cond truncation}, with the same event~$\mc G$, give the exponential remainder.
\end{proof}

\subsection{Elliptic regularity estimate}
The aim is now to prove a decay estimate on the solution~$\Psi$ of~\eqref{eq:equ 54}, which equals~$\psi^{(1)}$ on~$I$ by Lemma~\ref{lemma:ident}. We first treat the system
\begin{equation}\label{eq:per2n}
    \beta\sM\psi+\mc{L}^{\nu}\psi=\chi_n e_{i_0}+\lambda\mathbf1,
    \qquad
    \mathbf1\cdot\psi=0
    \,,
\end{equation}
with a multiplier $\lambda\in L^2(\nu)$, and Lemma~\ref{lemma:global bis} removes the constraint. Until the end of Subsection~\ref{sub:loc2}, $\psi$ denotes a solution on the cycle~$\bar I$, which Lemma~\ref{lemma:ident} calls~$\Psi$ when there is no constraint. With the exponents $3/2-s$ and $1-s/2$, the two quantities in~\eqref{eq:alpha gamma} become~$\|W_{3/2-s}\psi\|_2^2$ and~$\|W_{1-s/2}\nabla\psi\|_2^2$. Lemma~\ref{lemma:elliptic} bounds the $\ell^{1/(1-s/2)}$ norm of the fractional primitive~$\mathbb G_{\bn,s/2}\psi$ of order~$1-s/2$ (Remark~\ref{remark:discrete}). Its proof bounds the primitive of order~$3/2-s$ in~$\ell^2$ in terms of~${\|\chi_n\|_2+\|W_{1/2}\psi\|_2}$, and the primitive~$\mathbb G_{\bn,s}\psi$ of order~$1-s$ in~$\ell^1$ in terms of~$\|W_{3/2-s}\psi\|_2$. A discrete Gagliardo--Nirenberg inequality interpolates between these two bounds. From now on, for an index $i_0\in\bar I$ and $\alpha>0$, we use the weight of Section~\ref{section:decay short} on the periodic system,
\begin{equation}\label{eq:Lalpha5}
    W_\alpha=\diag\bigl(r_i^\alpha\bigr)_{1\leq i\leq\bn},
    \qquad r_i=1+d_{\bn}(i,i_0)
    \,.
\end{equation}
We also let $E=\sM-\sR_{\bn}$ be the difference with the full reference matrix. Replacing the truncated reference in~\eqref{eq:tildeA} by the full matrix adds a matrix of bandwidth~$K_0$ with absolute row sums at most $CK_0^{1-s}$. Thus the absolute row and column sums of~$E$ are bounded by~$C_n$.

\begin{lemma}\label{lemma:elliptic}
Let $s\in (0,1)$, and let~$\nu$ and~$\sM$ satisfy Assumptions~\ref{assumptions:mu2} and~\ref{assumptions:sM}. Let $\chi_n\in H^{1}(\nu)$ and $i_0\in\bar I=\{1,\ldots,\bn\}$, with~$\bn$ as before~\eqref{eq:blockHN}, and let $\psi\in H^1(\nu;\dR^{\bar I})$ and $\lambda\in L^2(\nu)$ solve~\eqref{eq:per2n}. Let~$W_\alpha$ and~$\mathbb G_{\bn,s/2}$ be given by~\eqref{eq:Lalpha5} and~\eqref{eq:def H matrix}. Then there exist $C(\beta)>0$ and $\kappa>0$ such that
\begin{equation}\label{eq:fou}
    \|\mathbb G_{\bn,s/2}\psi\|_{L^2(\nu;\ell^{1/(1-s/2)})}
    \leq C(\beta)n^{\kappa\ve}
    \bigl(\|\chi_n\|_2+\|W_{1/2}\psi\|_2\bigr)^s
    \|W_{3/2-s}\psi\|_2^{1-s}
    \,.
\end{equation}
\end{lemma}

\begin{proof}
We first bound the primitive of order $3/2-s$ of~$\psi$. Let $\theta_k=2\pi k/\bn$, and let~$\Lambda^{s-3/2}$ be the operator with Fourier multiplier~$\rho(\theta_k)^{s-3/2}$ at nonzero frequencies~$k$ and multiplier zero at~$k=0$. Let $\widehat v_k=\bn^{-1/2}\mathcal F_d(v)(\theta_k)$ be the unitary Fourier transform of a vector~$v$ on the cycle. By~\eqref{eq:Fdgs}, the eigenvalue of~$\sR_{\bn}$ at the frequency $k\neq0$ is $2S(\theta_k)$. The transform of~\eqref{eq:per2n} therefore gives, for $k\neq0$,
\begin{equation*}
    (2\beta S(\theta_k)+\mc L^\nu)\widehat\psi_k
    =\bn^{-1/2}\chi_ne^{ii_0\theta_k}
    -\beta\widehat{E\psi}_k
    \,.
\end{equation*}
The bound on the absolute column sums of~$E$ and the weighted Cauchy--Schwarz inequality give
\begin{equation*}
    \|E\psi\|_{\ell^1}\leq C_n\|\psi\|_{\ell^1}
    \quad\text{and}\quad
    \|\psi\|_{\ell^1}
    \leq\Bigl(\sum_i\frac{1}{1+d_{\bn}(i,i_0)}\Bigr)^{1/2}|W_{1/2}\psi|
    \leq C\sqrt{\log(2\bn)}\,|W_{1/2}\psi|
    \,.
\end{equation*}
Testing with $\widehat\psi_k$ bounds the inverse of $2\beta S(\theta_k)+\mc L^\nu$ on $L^2(\nu)$ by $(2\beta S(\theta_k))^{-1}$. Since $S(\theta_k)\asymp\rho(\theta_k)^{s-1}$ by~\eqref{eq:lowerS}, Plancherel's identity and $\bn^{-1}\sum_{k\neq0}\rho(\theta_k)^{-1}\leq C\log(2\bn)$ give
\begin{equation}\label{eq:s1}
    \|\Lambda^{s-3/2}\psi\|_2
    \leq C_n\bigl(\|\chi_n\|_2+\|W_{1/2}\psi\|_2\bigr)
    \,.
\end{equation}

By~\eqref{eq:kernel intervals} and the bounds~\eqref{eq:gss bounds} on~$g_{\bn,s}''$, we have
\begin{equation*}
    |g_{\bn,s}(k)|\leq C(1+d_{\bn}(k,0))^{-s}
    \quad\text{and}\quad
    |g_{\bn,s}(k+1)-g_{\bn,s}(k)|\leq C(1+d_{\bn}(k,0))^{-1-s}
    \,.
\end{equation*}
We use the first bound when~$d_{\bn}(i,j)\leq2d_{\bn}(j,i_0)$, and otherwise the second one along the short arc from~$j$ to~$i_0$. This gives
\begin{equation*}
    \sum_i|g_{\bn,s}(i-j)-g_{\bn,s}(i-i_0)|
    \leq C\sum_{r=0}^{3d_{\bn}(j,i_0)}(1+r)^{-s}
    +Cd_{\bn}(j,i_0)\sum_{r\geq d_{\bn}(j,i_0)}(1+r)^{-1-s}
    \,.
\end{equation*}
Since $0<s<1$, the right-hand side is at most $C(1+d_{\bn}(j,i_0))^{1-s}$. Since $\sum_j\psi_j=0$, another weighted Cauchy--Schwarz estimate yields
\begin{equation}\label{eq:gs psi}
    \|\mathbb G_{\bn,s}\psi\|_{\ell^1}
    \leq C\sqrt{\log(2\bn)}\,|W_{3/2-s}\psi|
    \,.
\end{equation}

We interpolate on a smooth dyadic partition of the nonzero frequencies. Let $C^{-1}\bn^{-1}\leq t\leq C$, and let~$\Pi_t$ be the Fourier multiplier with a smooth $2\pi$-periodic symbol~$m_t$, supported where $t\leq\rho\leq4t$, with $|m_t^{(j)}|\leq Ct^{-j}$ for $0\leq j\leq2$. We choose $O(\log(2\bn))$ such multipliers whose symbols sum to one at every nonzero frequency $2\pi k/\bn$. Since $\mathbf1\cdot\psi=0$, the vector $\mathbb G_{\bn,s/2}\psi$ also has zero sum, so it is the sum of its images under these multipliers. By~\eqref{eq:defS}--\eqref{eq:lowerS}, for~$s$ and for~$s/2$, the product of~$m_t$ with the ratio of the symbols of~$\mathbb G_{\bn,s/2}$ and~$\mathbb G_{\bn,s}$ has derivatives of order $j\leq2$ bounded by $Ct^{-s/2-j}$. The second part of Lemma~\ref{lemma:multiplier kernels}, with $B=Ct^{-s/2}$, therefore bounds the $\ell^1$ norm of the kernel of $\Pi_t\mathbb G_{\bn,s/2}\mathbb G_{\bn,s}^{-1}$ by~$Ct^{-s/2}$. Consequently
\begin{equation*}
    \|\Pi_t\mathbb G_{\bn,s/2}\psi\|_{\ell^1}\leq Ct^{-s/2}\|\mathbb G_{\bn,s}\psi\|_{\ell^1},
    \qquad
    \|\Pi_t\mathbb G_{\bn,s/2}\psi\|_{\ell^2}\leq Ct^{(1-s)/2}\|\Lambda^{s-3/2}\psi\|_{\ell^2}
    \,.
\end{equation*}
Since $1-s/2=(1-s)+s/2$, H\"older's inequality gives
\begin{equation*}
    \|\Pi_t\mathbb G_{\bn,s/2}\psi\|_{\ell^{1/(1-s/2)}}
    \leq\|\Pi_t\mathbb G_{\bn,s/2}\psi\|_{\ell^1}^{1-s}
    \|\Pi_t\mathbb G_{\bn,s/2}\psi\|_{\ell^2}^{s}
    \,,
\end{equation*}
and the factors $t^{-s(1-s)/2}$ and $t^{s(1-s)/2}$ of the two bounds above cancel. Summing the $O(\log(2\bn))$ blocks gives
\begin{equation*}
    \|\mathbb G_{\bn,s/2}\psi\|_{\ell^{1/(1-s/2)}}
    \leq C\log(2\bn)
    \|\mathbb G_{\bn,s}\psi\|_{\ell^1}^{1-s}
    \|\Lambda^{s-3/2}\psi\|_{\ell^2}^s
    \,.
\end{equation*}
We apply H\"older's inequality under~$\nu$, then~\eqref{eq:s1} and~\eqref{eq:gs psi}. Since $\bn\leq Cn^{1/2+1/s}$, the logarithms are absorbed into~$C_n$, which proves~\eqref{eq:fou}.
\end{proof}

\subsection{Control on derivatives}
We now bound the weighted derivatives of~$\psi$ in terms of the fractional primitive controlled in Lemma~\ref{lemma:elliptic}. We first bound the product of a row with~$\mathbb G_{\bn,s/2}^{-1}$ by the row and its increments.

\begin{lemma}\label{lemma:row cancellation}
Let $s\in(0,1)$ and $\bn\geq2$. For $h\in\mathbb Z$, let~$\mathfrak t_h$ be the cyclic translation of rows $w\in\mathbb R^{\bn}$ given by $(\mathfrak t_hw)(j)=w(j-h)$, with indices modulo~$\bn$. Then there exists $C(s)>0$ such that, for every $w\in\mathbb R^{\bn}$,
\begin{equation}\label{eq:row cancellation}
    \|w\,\mathbb G_{\bn,s/2}^{-1}\|_{\ell^{2/s}}
    \leq C(s)\bigl(\bn^{-(1-s/2)}\|w\|_{\ell^{2/s}}+\|w\|_{\ell^{2/s}}^{s/2}\|\mathfrak t_1w-w\|_{\ell^{2/s}}^{1-s/2}\bigr)
    \,.
\end{equation}
\end{lemma}

\begin{proof}
By the definition of~$g_{\bn,s/2}^{-1}$ after~\eqref{eq:def H matrix}, the row $w\,\mathbb G_{\bn,s/2}^{-1}$ is the sum of $g_{\bn,s/2}^{-1}(h)\mathfrak t_hw$ over $1\leq h\leq\bn$. We subtract~$w$ from each translate:
\begin{equation*}
    w\,\mathbb G_{\bn,s/2}^{-1}
    =\Bigl(\sum_{h=1}^{\bn}g_{\bn,s/2}^{-1}(h)\Bigr)w
    +\sum_{h=1}^{\bn}g_{\bn,s/2}^{-1}(h)(\mathfrak t_hw-w)
    \,.
\end{equation*}
Lemma~\ref{lemma:inverse}, with~$s/2$ in place of~$s$, bounds the first sum by $C(s)\bn^{-(1-s/2)}$ and gives the entry bound ${|g_{\bn,s/2}^{-1}(h)|\leq C(s)(1+d(h,0))^{s/2-2}}$. The difference $\mathfrak t_hw-w$ is a sum of $d(h,0)$ translates of $\pm(\mathfrak t_1w-w)$. Since translations preserve the $\ell^{2/s}$ norm, it follows that
\begin{equation*}
    \|\mathfrak t_hw-w\|_{\ell^{2/s}}
    \leq\min\bigl\{2\|w\|_{\ell^{2/s}},d(h,0)\|\mathfrak t_1w-w\|_{\ell^{2/s}}\bigr\}
    \,.
\end{equation*}
Each value $d\geq1$ of~$d(h,0)$ is attained by at most two~$h$, and $\mathfrak t_hw=w$ for $h=\bn$. Let $K>0$. We use the second bound when $d(h,0)\leq K$ and the first one otherwise. The sum of~$d^{s/2-1}$ over the integers $d\in[1,K]$ is at most $C(s)K^{s/2}$, and the sum of~$d^{s/2-2}$ over the integers $d\geq1$ with $d>K$ is at most $C(s)K^{s/2-1}$. Hence
\begin{equation*}
    \sum_{h=1}^{\bn}|g_{\bn,s/2}^{-1}(h)|\|\mathfrak t_hw-w\|_{\ell^{2/s}}
    \leq C(s)\bigl(K^{s/2}\|\mathfrak t_1w-w\|_{\ell^{2/s}}+K^{s/2-1}\|w\|_{\ell^{2/s}}\bigr)
    \,.
\end{equation*}
If $\mathfrak t_1w=w$, then the left-hand side vanishes. Otherwise, the choice $K=\|w\|_{\ell^{2/s}}/\|\mathfrak t_1w-w\|_{\ell^{2/s}}$ makes both terms on the right equal to $C(s)\|w\|_{\ell^{2/s}}^{s/2}\|\mathfrak t_1w-w\|_{\ell^{2/s}}^{1-s/2}$. This proves~\eqref{eq:row cancellation}.
\end{proof}

We now turn to the derivative bound.

\begin{lemma}\label{lemma:diff}
Let $s\in(0,1)$, and let $\nu$ and $\sM$ satisfy Assumptions~\ref{assumptions:mu2} and~\ref{assumptions:sM}. Let $\chi_n\in H^1(\nu)$ and $i_0\in\bar I=\{1,\ldots,\bn\}$, and let $\psi\in H^1(\nu;\dR^{\bar I})$ and $\lambda\in L^2(\nu)$ solve~\eqref{eq:per2n}.
Let~$W_\alpha$ and~$\mathbb G_{\bn,s/2}$ be given by~\eqref{eq:Lalpha5} and~\eqref{eq:def H matrix}. Then there exist $C(\beta)>0$ and $\kappa>0$ such that, for every $\eta\in(0,1/4)$,
\begin{multline}\label{eq:add gamma}
    \|W_{1-s/2}\nabla\psi\|_2^2
    \leq C(\beta)n^{\kappa\ve}\Bigl[\|\chi_n\|_2^2+\eta^{3-s}\|W_{3/2-s}\psi\|_2^2\\
    {}+\|W_{3/2-s}\psi\|_2\Bigl(\frac{\sqrt{\log(2\bn)}}{\eta}\,\|\mathbb G_{\bn,s/2}\psi\|_{L^2(\nu;\ell^{1/(1-s/2)})}+\bn\|\lambda\|_2\Bigr)\Bigr]
    \,.
\end{multline}
\end{lemma}

\begin{proof}
Testing~\eqref{eq:per2n} with~$\psi$ and using~\eqref{eq:hypsM} gives
\begin{equation}\label{eq:ap1}
    \|\psi\|_2+\|\nabla\psi\|_2\leq C_n\|\chi_n\|_2
    \,.
\end{equation}
We then test~\eqref{eq:per2n} with~$W_{2-s}\psi$. The diffusion term gives the left-hand side of~\eqref{eq:add gamma}, and Steps~1 and~2 bound the term $\beta\dE_\nu[W_{2-s}\psi\cdot\sM\psi]$.

\paragraph{\emph{Step 1: exact splitting of the interval representation}}
Let
\begin{equation*}
    \widetilde\sM
    =\sum_{J:\,|J|\leq\eta\min_{j\in J}r_j}
        {a_J}\mathbf1_J\mathbf1_J^T,
    \qquad \sM^{[\eta]}=\sM-\widetilde\sM
    \,.
\end{equation*}
Both matrices are positive semidefinite. A retained interval containing~$i$ has length at most~$\eta r_i$, so $\widetilde\sM_{ij}=0$ unless $d(i,j)\leq\eta r_i$. On this support, $r_j\asymp r_i$. The coefficient bound~\eqref{eq:alphaik} gives ${0\leq\widetilde\sM_{ij}\leq C_n(1+d(i,j))^{-s}}$. By symmetry,
\begin{equation}\label{eq:derivative identity}
    W_{2-s}\psi\cdot\sM\psi
    =W_{1-s/2}\psi\cdot\widetilde\sM W_{1-s/2}\psi
    +\frac12\sum_{i,j}\widetilde\sM_{ij}
        (r_i^{1-s/2}-r_j^{1-s/2})^2\psi_i\psi_j
    +\sum_i r_i^{2-s}\psi_i(\sM^{[\eta]}\psi)_i
    \,.
\end{equation}
For $d=d(i,j)$ on the retained support, the mean value theorem yields
\begin{equation*}
    {\frac{\widetilde\sM_{ij}(r_i^{1-s/2}-r_j^{1-s/2})^2}
         {r_i^{3/2-s}r_j^{3/2-s}}
    \leq C_n r_i^{s-3}\frac{d^2}{(1+d)^s}
        \mathds{1}_{d\leq\eta r_i}}
    \,.
\end{equation*}
At most two indices~$j$ lie at each distance $d\geq1$ from~$i$, and the term $d=0$ vanishes. Hence the row sums are at most
\begin{equation*}
    C_nr_i^{s-3}\sum_{1\leq d\leq\eta r_i}d^{2-s}
    \leq C_nr_i^{s-3}(\eta r_i)^{3-s}
    =C_n\eta^{3-s}
    \,,
\end{equation*}
and the column sums have the same bound by symmetry. Thus, pointwise in the configuration,
\begin{equation}\label{eq:derivative short}
    \Bigl|\frac12\sum_{i,j}\widetilde\sM_{ij}
        (r_i^{1-s/2}-r_j^{1-s/2})^2\psi_i\psi_j\Bigr|
    \leq C_n\eta^{3-s}|W_{3/2-s}\psi|^2
    \,.
\end{equation}

\paragraph{\emph{Step 2: the discarded intervals}}
If a discarded interval~$J$ contains~$i$, then $\min_{j\in J}r_j\geq r_i-|J|$, and hence $|J|>\eta r_i/(1+\eta)$. Counting intervals of each length in~\eqref{eq:alphaik} gives
\begin{equation}\label{eq:derivative tail}
    |\sM^{[\eta]}_{ij}|\leq\frac{C_n}{(1+\eta r_i+d(i,j))^s},
    \qquad
    |\sM^{[\eta]}_{i,j+1}-\sM^{[\eta]}_{ij}|
    \leq\frac{C_n}{(1+\eta r_i+d(i,j))^{1+s}}
    \,.
\end{equation}
Indeed, at most~$\ell$ intervals of length~$\ell$ contain both indices, whereas their first difference selects at most two intervals with an endpoint at~$j$. Every interval contributing to either sum has length at least a constant times~$1+\eta r_i+d(i,j)$.

Let~$\sM^{[\eta]}_i$ be the $i$-th row of~$\sM^{[\eta]}$, and let~$\mathfrak t_1$ be as in Lemma~\ref{lemma:row cancellation}. By~\eqref{eq:derivative tail},
\begin{equation*}
    {\|\sM^{[\eta]}_i\|_{\ell^{2/s}}\leq C_n(1+\eta r_i)^{-s/2},
    \qquad
    \|\mathfrak t_1\sM^{[\eta]}_i-\sM^{[\eta]}_i\|_{\ell^{2/s}}\leq C_n(1+\eta r_i)^{-1-s/2}}
    \,.
\end{equation*}
We insert these bounds into Lemma~\ref{lemma:row cancellation}. Since $1+\eta r_i\leq\bn$, this gives
\begin{equation*}
    {\|\sM^{[\eta]}_i\mathbb G_{\bn,s/2}^{-1}\|_{\ell^{2/s}}
    \leq C_n\bigl(\bn^{-(1-s/2)}(1+\eta r_i)^{-s/2}+(1+\eta r_i)^{-1}\bigr)
    \leq C_n(1+\eta r_i)^{-1}}
    \,.
\end{equation*}
Since $(\sM^{[\eta]}\psi)_i=(\sM^{[\eta]}_i\mathbb G_{\bn,s/2}^{-1})\cdot(\mathbb G_{\bn,s/2}\psi)$ and $s/2+(1-s/2)=1$, H\"older's inequality therefore gives
\begin{equation*}
    {|(\sM^{[\eta]}\psi)_i|
    \leq \frac{C_n}{\eta r_i}
        \|\mathbb G_{\bn,s/2}\psi\|_{\ell^{1/(1-s/2)}}}
    \,.
\end{equation*}
Since $r_i^{1-s}=r_i^{3/2-s}r_i^{-1/2}$ and $\sum_ir_i^{-1}\leq C\log(2\bn)$, the weighted Cauchy--Schwarz inequality yields
\begin{equation}\label{eq:derivative long}
    \Bigl|\dE_\nu\sum_i r_i^{2-s}\psi_i(\sM^{[\eta]}\psi)_i\Bigr|
    \leq \frac{C_n\sqrt{\log(2\bn)}}{\eta}\,\|W_{3/2-s}\psi\|_2\,
    \|\mathbb G_{\bn,s/2}\psi\|_{L^2(\nu;\ell^{1/(1-s/2)})}
    \,.
\end{equation}

\paragraph{\emph{Step 3: the weighted energy estimate}}
We insert~\eqref{eq:derivative identity}, discard the nonnegative quadratic form, and apply~\eqref{eq:derivative short} and~\eqref{eq:derivative long}. The source term is bounded by~$C_n\|\chi_n\|_2^2$ using~\eqref{eq:ap1}. Finally,
\begin{equation*}
    {\Bigl|\dE_\nu\Bigl[\lambda\sum_i r_i^{2-s}\psi_i\Bigr]\Bigr|
    \leq\|\lambda\|_2\|W_{3/2-s}\psi\|_2\Bigl(\sum_i r_i\Bigr)^{1/2}
    \leq C\bn\|\lambda\|_2\|W_{3/2-s}\psi\|_2}
    \,.
\end{equation*}
These estimates give~\eqref{eq:add gamma}.
\end{proof}

\subsection{Global decay estimate}
We multiply the system by a shifted inverse of the reference Riesz matrix. The commutators of~$W_{3/2-s}$ with this inverse and with its product with~$\sM$ combine with the derivative control of Lemma~\ref{lemma:diff} to give global decay. We first bound the commutator of~$W_{3/2-s}$ with a matrix whose entries decay like $(1+d(i,j))^{s-2}$.

\begin{lemma}\label{lemma:endpoint commutator}
Let $s\in(0,1)$, let $\bn\geq2$ and $i_0\in\bar I=\{1,\ldots,\bn\}$, and let $r_i=1+d(i,i_0)$, so that $W_\alpha=\operatorname{diag}(r_i^\alpha)$ by~\eqref{eq:Lalpha5}. Let~$\mathsf B$ be a real $\bn\times\bn$ matrix with $|\mathsf B_{ij}|\leq(1+d(i,j))^{s-2}$ for every $i,j\in\bar I$. Then there exists $C(s)>0$ such that, for every $R\geq1$ and $v\in\mathbb R^{\bn}$,
\begin{equation}\label{eq:res endpoint}
    |[W_{3/2-s},\mathsf B]v|
    \leq C(s)R^{s-1}|W_{3/2-s}v|+C(s)(R^2+\log(2\bn))|W_{1/2}v|
    \,.
\end{equation}
\end{lemma}

\begin{proof}
The entries of $[W_{3/2-s},\mathsf B]$ are $(r_i^{3/2-s}-r_j^{3/2-s})\mathsf B_{ij}$, and ${|r_i-r_j|\leq d(i,j)}$ by the triangle inequality. We bound separately the parts of the commutator supported on four sets of pairs~$(i,j)$.

For the pairs with $r_i\geq R$ and $d(i,j)\leq r_i/2$, we have $r_i/2\leq r_j\leq3r_i/2$, and the mean value theorem gives
\begin{equation*}
    \frac{|r_i^{3/2-s}-r_j^{3/2-s}||\mathsf B_{ij}|}{r_j^{3/2-s}}
    \leq\frac{Cd(i,j)}{r_i(1+d(i,j))^{2-s}}
    \,.
\end{equation*}
Since $s>0$, the sum of the right-hand side over~$j$ is at most $Cr_i^{s-1}$, which is at most $CR^{s-1}$ since $s<1$. These pairs also satisfy $r_i\geq\max(R,2r_j/3)$ and $d(i,j)\leq r_j$, so the sum over~$i$ is at most $Cr_j^s/\max(R,r_j)\leq CR^{s-1}$. The Schur test bounds this part by $CR^{s-1}|W_{3/2-s}v|$.

For the pairs with $r_i\geq R$ and $d(i,j)>r_i/2$, we have $|\mathsf B_{ij}|\leq Cr_i^{s-2}$. Since $\sum_jr_j^{-1}\leq C\log(2\bn)$, the Cauchy--Schwarz inequality gives, for the terms containing $r_i^{3/2-s}$,
\begin{equation*}
    \sum_{d(i,j)>r_i/2}r_i^{3/2-s}|\mathsf B_{ij}||v_j|
    \leq Cr_i^{-1/2}\sum_j|v_j|
    \leq Cr_i^{-1/2}\sqrt{\log(2\bn)}\,|W_{1/2}v|
    \,.
\end{equation*}
Summing the squares over~$i$ bounds these terms by $C\log(2\bn)|W_{1/2}v|$. In the terms containing $r_j^{3/2-s}$, the entries $|\mathsf B_{ij}|$ have row sums at most $Cr_i^{s-1}$ and, since $d(i,j)>R/2$, column sums at most $CR^{s-1}$. The Schur test bounds these terms by $CR^{s-1}|W_{3/2-s}v|$.

For the pairs with $r_i<R$ and $r_j\leq4R$, the entries are at most $(4R)^{3/2-s}$. At most $2R$ indices~$i$ satisfy $r_i<R$, and the sum of~$r_j^{-1}$ over $r_j\leq4R$ is at most $C\log(2R)$. Since $s>0$, the Cauchy--Schwarz inequality bounds this part by
\begin{equation*}
    C(2R)^{1/2}(4R)^{3/2-s}\sqrt{\log(2R)}\,|W_{1/2}v|
    \leq C(s)R^2|W_{1/2}v|
    \,.
\end{equation*}

For the pairs with $r_i<R$ and $r_j>4R$, we have $r_i<r_j$ and $d(i,j)\geq r_j-r_i\geq3r_j/4$. Hence the entries are at most $Cr_j^{3/2-s}r_j^{s-2}=Cr_j^{-1/2}$, and the Cauchy--Schwarz inequality gives, for $r_i<R$,
\begin{equation*}
    \sum_{r_j>4R}|r_i^{3/2-s}-r_j^{3/2-s}||\mathsf B_{ij}||v_j|
    \leq C\Bigl(\sum_{r_j>4R}r_j^{-2}\Bigr)^{1/2}|W_{1/2}v|
    \leq CR^{-1/2}|W_{1/2}v|
    \,.
\end{equation*}
Summing the squares over the at most $2R$ indices~$i$ with $r_i<R$ gives $C|W_{1/2}v|$. Adding the four bounds proves~\eqref{eq:res endpoint}.
\end{proof}

We now prove the global decay estimate.

\begin{lemma}\label{lemma:global}
Let $s\in(0,1)$, and let~$\nu$ and~$\sM$ satisfy Assumptions~\ref{assumptions:mu2} and~\ref{assumptions:sM}. Let $\chi_n\in H^{1}(\nu)$ and $i_0\in\bar I=\{1,\ldots,\bn\}$, with~$\bn$ as before~\eqref{eq:blockHN}, and let $\psi\in H^1(\nu;\dR^{\bar I})$ and $\lambda\in L^2(\nu)$ solve~\eqref{eq:per2n}.
Let~$W_\alpha$ be given by~\eqref{eq:Lalpha5}. Then there exist $C(\beta)>0$ and $\kappa>0$ such that
\begin{equation}\label{eq:glob opt}
    \|W_{1-s/2}\nabla\psi\|_2+\|W_{3/2-s}\psi\|_2\leq C(\beta)n^{\kappa\ve}\|\chi_n\|_2
    \,,
\end{equation}
and
\begin{equation}\label{eq:f lambda}
    \|\lambda\|_2\leq C(\beta)n^{\kappa\ve}\bn^{-1}\|\chi_n\|_2
    \,.
\end{equation}
\end{lemma}

\begin{proof}
By linearity, we may assume that $\|\chi_n\|_2=1$. Let $r_{ij}=1+d(i,j)$.

\paragraph{\emph{Step 1: a deterministic matrix whose product with $\sM$ is coercive}}
Assumptions~\ref{assumptions:sM} and the bound on the row sums of~$E=\sM-\sR_{\bn}$, defined after~\eqref{eq:Lalpha5}, give a deterministic $C_n\geq1$ such that $\sM\geq C_n^{-1}I$ and $\|E\|\leq C_n$. Let
\begin{equation*}
    {\tau=C_n^{-1}+C_n^3},
    \qquad \sA={\tau}(\sR_{\bn}+{\tau} I)^{-1}
    \,.
\end{equation*}
For $u=(\sR_{\bn}+{\tau} I)v$, Young's inequality gives
\begin{equation*}
    u\cdot\sA\sM u
    ={\tau}\bigl({|\sR_{\bn}v|^2+Ev\cdot\sR_{\bn}v}
                   +{\tau}\,v\cdot\sM v\bigr)
    \geq {\frac{\tau}{2}|\sR_{\bn}v|^2+\frac{\tau^2}{2C_n}|v|^2}
    \geq {\frac{1}{4C_n}|u|^2}
    \,.
\end{equation*}
Thus
\begin{equation}\label{eq:positivity}
    \frac{\sA\sM+\sM\sA}{2}\geq{\frac{1}{4C_n}I}
    \,.
\end{equation}

Since $\sR_{\bn}$ is symmetric and circulant, so is~$\sA$, and we let $\mathsf a(h)=\sA_{i,i+h}$, which does not depend on~$i$. Since $\sA(\sR_{\bn}+\tau I)=\tau I$ and $\sR_{\bn}\mathbf1=2\vartheta_{\bn}\mathbf1$,
\begin{equation}\label{eq:res mass}
    \sA\mathbf1=\frac{\tau}{2\vartheta_{\bn}+\tau}\mathbf1,
    \qquad \sA \sR_{\bn}={\tau I-\tau\sA}
    \,.
\end{equation}
At nonzero frequencies the symbol of~$\sA$ is $\tau/(2S(\theta)+\tau)$. Its derivatives of orders $j\leq3$ are bounded by $C\tau\rho(\theta)^{1-s-j}$, and those of its product with $e^{-i\theta}-1$ by $C\tau\rho(\theta)^{2-s-j}$. The first part of Lemma~\ref{lemma:multiplier kernels}, with $r=2$ and $\alpha=1-s$ for~$\mathsf a$, and with $r=3$ and $\alpha=2-s$ for $h\mapsto\mathsf a(h+1)-\mathsf a(h)$, whose zero mode vanishes, gives
\begin{equation}\label{eq:res kernel}
    |\mathsf a(h)|\leq\frac{C_n}{(1+d(h,0))^{2-s}},
    \qquad
    |\mathsf a(h+1)-\mathsf a(h)|\leq\frac{C_n}{(1+d(h,0))^{3-s}}
    \,.
\end{equation}
Since~$\sR_{\bn}$ is positive definite and does not depend on the configuration, $\sA$ is positive definite and commutes with~$\mc L^\nu$.

\paragraph{\emph{Step 2: the spatial commutator with~$W_{3/2-s}$}}
Let
\begin{equation*}
    X=\|W_{3/2-s}\psi\|_{2},\qquad
    Z=\|W_{1/2}\psi\|_{2},\qquad
    \mathfrak N=\|\mathbb G_{\bn,s/2}\psi\|_{L^2(\nu;\ell^{1/(1-s/2)})}
    \,.
\end{equation*}
For every $\omega\in(0,1)$, we prove
\begin{equation}\label{eq:res comm}
    \|[W_{3/2-s},\sA\sM]\psi\|_{2}
    \leq \omega X+C_{n,\omega}\bigl(\log(2\bn)Z+\mathfrak N\bigr)
    \,,
\end{equation}
where $C_{n,\omega}$ is bounded by a polynomial in $n^\ve$ and~$\omega^{-1}$.

With indices modulo~$\bn$, the entrywise bound~\eqref{eq:tildeA} and the interval representations give
\begin{equation}\label{eq:res error}
\begin{gathered}
    |E_{ij}|\leq C_nr_{ij}^{-1-s/2},\qquad
    |E_{i,j+1}-E_{ij}|+|E_{i+1,j}-E_{ij}|\leq C_nr_{ij}^{-1-s},\\
    |E_{i+1,j+1}-E_{i+1,j}-E_{i,j+1}+E_{ij}|\leq C_nr_{ij}^{-2-s}
    \,.
\end{gathered}
\end{equation}
Indeed, a first difference keeps only intervals with a prescribed endpoint. At distance at least two, a mixed difference selects at most two interval coefficients. At distance at most one, it is bounded by four entries of~$E$, hence by~$C_n$. These difference bounds also hold for the full reference~$\sR_{\bn}$.

Let $\zeta:[0,\infty)\to[0,1]$ be smooth, with $\zeta=0$ on $[0,1]$ and $\zeta=1$ on $[2,\infty)$, and, for $K\geq1$, let $E_K$ be the matrix with entries ${\zeta(d(i,j)/K)E_{ij}}$. If $K\geq\bn/8$, we take $E_K=0$. Otherwise~\eqref{eq:res error} holds for~$E_K$, and its absolute row and column sums are at most $C_nK^{-s/2}$. Since $\sM=\sR_{\bn}+(E-E_K)+E_K$, we have
\begin{equation*}
    [W_{3/2-s},\sA\sM]=[W_{3/2-s},\sA\sR_{\bn}]+[W_{3/2-s},\sA(E-E_K)]+W_{3/2-s}\sA E_K-\sA E_KW_{3/2-s}
    \,.
\end{equation*}
Since $\sA\sR_{\bn}=\tau I-\tau\sA$ and $E-E_K$ vanishes at distance larger than~$2K$, the entries of $\sA\sR_{\bn}$ and $\sA(E-E_K)$ are bounded by $r_{ij}^{s-2}$ times a polynomial in~$n^\ve$ and~$K$. Lemma~\ref{lemma:endpoint commutator} therefore bounds the first two commutators. The absolute row and column sums of~$\sA E_K$ are at most $C_nK^{-s/2}\|\mathsf a\|_{\ell^1}$, which bounds the last term by $C_nK^{-s/2}X$.

It remains to estimate $W_{3/2-s}\sA E_K\psi$. The entries of row~$i$ of~$\sA E_K$ are $\sum_h\mathsf a(h)(E_K)_{i-h,j}$. We split the kernel~$\mathsf a(h)$ with the cutoff $1-\zeta(16d(h,0)/r_i)$, which vanishes for $d(h,0)\geq r_i/8$ and equals one for $d(h,0)\leq r_i/16$. By~\eqref{eq:res kernel} and the column sums of~$E_K$, the complementary part has entries
\begin{equation*}
    \sum_{h:\,d(h,0)\geq r_i/16}|\mathsf a(h)(E_K)_{i-h,j}|
    \leq C_nr_i^{s-2}\sum_h|(E_K)_{i-h,j}|
    \leq C_nr_i^{s-2}
    \,.
\end{equation*}
After weighting, its row~$i$ is at most $C_nr_i^{-1/2}\|\psi\|_{\ell^1}$. Since $\sum_ir_i^{-1}\leq C\log(2\bn)$, the weighted Cauchy--Schwarz inequality gives $\|\psi\|_{\ell^1}\leq C\sqrt{\log(2\bn)}\,|W_{1/2}\psi|$, and this part contributes at most $C_n\log(2\bn)Z$. We split the retained part with the second cutoff $\zeta(4d(i,j)/r_i)$, which vanishes for $d(i,j)\leq r_i/4$ and equals one for $d(i,j)\geq r_i/2$. The part with the factor $1-\zeta(4d(i,j)/r_i)$ is supported on $d(i,j)\leq r_i/2$, where $r_i/r_j$ lies between two positive constants. Since both cutoffs take values in~$[0,1]$, its absolute row sums satisfy
\begin{equation*}
    \sum_{j:\,d(i,j)\leq r_i/2}\Bigl|\sum_h\mathsf a(h)\Bigl(1-\zeta\Bigl(\frac{16d(h,0)}{r_i}\Bigr)\Bigr)(E_K)_{i-h,j}\Bigr|
    \leq\|\mathsf a\|_{\ell^1}\max_k\sum_j|(E_K)_{kj}|
    \leq C_nK^{-s/2}
    \,,
\end{equation*}
and its absolute column sums have the same bound. The Schur test therefore bounds its contribution by $C_nK^{-s/2}X$.

The part with the factor $\zeta(4d(i,j)/r_i)$ has the rows
\begin{equation*}
    {w_i(j)=\zeta\Bigl(\frac{4d(i,j)}{r_i}\Bigr)
    \sum_h \mathsf a(h)\Bigl(1-\zeta\Bigl(\frac{16d(h,0)}{r_i}\Bigr)\Bigr)(E_K)_{i-h,j}}
    \,.
\end{equation*}
We write $(E_K)_{i-h,j}=(E_K)_{ij}+((E_K)_{i-h,j}-(E_K)_{ij})$. By~\eqref{eq:res mass}, the sum of~$\mathsf a$ is $\tau/(2\vartheta_{\bn}+\tau)\leq C_n\bn^{s-1}$, and by~\eqref{eq:res kernel} the sum of~$|\mathsf a(h)|$ over $d(h,0)\geq r_i/16$ is at most $C_nr_i^{s-1}$. Hence the factor of~$(E_K)_{ij}$ is at most $C_nr_i^{s-1}$. On the support of~$w_i$, we have $d(i-h,j)\geq d(i,j)/2$. For the increments $(E_K)_{i-h,j}-(E_K)_{ij}$, we use the first-difference bound of~\eqref{eq:res error} when $d(h,0)\leq r_{ij}^{s/2}$, and the entry bound otherwise. This gives
\begin{align*}
    |w_i(j)|
    &\leq C_n\bigl(r_i^{s-1}r_{ij}^{-1-s/2}+r_{ij}^{-(1+s-s^2/2)}\bigr)
      \leq C_nr_i^{-(1+s-s^2/2)}(r_i/r_{ij})^{1+s/2},\\
    |w_i(j-1)-w_i(j)|
    &\leq C_n\bigl(r_i^{s-1}r_{ij}^{-1-s}+r_i^sr_{ij}^{-2-s}\bigr)
      \leq C_nr_i^{-2}(r_i/r_{ij})^{1+s}
    \,.
\end{align*}
The derivatives of the cutoffs satisfy the same bounds. Consequently
\begin{equation*}
    \|w_i\|_{\ell^{2/s}}\leq C_nr_i^{-(1+s/2-s^2/2)},
    \qquad
    \|\mathfrak t_1w_i-w_i\|_{\ell^{2/s}}\leq C_nr_i^{-(2-s/2)}
    \,.
\end{equation*}
We apply Lemma~\ref{lemma:row cancellation} to~$w_i$. The second term of~\eqref{eq:row cancellation} has the exponent
\begin{equation*}
    \frac s2\Bigl(1+\frac s2-\frac{s^2}2\Bigr)+\Bigl(1-\frac s2\Bigr)\Bigl(2-\frac s2\Bigr)=(2-s)\Bigl(1+\frac{s^2}4\Bigr)
    \,,
\end{equation*}
and the first term has the larger exponent $2-s^2/2$, since $r_i\leq\bn$. Since $s/2+(1-s/2)=1$, H\"older's inequality therefore gives
\begin{equation*}
    |w_i\cdot\psi|\leq\|w_i\,\mathbb G_{\bn,s/2}^{-1}\|_{\ell^{2/s}}\|\mathbb G_{\bn,s/2}\psi\|_{\ell^{1/(1-s/2)}}
    \leq C_nr_i^{-(2-s)(1+s^2/4)}\|\mathbb G_{\bn,s/2}\psi\|_{\ell^{1/(1-s/2)}}
    \,.
\end{equation*}
We multiply by $r_i^{3/2-s}$ and use $\sum_i r_i^{-1-s^2(2-s)/2}\leq C(s)$. This proves
\begin{equation*}
    \|W_{3/2-s}\sA E_K\psi\|_{2}
    \leq C_nK^{-s/2}X+C_n\bigl(\log(2\bn)Z+\mathfrak N\bigr)
    \,.
\end{equation*}
We choose $K$ to be a sufficiently large power of $n^\ve/\omega$, and we apply Lemma~\ref{lemma:endpoint commutator} by linearity, with~$R$ a sufficiently large power of $n^\ve/\omega$. This proves~\eqref{eq:res comm}.

\paragraph{\emph{Step 3: the differential commutator}}
For a vector~${v}$ with $\sum_i {v_i}=0$, symmetry gives
\begin{equation*}
    W_{3/2-s}{v}\cdot[W_{3/2-s},\sA]{v}
    =\frac12\sum_{i,j}\sA_{ij}(r_i^{3/2-s}-r_j^{3/2-s})^2v_iv_j
    \,.
\end{equation*}
On $r_i/2\leq r_j\leq2r_i$, we divide by ${r_i^{1-s/2}r_j^{1-s/2}}$ and use~\eqref{eq:res kernel}. The row and column sums of the resulting kernel are bounded by~$C_n$, since
$2(3/2-s)-2(1-s/2)+1-(2-s)=0$.
By symmetry in~$i$ and~$j$, it remains to treat the pairs with $r_j<r_i/2$. There we expand $(r_i^{3/2-s}-r_j^{3/2-s})^2$ into three terms. In the term with~$r_i^{3-2s}$, the zero sum of~$v$ gives
\begin{equation*}
    \sum_{r_j<r_i/2}\sA_{ij}{v_j}
    =\sum_{r_j<r_i/2}(\sA_{ij}-\sA_{i,i_0})v_j
      -\sA_{i,i_0}\sum_{r_j\geq r_i/2}{v_j}
    \,.
\end{equation*}
By the difference bound in~\eqref{eq:res kernel} and the weighted Cauchy--Schwarz inequality, this sum is at most $C_nr_i^{-(5-3s)/2}|W_{1-s/2}v|$. We multiply by $r_i^{3-2s}|v_i|$ and sum over~$i$. Since
\begin{equation*}
    3-2s-\frac{5-3s}{2}-\Bigl(1-\frac s2\Bigr)=-\frac12
    \,,
\end{equation*}
the weighted Cauchy--Schwarz inequality shows that this gives a factor~$\sqrt{\log(2\bn)}$.
The other two terms, with ${r_i^{3/2-s}r_j^{3/2-s}}$ and~${r_j^{3-2s}}$, have the same bound by the weighted Cauchy--Schwarz inequality. Hence
\begin{equation}\label{eq:subtle}
    |W_{3/2-s}{v}\cdot[W_{3/2-s},\sA]{v}|
    \leq C_n\sqrt{\log(2\bn)}\,|W_{1-s/2}{v}|^2
    \,.
\end{equation}
We apply this to each coordinate of~$\nabla\psi$.

\paragraph{\emph{Step 4: the deformed equation}}
Let $Y=\|W_{1-s/2}\nabla\psi\|_2$. Since $\sum_i\psi_i=0$ and $\sR_{\bn}$ has constant row sums, summing~\eqref{eq:per2n} gives
\begin{equation}\label{eq:ee2}
    \bn\lambda=\beta\mathbf1\cdot E\psi-\chi_n,
    \qquad
    \bn\|\lambda\|_{2}
    \leq C_n\bigl(1+\sqrt{\log(2\bn)}Z\bigr)
    \,.
\end{equation}
By~\eqref{eq:res kernel} and~\eqref{eq:res mass}, the source terms satisfy
$|W_{3/2-s}\sA e_{i_0}|\leq C_n\sqrt{\log(2\bn)}$
and $|W_{3/2-s}\sA\mathbf1|\leq C_n\bn$.

Let $u=W_{3/2-s}\psi$. Multiplication of~\eqref{eq:per2n} by $W_{3/2-s}\sA$ yields
\begin{equation*}
    \beta\sA\sM u+\sA\mc L^\nu u
    =W_{3/2-s}\sA(\chi_ne_{i_0}+\lambda\mathbf1)
    -\beta[W_{3/2-s},\sA\sM]\psi
    -[W_{3/2-s},\sA]\mc L^\nu\psi
    \,.
\end{equation*}
We test with~$u$ and integrate under~$\nu$. Since $\nabla u_i=r_i^{3/2-s}\nabla\psi_i$, this gives
\begin{multline*}
    \beta\dE_\nu[u\cdot\sA\sM u]+\dE_\nu\Bigl[\sum_{i,j}\sA_{ij}\nabla u_i\cdot\nabla u_j\Bigr]
    =\dE_\nu[u\cdot W_{3/2-s}\sA(\chi_ne_{i_0}+\lambda\mathbf1)]\\
    -\beta\dE_\nu\bigl[u\cdot[W_{3/2-s},\sA\sM]\psi\bigr]
    -\dE_\nu\Bigl[\sum_{i,j}r_i^{3/2-s}[W_{3/2-s},\sA]_{ij}\nabla\psi_i\cdot\nabla\psi_j\Bigr]
    \,.
\end{multline*}
Since~$\sA$ is nonnegative, the second term on the left is nonnegative. By~\eqref{eq:positivity}, the first one is at least $\beta X^2/(4C_n)$. On the right,~\eqref{eq:ee2} and the source bounds above bound the first term by $C_n\sqrt{\log(2\bn)}\,X(1+Z)$. By~\eqref{eq:res comm} with $\omega=(8C_n)^{-1}$ for this constant~$C_n$, the second term is at most $\beta X^2/(8C_n)+C_nX(\log(2\bn)Z+\mathfrak N)$. We apply~\eqref{eq:subtle} to each coordinate of~$\nabla\psi$, and bound the third term by $C_n\sqrt{\log(2\bn)}\,Y^2$. We absorb the logarithms into~$C_n$. This gives
\begin{equation}\label{eq:input1}
    X^2\leq C_n\bigl(X(1+Z+\mathfrak N)+Y^2\bigr)
    \,.
\end{equation}

\paragraph{\emph{Step 5: combination of the weighted estimates}}
The unweighted energy estimate for~\eqref{eq:per2n} gives the bound ${\|\psi\|_2+\|\nabla\psi\|_2\leq C_n}$. Weighted interpolation between~$\|\psi\|_2$ and~$X$, and~\eqref{eq:fou}, imply
\begin{equation*}
    Z\leq {C_nX^{1/(3-2s)}},\qquad
    \mathfrak N\leq C_n(1+Z)^sX^{1-s}
    \,.
\end{equation*}
We apply Lemma~\ref{lemma:diff}, bound $\bn\|\lambda\|_2$ by~\eqref{eq:ee2}, and absorb the logarithms into~$C_n$. This gives, for every $\eta\in(0,1/4)$,
\begin{equation}\label{eq:input2}
    Y^2\leq C_n\left[
    X\bigl(\eta^{3-s}X+\eta^{-1}{\mathfrak N}+1+Z\bigr)+1
    \right]
    \,.
\end{equation}
We fix $C_n\geq1$ large enough for~\eqref{eq:input1} and~\eqref{eq:input2}, and choose $\eta=(64C_n^2)^{-1/(3-s)}$. Then $\eta\in(0,1/4)$ and $\eta^{-1}\leq Cn^{C\ve}$. Inserting~\eqref{eq:input2} into~\eqref{eq:input1} absorbs the quadratic term. Increasing~$C_n$ gives
\begin{equation*}
    X^2\leq C_n\bigl[X(1+Z+\mathfrak N)+1\bigr]
    \,.
\end{equation*}
After substituting the interpolation bounds, the powers of~$X$ on the right are
\begin{equation*}
    1,\qquad 1+\frac1{3-2s},\qquad 2-s,\qquad
    2-\frac{s(2-2s)}{3-2s}
    \,.
\end{equation*}
Each is strictly less than~$2$, so Young's inequality gives ${X^2\leq X^2/2+C_n}$. Equation~\eqref{eq:input2} then bounds~$Y$, and~\eqref{eq:ee2} bounds~$\lambda$. These are~\eqref{eq:glob opt} and~\eqref{eq:f lambda}.
\end{proof}

We now extend the global decay estimate of Lemma~\ref{lemma:global} to the periodic system without the constraint $\mathbf1\cdot\psi=0$.

\begin{lemma}\label{lemma:global bis}
Let $s\in (0,1)$, and let~$\nu$ and~$\sM$ satisfy Assumptions~\ref{assumptions:mu2} and~\ref{assumptions:sM}. Let $\chi_n\in H^{1}(\nu)$ and $i_0\in\bar I=\{1,\ldots,\bn\}$, with~$\bn$ as before~\eqref{eq:blockHN}, and let $\psi\in H^1(\nu;\dR^{\bar I})$ solve
\begin{equation}\label{eq:568}
    \beta\sM\psi+\mc{L}^{\nu} \psi=\chi_n e_{i_0}
    \,.
\end{equation}
Let~$W_\alpha$ be given by~\eqref{eq:Lalpha5}. Then there exist $C(\beta)>0$ and $\kappa>0$ such that
\begin{equation*}
    {\|W_{1-s/2}\nabla\psi\|_2+\|W_{3/2-s}\psi\|_2\leq C(\beta)n^{\kappa\ve}\|\chi_n\|_2}
    \,.
\end{equation*}
\end{lemma}

\begin{proof}
By linearity, we may assume that $\|\chi_n\|_2=1$. Let $v$ and~$\lambda$ solve~\eqref{eq:per2n}, with~$v$ in place of~$\psi$. The difference $w=\psi-v$ solves
\begin{equation}\label{eq:secw}
    (\beta\sM+\mc L^\nu)w=-\lambda\mathbf1
    \,.
\end{equation}
Lemma~\ref{lemma:global} gives
\begin{equation}\label{eq:viagain}
    \|W_{3/2-s}v\|_2+\|W_{1-s/2}\nabla v\|_2\leq C_n
    \,,
\end{equation}
and~\eqref{eq:f lambda} gives $\bn\|\lambda\|_2\leq C_n$.

Let $f\in H^1(\nu)$ solve the scalar weak equation
\begin{equation}\label{eq:def fff}
    (2\beta\vartheta_{\bn}+\mc L^\nu)f=-\lambda
    \,.
\end{equation}
Testing with~$f$ gives $2\beta\vartheta_{\bn}\|f\|_2^2+\|\nabla f\|_2^2\leq\|\lambda\|_2\|f\|_2$. Consequently,
\begin{equation}\label{eq:f1}
    \|f\|_2\leq C_n\bn^{-(2-s)},
    \qquad
    \|\nabla f\|_2\leq C_n\bn^{-(3-s)/2}
    \,.
\end{equation}
Let $w^{(2)}=w-f\mathbf1$. Subtracting~\eqref{eq:def fff} from~\eqref{eq:secw} yields the componentwise identity
\begin{equation}\label{eq:mean correction}
    (\beta\sM+\mc L^\nu)w^{(2)}=-\beta fE\mathbf1
    \,.
\end{equation}
Since the absolute row sums of~$E=\sM-\sR_{\bn}$, defined after~\eqref{eq:Lalpha5}, are bounded by~$C_n$, testing~\eqref{eq:mean correction} with~$w^{(2)}$ and using~\eqref{eq:hypsM} give
\begin{equation*}
    \|w^{(2)}\|_2+\|\nabla w^{(2)}\|_2
    \leq C_n\sqrt{\bn}\,\|f\|_2
    \leq C_n\bn^{-(3/2-s)}
    \,.
\end{equation*}
Since ${3/2-s>1-s/2}$, adding~$f\mathbf1$ and using~\eqref{eq:f1} yield
\begin{equation*}
    {\|w\|_2\leq C_n\bn^{-(3/2-s)},
    \qquad
    \|\nabla w\|_2\leq C_n\bn^{-(1-s/2)}}
    \,.
\end{equation*}
The entries of $W_{3/2-s}$ and $W_{1-s/2}$ are at most $C\bn^{3/2-s}$ and~$C\bn^{1-s/2}$. Thus the norms $\|W_{3/2-s}w\|_2$ and $\|W_{1-s/2}\nabla w\|_2$ are at most~$C_n$, and~\eqref{eq:viagain} proves the assertion for~$\psi=v+w$.
\end{proof}

\subsection{Localization and optimal decay}\label{sub:loc2}
We now improve the global weighted bounds of Lemmas~\ref{lemma:global} and~\ref{lemma:global bis} to pointwise decay.

\begin{proposition}\label{proposition:loc}
Let $s\in (0,1)$, and let~$\nu$ and~$\sM$ satisfy Assumptions~\ref{assumptions:mu2} and~\ref{assumptions:sM}. Let $\chi_n\in H^{1}(\nu)$ and $i_0\in\bar I=\{1,\ldots,\bn\}$, with~$\bn$ as before~\eqref{eq:blockHN}, and let $\psi\in H^1(\nu;\dR^{\bar I})$ solve either~\eqref{eq:per2n} for some $\lambda\in L^2(\nu)$, or~\eqref{eq:568}.
Then there exist $C(\beta)>0$ and $\kappa>0$ such that, for every $i\in\bar I$,
\begin{equation}\label{eq:final control}
    \|\psi_i\|_2\leq \frac{C(\beta)n^{\kappa\ve}}{(1+d_{\bn}(i,i_0))^{2-s}}\|\chi_n\|_2
    \,,
\end{equation}
\begin{equation}\label{eq:final control bis}
    \|\nabla\psi_i\|_2\leq \frac{C(\beta)n^{\kappa\ve}}{(1+d_{\bn}(i,i_0))^{(3-s)/2}}\|\chi_n\|_2
    \,.
\end{equation}
\end{proposition}

\begin{proof}
By density and coercivity, the estimates of Lemmas~\ref{lemma:global} and~\ref{lemma:global bis} extend to sources $\chi_n\in L^2(\nu)$. We bootstrap the decay exponents of the solutions of~\eqref{eq:per2n} and~\eqref{eq:568}. Steps~1 to~6 treat~\eqref{eq:per2n}. They apply the current bounds to equations without constraint on a cycle built from a window of $n_0$ labels, with $n_0$ in place of~$n$ and $\ve\log n/\log n_0$ in place of~$\ve$, so that the losses $n^{C\ve}$ do not change. Step~7 deduces the bounds for~\eqref{eq:568} from those for~\eqref{eq:per2n}. The bounds must therefore hold for both equations, uniformly in~$\sM$ and in~$\ve$ bounded below. Let $p\geq3/2-s$ and $1-s/2\leq q\leq p$, and let $\kappa>0$. We assume that, for every $\ve_{\min}>0$, there exists $C(\beta,\ve_{\min})$ such that, uniformly in $\ve\geq\ve_{\min}$, in~$\sM$ satisfying Assumptions~\ref{assumptions:sM}, in $\chi_n\in L^2(\nu)$ and in $i_0\in\bar I$, every solution $\psi\in H^1(\nu;\dR^{\bar I})$ of~\eqref{eq:per2n} or~\eqref{eq:568} satisfies, for each $j\in\bar I$,
\begin{equation}\label{eq:as alpha}
    \|\psi_j\|_2\leq \frac{C(\beta,\ve_{\min})n^{\kappa\ve}}{(1+d_{\bn}(j,i_0))^{p}}\|\chi_n\|_2
    \,,
\end{equation}
\begin{equation}\label{eq:as gamma}
    \|\nabla\psi_j\|_2\leq  \frac{C(\beta,\ve_{\min})n^{\kappa\ve}}{(1+d_{\bn}(j,i_0))^{q}}\|\chi_n\|_2
    \,.
\end{equation}
In addition to~\eqref{eq:as alpha} and~\eqref{eq:as gamma}, we also use the global estimates of Lemma~\ref{lemma:global} and Lemma~\ref{lemma:global bis} throughout. Before the logarithms are absorbed, the constants of Lemmas~\ref{lemma:elliptic}, \ref{lemma:diff}, \ref{lemma:global} and~\ref{lemma:global bis} are bounded by $Cn^{C\ve}\log^C(2\bn)$, where~$C$ does not depend on~$\ve$. Since $\bn\leq Cn^{1/2+1/s}$, we have ${\log^C(2\bn)\leq C(s,\ve_{\min})n^{C\ve}}$ for $\ve\geq\ve_{\min}$. Lemmas~\ref{lemma:global} and~\ref{lemma:global bis} therefore give the hypothesis for $p=3/2-s$ and $q=1-s/2$, uniformly in $\ve\geq\ve_{\min}$. By linearity, we may assume that $\|\chi_n\|_2=1$.

\paragraph{\emph{Step 1: projection and embedding of the equation}}
Fix $j\in\bar I$ and let $R=d_{\bar n}(j,i_0)$. We may assume that $R\geq n^{C\ve}$ for a sufficiently large constant~$C$. Indeed, the unweighted energy estimate $\|\psi\|_2+\|\nabla\psi\|_2\leq C_n$ gives the bounds for smaller~$R$, after increasing the exponent of~$n^\ve$ in their conclusions. Let
\begin{equation*}
    J=\{i\in\bar I:d_{\bar n}(i,j)\leq R/8\}
\end{equation*}
and $n_0=|J|$. Let $\psi^J=(\psi_i)_{i\in J}$ and $\sM^J=(\sM_{i\ell})_{i,\ell\in J}$, and let $U_\ell=\sum_{i\notin J}\sM_{\ell i}\psi_i$ for $\ell\in J$. Projecting~\eqref{eq:per2n} on~$J$ gives
\begin{equation*}
    \beta\sM^J\psi^J+\mc L^\nu\psi^J
    =-\beta U+\lambda\mathbf1_J
    \,.
\end{equation*}
We shift the labels of~$\bar I$ cyclically so that $J=\{1,\ldots,n_0\}$, keep the letters~$j$ and~$i_0$ for the shifted labels, and embed~$J$ in the cycle~$\bar J$ of size
\begin{equation*}
    \bar n_0=\min\{\bar n,\max\{2n_0,\lceil n_0^{1/2+1/s}\rceil\}\}
    \,.
\end{equation*}
We keep the cutoff~$K_0$ of~\eqref{eq:blockHN}. The lower bound on~$R$ ensures that $K_0\leq\bar n_0/4$. Let~$\sM_0$ be the matrix obtained from~$\sR_{\bar n_0}^{[K_0]}$ by replacing the terms of the intervals contained in~$J$ by the corresponding terms of~$\sM$. Let~$B_0$ and~$D_0$ be its blocks on $J\times(\bar J\setminus J)$ and on $(\bar J\setminus J)^2$, and let~$E_\partial$ be the difference between~$\sM^J$ and the block of~$\sM_0$ on~$J\times J$.

We check that~$\sM_0$ satisfies Assumptions~\ref{assumptions:sM} on the cycle~$\bar J$. The coefficient of each interval in~$\sM_0$ is a coefficient of~$\sM$ or of~$\sR_{\bar n_0}^{[K_0]}$. The coefficients of~$\sR_{\bar n_0}^{[K_0]}$ vanish for intervals of length less than~$K_0$, so~\eqref{eq:alphaik} holds after we increase~$\kappa$ so that $n^{\kappa\ve}\geq K_0$. For~\eqref{eq:hypsM}, the lower bound in~\eqref{eq:alphaik} shows that every coefficient of~$\sM_0$ of length at least $\max(K_0,\lceil n^{\kappa\ve}\rceil)$ is at least $cn^{-C\ve}$ times the corresponding coefficient of~$\sR_{\bar n_0}$. The lower bound on~$R$ also ensures that $\max(K_0,\lceil n^{\kappa\ve}\rceil)\leq\bar n_0/4$. The truncated coercivity bound~\eqref{eq:truncated coercivity} therefore gives
\begin{equation*}
    \sM_0
    \geq cn^{-C\ve}\sR_{\bar n_0}^{[\max(K_0,\lceil n^{\kappa\ve}\rceil)]}
    \geq cn^{-C\ve}\max(K_0,\lceil n^{\kappa\ve}\rceil)^{-1-s}I_{\bar n_0}
    \geq cn^{-C\ve}I_{\bar n_0}
    \,.
\end{equation*}

To verify~\eqref{eq:tildeA}, let $E^{[K_0]}=\sM-\sR_{\bn}^{[K_0]}$. By~\eqref{eq:blockHN} and~\eqref{eq:defsM'}, this matrix equals $A_0-A$ on $I\times I$ and vanishes elsewhere, so~\eqref{eq:tildeA} bounds all its entries. Since $|J|\leq\bar n/4$, inclusion--exclusion shows that, for every ${1\leq i\leq \ell\leq n_0}$, the terms of~$E^{[K_0]}$ from the intervals contained in~$J$ sum to
\begin{equation*}
    E^{[K_0]}_{i\ell}-E^{[K_0]}_{0\ell}-E^{[K_0]}_{i,n_0+1}+E^{[K_0]}_{0,n_0+1}
    \,.
\end{equation*}
Each distance in this formula is at least~$\ell-i$, so~\eqref{eq:tildeA} bounds the sum by $C_n(1+\ell-i)^{-1-s/2}$. If $\bar n_0=\bar n$, then the reference period does not change. If $\bar n_0<\bar n$, then changing the reference period adds at most $Cn_0^2\bar n_0^{-s-2}$ to an entry. In that case $\bar n_0\geq n_0^{1/2+1/s}$, and, since $2/s\geq1$,
\begin{equation*}
    n_0^2\bar n_0^{-s-2}\leq n_0^{-s/2-2/s}\leq n_0^{-1-s/2}
    \,.
\end{equation*}
This proves~\eqref{eq:tildeA} for~$\sM_0$. Below, we apply~\eqref{eq:as alpha}, \eqref{eq:as gamma}, Lemma~\ref{lemma:global bis} and Lemma~\ref{lemma:BDC} to the system on~$\bar J$, with $n_0$, $\bn$, $J$, $\sM_0$ and $\ve\log n/\log n_0$ in place of $n$, $N$, $I$, $\sM$ and~$\ve$, so that $n_0^{\ve\log n/\log n_0}=n^\ve$. Since $n_0\leq\bar n\leq n^{3/2+1/s}$, this exponent is at least $\ve/(3/2+1/s)$, so the induction hypothesis applies with $\ve_{\min}/(3/2+1/s)$ in place of~$\ve_{\min}$.

The same identity controls~$E_\partial$. Its entries are the terms of~$E^{[K_0]}$ from the intervals not contained in~$J$, plus the corresponding terms of the difference between the two reference matrices $\sR_{\bn}^{[K_0]}$ and~$\sR_{\bar n_0}^{[K_0]}$. For $i\leq \ell$, the first part equals $E^{[K_0]}_{0\ell}+E^{[K_0]}_{i,n_0+1}-E^{[K_0]}_{0,n_0+1}$. The second part vanishes if $\bar n_0=\bar n$, and otherwise has entries bounded by $C\bar n_0^{-s}+Cn_0^2\bar n_0^{-s-2}$. We sum the squares of these bounds over $i\leq \ell$, and reflect for $i\geq \ell$. This gives
\begin{equation*}
    \sum_{i\in J}|(E_\partial)_{\ell i}|^2
    \leq C_n\bigl(\ell^{-1-s}+(n_0+1-\ell)^{-1-s}\bigr)
    +C\mathds1_{\bar n_0<\bar n}n_0\bigl(\bar n_0^{-s}+n_0^2\bar n_0^{-s-2}\bigr)^2
    \,.
\end{equation*}
If $\bar n_0<\bar n$, then $\bar n_0^{2s}\geq n_0^{2+s}$, and the bound on~$n_0^2\bar n_0^{-s-2}$ above gives
\begin{equation*}
    n_0\bar n_0^{-2s}\leq n_0^{-1-s},
    \qquad
    n_0\bigl(n_0^2\bar n_0^{-s-2}\bigr)^2\leq n_0^{-1-s}
    \,.
\end{equation*}
Let~$\delta_\ell$ be computed in the relabeled window~$J$, with $n_0$ in place of~$n$. Since $\delta_\ell\leq n_0$, these bounds give
\begin{equation}\label{eq:boot boundary row}
    \Bigl(\sum_{i\in J}|(E_\partial)_{\ell i}|^2\Bigr)^{1/2}
    \leq C_n\delta_\ell^{-(1+s)/2}
    \,.
\end{equation}

Let~$\dV$ vanish outside~$J$, with
\begin{equation*}
    \dV|_J=-\beta U
    -\beta B_0(D_0+\beta^{-1}\mc L^\nu\otimes I_{\bar n_0-n_0})^{-1}B_0^T\psi^J
    +\lambda\mathbf1_J
    \,.
\end{equation*}
The second term of~$\dV|_J$ is the Schur term. Let~$\psi^0$ solve
\begin{equation*}
    (\beta\sM_0+\mc L^\nu)\psi^0
    =\dV-\beta\begin{pmatrix}E_\partial\psi^J\\0\end{pmatrix}
    \,.
\end{equation*}
By the projection of~\eqref{eq:per2n} on~$J$, the vector equal to~$\psi^J$ on~$J$ and to $-(D_0+\beta^{-1}\mc L^\nu\otimes I_{\bar n_0-n_0})^{-1}B_0^T\psi^J$ outside~$J$ also solves this equation, since~$\sM^J$ is the sum of~$E_\partial$ and the block of~$\sM_0$ on~$J\times J$. By uniqueness, $\psi^0=\psi^J$ on~$J$.
We treat the boundary source separately. Let~$z$ solve
\begin{equation*}
    (\beta\sM_0+\mc L^\nu)z=-\beta\begin{pmatrix}E_\partial\psi^J\\0\end{pmatrix}
    \,.
\end{equation*}
By~\eqref{eq:glob opt}, $\|\psi^J\|_2\leq C_nR^{s-3/2}$. By~\eqref{eq:boot boundary row} and the Cauchy--Schwarz inequality, the $\ell$-th coordinate of the source of~$z$ has norm at most $C_nR^{s-3/2}\delta_\ell^{-(1+s)/2}$. We apply Lemma~\ref{lemma:global bis} to each coordinate of the source. For $\alpha\in\{3/2-s,1-s/2\}$, we split the sum into $\delta_\ell\geq cn_0$ and its complement, on which $d_{\bar n_0}(\ell,j)\geq cn_0$:
\begin{equation*}
    \sum_{\ell\in J}\frac{\delta_\ell^{-(1+s)/2}}{(1+d_{\bar n_0}(\ell,j))^\alpha}
    \leq C\Bigl[R^{-(1+s)/2}\sum_{h=0}^{CR}(1+h)^{-\alpha}
    +R^{-\alpha}\sum_{h=1}^{CR}h^{-(1+s)/2}\Bigr]
    \,.
\end{equation*}
Multiplying by $C_nR^{s-3/2}$ gives
\begin{equation*}
    \|z_j\|_2\leq C_n\log(2R)\bigl(R^{s/2-2}+R^{3s/2-5/2}\bigr),
    \qquad
    \|\nabla z_j\|_2\leq C_nR^{s-2}
    \,.
\end{equation*}
Since $0<s<1$, these imply
\begin{equation}\label{eq:boot boundary solution}
    R^{2-s}\|z_j\|_2+R^{(3-s)/2}\|\nabla z_j\|_2\leq C_n
    \,.
\end{equation}
It remains to estimate $\psi^0-z$, whose source is~$\dV$.

\paragraph{\emph{Step 2: decay of the source}}
We now show that, for $1\leq \ell\leq n_0$ in the first bound and $1\leq \ell<n_0$ in the second,
\begin{align}
    {\|\dV_\ell\|_2}
    &{{}\leq C_nR^{s/2-1}\delta_\ell^{-s/2},}
        \label{eq:loc field}\\
    \|\dV_{\ell+1}-\dV_\ell\|_2
    &{{}\leq C_nR^{s/2-1}\min(\delta_\ell,\delta_{\ell+1})^{-1-s/2}}
        \label{eq:loc increments}
    \,.
\end{align}
For~$U$, we use $\sum_i\psi_i=0$ to write
\begin{equation*}
    {U_\ell=\sum_{i\notin J}(\sM_{\ell i}-\sM_{\ell i_0})\psi_i
    -\sM_{\ell i_0}\sum_{i\in J}\psi_i}
    \,.
\end{equation*}
The interval representation~\eqref{eq:alphaik} gives, for every $\ell,i\in\bar I$,
\begin{equation*}
    |\sM_{\ell,i+1}-\sM_{\ell i}|\leq\frac{C_n}{(1+d_{\bar n}(\ell,i))^{1+s}},
    \qquad
    |\sM_{\ell+1,i+1}-\sM_{\ell+1,i}-\sM_{\ell,i+1}+\sM_{\ell i}|\leq\frac{C_n}{(1+d_{\bar n}(\ell,i))^{2+s}}
    \,.
\end{equation*}
Let $\ell\in J$ and $d_{\bar n}(i,i_0)\leq R/2$. The short arc from~$i_0$ to~$i$ is at distance at least $3R/8$ from~$J$. Telescoping along this arc gives
\begin{equation*}
    |\sM_{\ell i}-\sM_{\ell i_0}|\leq\frac{C_nd_{\bar n}(i,i_0)}{R^{1+s}},
    \qquad
    |\sM_{\ell+1,i}-\sM_{\ell+1,i_0}-\sM_{\ell i}+\sM_{\ell i_0}|\leq\frac{C_nd_{\bar n}(i,i_0)}{R^{2+s}}
    \,.
\end{equation*}
We split the sum at $d_{\bar n}(i,i_0)=R/2$. The Cauchy--Schwarz inequality and~\eqref{eq:glob opt} give
\begin{align*}
    {\|U_\ell\|_2}
    &{{}\leq C_n\Bigl[R^{-1}+R^{s-3/2}
        \Bigl(\sum_{r=\delta_\ell}^{CR}r^{-2s}\Bigr)^{1/2}\Bigr],}\\
    \|U_{\ell+1}-U_\ell\|_2
    &{{}\leq C_n\bigl[R^{-2}+R^{s-3/2}\min(\delta_\ell,\delta_{\ell+1})^{-1/2-s}\bigr]}
    \,.
\end{align*}
Since $\sum_{r=\delta}^{CR}r^{-2s}\leq CR^{1-s}\delta^{-s}$ for $\delta\leq CR$, these imply~\eqref{eq:loc field}--\eqref{eq:loc increments} for~$U$. For the Schur term, \eqref{eq:BDC} gives~\eqref{eq:loc field}, and~\eqref{eq:BDC b} gives~\eqref{eq:loc increments}, since
\begin{equation*}
    {\sum_{i\in J}\delta_i^{-s/2}\|\psi_i\|_2
    \leq CR^{(1-s)/2}\|\psi^J\|_2
    \leq C_nR^{s/2-1}}
    \,.
\end{equation*}
Finally, \eqref{eq:f lambda} gives $\|\lambda\|_2\leq C_n/\bar n\leq C_n/R$.

{We extend~$\dV$ by zero to~$\bar J$ and let $(\mathfrak t_h\dV)_i=\dV_{i-h}$. Summing~\eqref{eq:loc field} gives ${\sum_{\ell\in\bar J}\|\dV_\ell\|_2\leq C_n}$. For $1\leq h\leq cR$, we sum~\eqref{eq:loc field} over the two boundary strips of width~$2h$, and telescope~\eqref{eq:loc increments} on their complement. Both contributions are bounded by $C_nR^{s/2-1}h^{1-s/2}$. For larger translations, the left-hand side of~\eqref{eq:loc modulus} is at most $2\sum_{\ell\in\bar J}\|\dV_\ell\|_2$. This gives}
\begin{equation}\label{eq:loc modulus}
    \sum_{i\in\bar J}\|\dV_{i-h}-\dV_i\|_2
    \leq C_n\min\bigl\{1,(d_{\bar n_0}(h,0)/R)^{1-s/2}\bigr\}
    \quad\text{for }h\in\bar J
    \,.
\end{equation}

\paragraph{\emph{Step 3: reference resolvent bounds}}
For $k\in\{0,1\}$, let $K^{(k)}(h)$ be the operator-valued cyclic convolution kernel of $(\mc L^\nu)^{k/2}(\beta\sR_{\bar n_0}+\mc L^\nu)^{-1}$. Then
\begin{equation}\label{eq:loc resolvent kernels}
    \|K^{(k)}(h)\|_{L^2(\nu)\to L^2(\nu)}
    \leq \frac{C(\beta,s)}{(1+d_{\bar n_0}(h,0))^{1+(1-k/2)(1-s)}},
    \qquad
    \Bigl\|\sum_{h\in\bar J}K^{(k)}(h)\Bigr\|_{L^2(\nu)\to L^2(\nu)}
    \leq \frac{C(\beta,s)}{\bar n_0^{(1-k/2)(1-s)}}
    \,.
\end{equation}
Indeed, at nonzero frequencies the Fourier multipliers are $t^{k/2}/(2\beta S(\theta)+t)$, with $t\geq0$ in the spectrum of~$\mc L^\nu$. Equations~\eqref{eq:defS}--\eqref{eq:lowerS} bound their derivatives of order $j\leq2$ by $C(\beta,s)\rho(\theta)^{(1-k/2)(1-s)-j}$, uniformly in~$t$. The first part of Lemma~\ref{lemma:multiplier kernels}, with $r=2$ and $\alpha=(1-k/2)(1-s)$, gives the first bound, up to the contribution of the zero mode. The zero mode is $t^{k/2}/(2\beta\vartheta_{\bar n_0}+t)$. Its norm is at most $C\bar n_0^{-(1-k/2)(1-s)}$, and its contribution to each kernel entry is at most $C\bar n_0^{-1-(1-k/2)(1-s)}$. Spectral calculus proves~\eqref{eq:loc resolvent kernels}.

\paragraph{\emph{Step 4: decay of the reference solution}}
Let $\psi^{\mathrm{ref}}=(\beta\sR_{\bar n_0}+\mc L^\nu)^{-1}\dV$. Subtracting translates of~$\dV$ gives
\begin{equation*}
    {\psi^{\mathrm{ref}}=(2\beta\vartheta_{\bar n_0}+\mc L^\nu)^{-1}\dV
    +\sum_{h\in\bar J}K^{(0)}(h)(\mathfrak t_h\dV-\dV)}
    \,.
\end{equation*}
Since $1-s/2>1-s$ and $\bar n_0\geq cR$, \eqref{eq:loc modulus} and~\eqref{eq:loc resolvent kernels} yield
\begin{equation}\label{eq:loc phi mass}
    \sum_{i\in\bar J}\|\psi_i^{\mathrm{ref}}\|_2
    \leq C_n\Bigl[\bar n_0^{s-1}
    +\sum_{h=1}^{\lfloor\bar n_0/2\rfloor}
        h^{s-2}\min\bigl\{1,(h/R)^{1-s/2}\bigr\}\Bigr]
    \leq C_nR^{s-1}
    \,.
\end{equation}
If $d_{\bar n_0}(i,j)\leq R/32$ and $d_{\bar n_0}(h,0)\leq R/32$, then~\eqref{eq:loc field}--\eqref{eq:loc increments} give
\begin{equation*}
    \|\dV_i\|_2\leq C_n/R
    \quad\text{and}\quad
    \|\dV_{i-h}-\dV_i\|_2\leq C_nd_{\bar n_0}(h,0)/R^2
    \,.
\end{equation*}
We use these bounds in the kernel sum for $(\mc L^\nu)^{k/2}\psi_i^{\mathrm{ref}}$, after subtracting~$\dV_i$, and ${\sum_{\ell\in\bar J}\|\dV_\ell\|_2\leq C_n}$ for the remaining indices. Thus
\begin{equation*}
    {\|(\mc L^\nu)^{k/2}\psi_i^{\mathrm{ref}}\|_2
    \leq C_n\Bigl[\frac{R^{-1}}{\bar n_0^{(1-k/2)(1-s)}}
        +R^{-2}\sum_{h=1}^{\lfloor R/32\rfloor}h^{-(1-k/2)(1-s)}
        +\frac{R^{-1}}{R^{(1-k/2)(1-s)}}\Bigr]
    \leq \frac{C_n}{R^{1+(1-k/2)(1-s)}}}
    \,.
\end{equation*}
Since $\|(\mc L^\nu)^{1/2}u\|_2=\|\nabla u\|_2$, the cases $k=0$ and $k=1$ give, for $d_{\bar n_0}(i,j)\leq R/32$,
\begin{equation}\label{eq:finalphi}
    \|\psi_i^{\mathrm{ref}}\|_2\leq C_nR^{s-2},
    \qquad
    \|\nabla\psi_i^{\mathrm{ref}}\|_2\leq C_nR^{-(3-s)/2}
    \,.
\end{equation}

\paragraph{\emph{Step 5: the comparison error}}
Let $E_0=\sM_0-\sR_{\bar n_0}$ and $\psi^{\mathrm{err}}=\psi^0-z-\psi^{\mathrm{ref}}$. Then
\begin{equation}\label{eq:loc error}
    (\beta\sM_0+\mc L^\nu)\psi^{\mathrm{err}}=-\beta E_0\psi^{\mathrm{ref}}
    \,.
\end{equation}
The difference $\sR_{\bar n_0}-\sR_{\bar n_0}^{[K_0]}$ is supported on cyclic distances less than~$K_0$ and has operator norm at most $C(s)K_0^{1-s}\leq Cn^{C\ve}$. In particular, its entries are at most $Cn^{C\ve}$. Consequently, Step~1 gives
\begin{equation*}
    |(E_0)_{i\ell}|\leq C_n(1+d_{\bar n_0}(i,\ell))^{-1-s/2}
    \,.
\end{equation*}
The absolute row and column sums of~$E_0$ are therefore bounded by~$C_n$. With~\eqref{eq:loc phi mass}, this gives ${\sum_{i\in\bar J}\|(E_0\psi^{\mathrm{ref}})_i\|_2\leq C_nR^{s-1}}$. For $d_{\bar n_0}(i,j)\leq R/64$, we split $(E_0\psi^{\mathrm{ref}})_i$ at $d_{\bar n_0}(\ell,j)=R/32$ and use~\eqref{eq:finalphi} on the central part. The other part is bounded by $C_nR^{-1-s/2}\sum_{\ell\in\bar J}\|\psi_\ell^{\mathrm{ref}}\|_2$. Since $1+s/2+(1-s)>2-s$, we obtain
\begin{equation*}
    {\|(E_0\psi^{\mathrm{ref}})_i\|_2\leq C_nR^{s-2}
    \quad\text{if }d_{\bar n_0}(i,j)\leq R/64}
    \,.
\end{equation*}
We apply~\eqref{eq:as alpha}--\eqref{eq:as gamma} to~\eqref{eq:loc error}, coordinate by coordinate of the source. For $\alpha\in\{p,q\}$, splitting the source at distance~$R/64$ from~$j$ gives
\begin{equation*}
    {\sum_{i\in\bar J}\frac{\|(E_0\psi^{\mathrm{ref}})_i\|_2}{(1+d_{\bar n_0}(i,j))^\alpha}
    \leq C_n\Bigl[R^{s-2}\sum_{h=0}^{\lfloor R/64\rfloor}(1+h)^{-\alpha}
        +R^{s-1-\alpha}\Bigr]
    \leq C_n\bigl(R^{s-2}+R^{s-1-\alpha}\bigr)}
    \,,
\end{equation*}
where the logarithm at $\alpha=1$ is absorbed into~$C_n$. We conclude that
\begin{equation}\label{eq:loc error value}
    \|\psi_j^{\mathrm{err}}\|_2\leq C_n\bigl(R^{s-2}+R^{s-1-p}\bigr),
    \qquad
    \|\nabla\psi_j^{\mathrm{err}}\|_2\leq C_n\bigl(R^{s-2}+R^{s-1-q}\bigr)
    \,.
\end{equation}

\paragraph{\emph{Step 6: closing the constrained induction step}}
Since $\psi_j=z_j+\psi_j^{\mathrm{ref}}+\psi_j^{\mathrm{err}}$, \eqref{eq:boot boundary solution}, \eqref{eq:finalphi} and~\eqref{eq:loc error value} give
\begin{equation*}
    {\|\psi_j\|_2\leq C_n\bigl(R^{s-2}+R^{s-1-p}\bigr),
    \qquad
    \|\nabla\psi_j\|_2\leq C_n\bigl(R^{-(3-s)/2}+R^{s-1-q}\bigr)}
    \,.
\end{equation*}
Thus the exponents improve to $\min\{2-s,p+1-s\}$ and $\min\{(3-s)/2,q+1-s\}$.

\paragraph{\emph{Step 7: transfer to the unconstrained equation}}
Let~$v$ and~$\lambda$ solve~\eqref{eq:per2n}, with~$v$ in place of~$\psi$, and let~$f$ solve $(2\beta\vartheta_{\bn}+\mc L^\nu)f=-\lambda$. As in the proof of Lemma~\ref{lemma:global bis}, $w^{(2)}=\psi-v-f\mathbf1$ solves $(\beta\sM+\mc L^\nu)w^{(2)}=-\beta fE\mathbf1$, and testing with~$f$ and~\eqref{eq:f lambda} give $\|f\|_2\leq C_n\bn^{s-2}$ and $\|\nabla f\|_2\leq C_n\bn^{-(3-s)/2}$. Step~6 applies to~$v$ and gives
\begin{equation}\label{eq:s7v}
    \|v_j\|_2\leq C_n\bigl(r_j^{s-1-p}+r_j^{s-2}\bigr),
    \qquad
    \|\nabla v_j\|_2\leq C_n\bigl(r_j^{s-1-q}+r_j^{-(3-s)/2}\bigr)
    \,.
\end{equation}
By the bound on~$\|f\|_2$ and the bound on the absolute row sums of~$E=\sM-\sR_{\bn}$, defined after~\eqref{eq:Lalpha5}, the coordinates of the source $-\beta fE\mathbf1$ of this equation have $L^2(\nu)$ norms at most $C_n\bn^{s-2}$. We apply~\eqref{eq:as alpha}--\eqref{eq:as gamma} to each of these point sources and sum by linearity. For every $\alpha>0$,
\begin{equation*}
    {\sum_{i=1}^{\bn}(1+d_{\bn}(i,j))^{-\alpha}
    \leq C(\alpha)\log(2\bn)\bigl(1+\bn^{\max\{1-\alpha,0\}}\bigr)}
    \,.
\end{equation*}
After absorbing the logarithm into~$C_n$, this gives
\begin{equation*}
    {\|w_j^{(2)}\|_2\leq C_n\bigl(\bn^{s-2}+\bn^{s-1-p}\bigr),
    \qquad
    \|\nabla w_j^{(2)}\|_2\leq C_n\bigl(\bn^{s-2}+\bn^{s-1-q}\bigr)}
    \,.
\end{equation*}
The constant part $f\mathbf1$ satisfies the bounds above. Since $r_j\leq\bn$ and $2-s>(3-s)/2$, these bounds and~\eqref{eq:s7v} give
\begin{equation*}
    {\|\psi_j\|_2\leq C_n\bigl(r_j^{s-1-p}+r_j^{s-2}\bigr),
    \qquad
    \|\nabla\psi_j\|_2\leq C_n\bigl(r_j^{s-1-q}+r_j^{-(3-s)/2}\bigr)}
    \,.
\end{equation*}
This completes the induction step for both equations. Starting from $p=3/2-s$ and $q=1-s/2$, after $\lceil(2(1-s))^{-1}\rceil$ iterations, we reach $p=2-s$ and $q=(3-s)/2$. Taking~$\ve_{\min}=\ve$ proves~\eqref{eq:final control} and~\eqref{eq:final control bis}.
\end{proof}

\begin{remark}
Even though the Lagrange multiplier in~\eqref{eq:per2n} is of order $1/\bn$, there is no correction of order $1/\bn$ in~\eqref{eq:final control}, unlike the case $s>1$. This is related to the fact that the vector ${\mathbb G_{\bn,s}^{-1}\mathbf1}$ has entries of order~$\bn^{s-1}$. In the proof, the multiplier is included in~$\dV$, whose reference solution satisfies~\eqref{eq:finalphi}.
\end{remark}

\subsection{Decay for the marginal measure}
For $n\leq N/2$, we now deduce from Proposition~\ref{proposition:loc} a bound on the solution of~\eqref{eq:eqqq}. The following lemma bounds the pairing of an operator that satisfies the bounds of Lemma~\ref{lemma:cond variance} with a field that decays like the solution of Proposition~\ref{proposition:loc}. We use it in the proof of Proposition~\ref{prop:main2} and in Section~\ref{section:change measure}.

\begin{lemma}\label{lemma:column subtraction}
Let $0<s<1$ and $n\geq2$, and let $\delta_k=\min(k,n+1-k)$ for $k\in I=\{1,\ldots,n\}$. Let~$\nu$ be a probability measure, let $c,\delta>0$, and let~$\tM$ be a linear map from $L^\infty(\nu;\dR^n)$ to $L^1(\nu;\dR^n)$ such that, up to an additive error $C_ne^{-cn^\delta}\|\phi_1\|_{L^\infty(\nu)}\|\phi_2\|_{L^\infty(\nu)}$, for all $\phi_1,\phi_2\in L^\infty(\nu)$ and $i,j,\ell\in I$,
\begin{equation*}
    |\dE_\nu[\phi_1e_i\cdot\tM(\phi_2e_j)]|
    \leq\frac{C_n\|\phi_1\|_2\|\phi_2\|_2}{\delta_i^{s/2}\delta_j^{s/2}},
    \qquad
    |\dE_\nu[\phi_1e_i\cdot\tM(\phi_2(e_j-e_\ell))]|
    \leq\frac{C_n|j-\ell|\|\phi_1\|_2\|\phi_2\|_2}{\delta_i^{s/2}\min(\delta_j,\delta_\ell)^{1+s/2}}
    \,.
\end{equation*}
Let $i_0\in I$ with $\delta_{i_0}\geq n/4$, and let $v\in L^\infty(\nu;\dR^n)$ satisfy, for every $k\in I$,
\begin{equation*}
    \|v_k\|_2\leq\frac{C_n}{(1+|k-i_0|)^{2-s}},
    \qquad
    \|\mathbf1_I\cdot v\|_2\leq C_nn^{s-1}
    \,.
\end{equation*}
Then there exists $C>0$ such that, for every $u\in L^\infty(\nu;\dR^n)$,
\begin{equation*}
    |\dE_\nu[u\cdot\tM v]|
    \leq C_nn^{-1/2}\|u\|_2+Ce^{-cn^\delta/2}\|u\|_{L^\infty(\nu)}\|v\|_{L^\infty(\nu)}
    \,.
\end{equation*}
\end{lemma}

\begin{proof}
By linearity, the operator~$\tM$ satisfies
\begin{equation}\label{eq:b1 H}
    \tM v
    =\sum_{k\in I}\tM\bigl((e_k-e_{i_0})v_k\bigr)
        +\tM\bigl(e_{i_0}\,\mathbf1_I\cdot v\bigr)
    \,.
\end{equation}
Fix $j\in I$, and take $\phi_1=u_j$ in the two bounds. If $\delta_k\leq n/8$, then $|k-i_0|\geq n/8$. For these~$k$, we apply the first bound to the two columns separately. The sum of their source norms satisfies
\begin{equation*}
    {\sum_{\delta_k\leq n/8}
        (\delta_k^{-s/2}+\delta_{i_0}^{-s/2})\|v_k\|_2
    \leq C_nn^{s-2}\sum_{\delta_k\leq n/8}\delta_k^{-s/2}
    \leq C_nn^{s/2-1}}
    \,.
\end{equation*}
If $\delta_k>n/8$, then $\min(\delta_k,\delta_{i_0})\geq n/8$. For these~$k$, we apply the second bound and sum
\begin{equation*}
    {\sum_{\delta_k>n/8}
        \frac{|k-i_0|\,\|v_k\|_2}
        {\min(\delta_k,\delta_{i_0})^{1+s/2}}
    \leq C_nn^{-1-s/2}\sum_{r=0}^n\frac{r}{(1+r)^{2-s}}
    \leq C_nn^{s/2-1}}
    \,.
\end{equation*}
The last term of~\eqref{eq:b1 H} has the same bound by the first bound, since $\delta_{i_0}\geq n/4$. Each of the at most $2n+1$ terms adds an error at most $C_ne^{-cn^\delta}\sqrt n\|u\|_{L^\infty(\nu)}\|v\|_{L^\infty(\nu)}$, since $|\mathbf1_I\cdot v|\leq\sqrt n|v|$. Hence
\begin{equation*}
    \bigl|\dE_\nu[u_je_j\cdot\tM v]\bigr|
    \leq C_nn^{s/2-1}\delta_j^{-s/2}\|u_j\|_2
        +C_nn^{3/2}e^{-cn^\delta}\|u\|_{L^\infty(\nu)}\|v\|_{L^\infty(\nu)}
    \,.
\end{equation*}
Since $\sum_{j\in I}\delta_j^{-s}\leq Cn^{1-s}$, summing over~$j$ and applying the Cauchy--Schwarz inequality proves the lemma.
\end{proof}

\begin{proposition}\label{prop:main2}
Let $s\in (0,1)$ and~$\ve>0$. Let~$n\leq N/2$ and $I=\{1,\ldots,n\}$, and let~$\nu$ be given by~\eqref{eq:def nu hypersing} with $\iota=0$. Let~$i_0$ be an integer with~$\lceil n/4\rceil\leq i_0\leq\lfloor3n/4\rfloor$. Let $\chi_n\in H^1(\nu)$, and let $\psi\in H^1(\nu;\dR^I)$ solve
\begin{equation}\label{eq:main1}
    A_1^{\nu}\psi=\chi_ne_{i_0}
    \,.
\end{equation}
Then there exist constants $C(\beta)>0$, $c(\beta)>0$, $\delta>0$ and $\kappa>0$ such that, for every $1\leq j\leq n$,
\begin{equation}\label{eq:main1r}
   \|\psi_j\|_2\leq C(\beta)n^{\kappa\ve}\Bigl(\frac{1}{(1+|j-i_0|)^{2-s}}+\frac{1}{\sqrt{n}}\Bigr)\bigl(\|\chi_n\|_2+e^{-c(\beta)n^{\delta}}\|\chi_n\|_\infty\bigr)
   \,.
\end{equation}
\end{proposition}

\begin{proof}
We may assume that $\ve\leq1/[4(\mathfrak m+1)]$, so that the first item of Lemma~\ref{lemma:ident} applies. Indeed, Lemma~\ref{lemma:the convexity lemma}, applied to the nearest-pair terms for $i\in I$ and the penalty~$\FF^\g$, gives~$\nabla^2\widetilde{\Hc}_n^\g\geq c(s)n^{-(s+2)\ve}I_n$. Testing~\eqref{eq:main1} with~$\psi$ then gives~$\|\psi\|_2\leq C\beta^{-1}n^{(s+2)\ve}\|\chi_n\|_2$. If $\ve>1/[4(\mathfrak m+1)]$ and $\kappa\geq s+2+2(\mathfrak m+1)$, then $(\kappa-s-2)\ve\geq1/2$, and this bound implies~\eqref{eq:main1r}. By linearity, we may also assume that $\|\chi_n\|_2=1$.

Let $\psi^{(1)}$ solve~\eqref{eq:equ psi1}. We use~\eqref{eq:ifp} for $\psi^{(2)}=\psi^{(1)}-\psi$.

We first control the sum of the coordinates of~$\psi^{(1)}$. Let $\Psi$ be the solution of~\eqref{eq:equ 54} on~$\bar I$, so that $\Psi_k=\psi^{(1)}_k$ for $k\in I$ by Lemma~\ref{lemma:ident}. Proposition~\ref{proposition:loc}, applied to~$\Psi$, gives, for every $k\in\bar I$,
\begin{equation}\label{eq:comparison lifted decay}
    \|\Psi_k\|_2
    \leq \frac{C_n}{(1+d_{\bn}(k,i_0))^{2-s}},
    \qquad
    \sum_{k\in\bar I}\|\Psi_k\|_2\leq C_n
    \,.
\end{equation}
The absolute column sums of~$E=\sM-\sR_{\bn}$, defined after~\eqref{eq:Lalpha5}, are bounded by~$C_n$. Summing~\eqref{eq:equ 54} over the coordinates gives
\begin{equation*}
    {(2\beta\vartheta_{\bn}+\mc L^\nu)(\mathbf1\cdot\Psi)
    =\chi_n-\beta\mathbf1\cdot E\Psi}
    \,.
\end{equation*}
Thus $\|\mathbf1\cdot\Psi\|_2\leq C_n\bn^{s-1}$. Since $\bn\geq2n$ and $i_0$ lies in the middle half of~$I$, every $k\in\bar I\setminus I$ satisfies $d_{\bn}(k,i_0)\geq n/4$. The coordinates of~$\psi^{(1)}$ are the first~$n$ coordinates of~$\Psi$, so~\eqref{eq:comparison lifted decay} gives
\begin{equation}\label{eq:comparison interior mean}
    \|\mathbf1_I\cdot\psi^{(1)}\|_2
    \leq\|\mathbf1\cdot\Psi\|_2+\sum_{k\in\bar I\setminus I}\|\Psi_k\|_2
    \leq C_nn^{s-1}
    \,.
\end{equation}

To control the exponential remainders, we apply the maximum principle of Proposition~\ref{proposition:maximum principle} to~\eqref{eq:equ 54} and~\eqref{eq:main1}. It gives
\begin{equation*}
    \|\Psi\|_{L^\infty(\nu)}+\|\psi\|_{L^\infty(\nu)}\leq C_n\|\chi_n\|_\infty
    \,.
\end{equation*}
By Lemma~\ref{lemma:cond variance}, the operator~$\tM$ of~\eqref{eq:tH5} satisfies the hypotheses of Lemma~\ref{lemma:column subtraction}. By~\eqref{eq:comparison lifted decay} and~\eqref{eq:comparison interior mean}, so does the field $v=\psi^{(1)}$. That lemma, with $u=\psi^{(2)}$, gives
\begin{equation*}
    {\bigl|\dE_\nu[\psi^{(2)}\cdot\tM\psi^{(1)}]\bigr|
    \leq C_nn^{-1/2}\|\psi^{(2)}\|_2
        +Ce^{-cn^\delta}\|\chi_n\|_\infty^2}
    \,.
\end{equation*}
We insert this into~\eqref{eq:ifp} and use Young's inequality. After decreasing~$c$, we obtain
\begin{equation}\label{eq:fpsi}
    \|\psi^{(2)}\|_2^2
    \leq \frac{C_n}{n}\bigl(1+e^{-cn^\delta}\|\chi_n\|_\infty^2\bigr)
    \,.
\end{equation}
Combining~\eqref{eq:fpsi} with~\eqref{eq:comparison lifted decay} on~$I$ proves~\eqref{eq:main1r}.
\end{proof}

\subsection{\texorpdfstring{Proof of Theorem~\ref{theorem:decay gap correlations}}{Proof of Theorem \getrefnumber{theorem:decay gap correlations}}}
We now conclude the proof of Theorem~\ref{theorem:decay gap correlations} as in Subsection~\ref{sub:hyper decay}. When the gaps~$i$ and~$j$ are at distance larger than a small power of~$N$, we take $n=N$ and use Proposition~\ref{proposition:loc}. Otherwise, we take~$n$ to be a power of~$d_N(i,j)$ and use Proposition~\ref{prop:main2}. For the bound~\eqref{eq:l corr}, we use the following consequence of~\cite[Corollary~1]{boursier2021optimal}.

\begin{lemma}\label{lemma:block variance}
Let $0<s<1$ and $\beta>0$, and let $\sigma^2=\cot(\pi s/2)/(2\beta\pi s)$. Then, under~$\dGi$,
\begin{equation*}
    K^{-s/2}\bigl(N(x_{K+1}-x_1)-K\bigr)\mathrel{\underset{\substack{K\to\infty\\ K/N\to0}}{\Longrightarrow}}\mathcal N(0,\sigma^2)
\end{equation*}
and
\begin{equation}\label{eq:mesoscopic gap variance}
    \Var_{\dGi}[N(x_{K+1}-x_1)]=\sigma^2K^s+o(K^s)
    \qquad(K\to\infty,\ K/N\to0)
    \,.
\end{equation}
\end{lemma}

\begin{proof}
Let~$K\to\infty$ with~$K/N\to0$, so that eventually~$1\leq K\leq N/2$. The second assertion of~\cite[Corollary~1]{boursier2021optimal} gives
\begin{equation*}
    \frac{N(x_{K+1}-x_1)-K}{\sqrt{N^sV_s(\frac{K}{N})}}\mathrel{\underset{\substack{K\to\infty\\ K/N\to0}}{\Longrightarrow}}\mathcal N(0,\sigma^2)\,,
\end{equation*}
where, by~\cite[Lemma~2.3]{boursier2021optimal}, the function~$V_s$ is positive on~$(0,1)$ and satisfies~$V_s(t)=t^s+O(t^2)$ as~$t\downarrow0$. Hence~$N^sV_s(K/N)=K^s(1+o(1))$, and the normalized span~$K^{-s/2}(N({x_{K+1}}-x_1)-K)$ has the same Gaussian limit. The gap law under~$\dGi$ is log-concave, and the span is a nonconstant linear function of the gaps. The normalized span therefore has a log-concave density, and~\cite[Proposition~2(c)]{cule2010theoretical} shows that its first two moments converge to those of~$\mathcal N(0,\sigma^2)$. This proves~\eqref{eq:mesoscopic gap variance}.
\end{proof}

\begin{proof}[Proof of Theorem~\ref{theorem:decay gap correlations}]
Let $r=1+d_N(i,j)$, and fix $p>2(2-s)$ and $\delta_0>0$ such that $p\delta_0<1$. By approximation, we may assume that~$\xi$ and~$\chi$ are smooth. If $r\leq N^{\delta_0}$, let~$n$ be an integer with $n\asymp r^p$. Since $p\delta_0<1$, we have $n\leq N/2$ for~$N$ large. By invariance of the gap law under cyclic shifts of the labels, we may assume that~$i$ and~$j$ lie in the middle half of the window $I=\{1,\ldots,n\}$. If $r>N^{\delta_0}$, let $n=N$. In both cases, let~$\nu$ be given by~\eqref{eq:def nu hypersing} with $\iota=0$.

Since Lemma~\ref{lemma:single gap tails} holds for $0<s<1$, the proof of~\eqref{eq:tilt transfer} applies and gives
\begin{equation*}
    \Cov_{\dGi^\g}[\xi(x_i),\chi(x_j)]
    =\Cov_\nu[\xi(x_i),\chi(x_j)]+O\left(\|\xi'\|_\infty\|\chi'\|_\infty e^{-cn^\delta}\right)
    \,.
\end{equation*}
The same argument compares the second moments of~$\xi'$ and~$\chi'$ under~$\dGi^\g$ and~$\nu$. By the Helffer--Sj\"ostrand representation of Proposition~\ref{proposition: HS gaps c}, which the end of Appendix~\ref{section:existence c} extends to~$\nu$, and by polarization,
\begin{equation*}
    \Cov_\nu[\xi(x_i),\chi(x_j)]=\dE_\nu[\xi'(x_i)\psi_i]
    \,,
\end{equation*}
where~$\psi$ solves~\eqref{eq:main1} with source~$\chi'(x_j)e_j$ if $n\leq N/2$, and~\eqref{eq:billus 2} with $M=\nabla^2\widetilde\Hc_N^\g$ and source~$\chi'(x_j)e_j$ if $n=N$. By the Cauchy--Schwarz inequality, it remains to bound~$\|\psi_i\|_2$.

If $n\leq N/2$, then $|i-j|=d_N(i,j)$, and Proposition~\ref{prop:main2} gives
\begin{equation*}
    \|\psi_i\|_2\leq C(\beta)n^{\kappa\ve}\left(r^{-(2-s)}+n^{-1/2}\right)\left(\|\chi'(x_j)\|_2+e^{-cn^\delta}\|\chi'\|_\infty\right)
    \,.
\end{equation*}
Since $p>2(2-s)$, the term~$n^{-1/2}$ is at most~$Cr^{-(2-s)}$.

If $n=N$, let~$\mc A$ be the good event~\eqref{eq:good5 gap} for $n=N$. There are no exterior blocks and $P_\partial=Q=0$, so the matrix~$A_0$ of~\eqref{eq:A0 e} equals~$\nabla^2\widetilde\Hc_N^\g$ on~$\mc A$ and~$A$ on~$\mc A^c$. Let~$u$ solve~\eqref{eq:per2n} with $\sM=A_0$ and source~$\chi'(x_j)e_j$. As in the proof of Proposition~\ref{prop:main2}, we may assume that $\ve\leq1/[4(\mathfrak m+1)]$. Lemma~\ref{lemma:ident} then shows that~$A_0$ satisfies Assumptions~\ref{assumptions:sM}, and Proposition~\ref{proposition:loc} gives
\begin{equation*}
    \|u_i\|_2\leq C(\beta)N^{\kappa\ve}r^{-(2-s)}\|\chi'(x_j)\|_2
    \,.
\end{equation*}
The maximum principle of Proposition~\ref{proposition:maximum principle}, applied in an orthonormal basis of~$\mathbf1^\perp$, gives $\|u\|_{L^\infty(\nu)}\leq C_N\|\chi'\|_\infty$. We subtract the equations of~$\psi$ and~$u$ and test on the tangent space, where both multipliers disappear. The difference of the matrices is $\nabla^2\widetilde\Hc_N^\g-A$ on~$\mc A^c$, and zero on~$\mc A$. The moment bounds of Lemma~\ref{lemma:good event long} and the lower bound $\nabla^2\widetilde\Hc_N^\g\geq c(s)N^{-(s+2)\ve}I_N$ therefore give $\|\psi-u\|_2\leq Ce^{-cN^\delta}\|\chi'\|_\infty$.

In both cases, the loss is at most $n^{\kappa\ve}\leq Cr^{\kappa\ve/\delta_0}$, since $n\leq\max(Cr^p,r^{1/\delta_0})$ and $p\delta_0<1$. Replacing~$\ve$ by~$\delta_0\ve/\kappa$ in the penalty turns it into a factor~$r^\ve$. Since $n\geq r$, the remainders are at most~$e^{-cr^\delta}$. Inserting these bounds into the first display of the proof proves~\eqref{eq:decay corr stat}.

For the lower bound, let $V_N(k)=\Var_{\dGi}[N(x_{k+1}-x_1)]$, and let
$m=\lfloor n_N/4\rfloor$. Cyclic invariance gives
\begin{equation*}
    {\Cov_{\dGi}[N(x_{m+1}-x_1),N(x_{n_N+1}-x_{n_N-m+1})]}
    =\frac12\bigl(V_N(n_N)+V_N(n_N-2m)-2V_N(n_N-m)\bigr)
    \,.
\end{equation*}
By Lemma~\ref{lemma:block variance}, the right-hand side divided by $n_N^s$
converges to
\begin{equation*}
    \frac{\sigma^2}{2}
      \left(1+2^{-s}-2(3/4)^s\right)<0
    \,,
\end{equation*}
since $t\mapsto t^s$ is strictly concave. By bilinearity, the left-hand side is the sum of the $m^2$ covariances of $N(x_{a+1}-x_a)$ and $N(x_{b+1}-x_b)$ with $1\leq a\leq m$ and $n_N-m<b\leq n_N$, and these labels satisfy $n_N/2\leq d_N(a,b)\leq n_N$. At least one of these covariances is therefore at most $-cn_N^{s-2}$, which proves~\eqref{eq:l corr}.

\end{proof}

\section{Uniqueness of the limiting measure}\label{section:change measure}

In this section, we prove that the laws~\eqref{eq:def Qx} converge locally in total variation, as claimed in Theorem~\ref{theorem:convergence}. The gap laws are tight by Lemma~\ref{lemma:single gap tails}, so the theorem is a uniqueness statement. We show that the gap laws form a Cauchy sequence as $N\to\infty$. The main step is Proposition~\ref{proposition:distinct}. For $N'\geq N\geq\max(4n,n^{2/s})$, replacing $N$ particles by~$N'$ changes the expectation of a Lipschitz function of the gaps in the middle of a window of~$n$ labels by at most $C(\beta)n^{\kappa\ve-\min(s,1)/2}$ times the sum of the suprema of its partial derivatives, up to an exponentially small error. To prove it, we condition on the gaps outside the window. For two exterior configurations that give the window the same gap sum, we interpolate between their energies (Lemma~\ref{lemma:integrated repre}) and bound the covariance along the path by the decay of correlations (Lemma~\ref{lemma:re eq}). Lemma~\ref{lemma:length averaging} shows that the conditional expectation varies little with this gap sum, so that we can average over a common law of the gap sum. Subsection~\ref{sub:proof conv} proves Proposition~\ref{proposition:distinct} and deduces Theorems~\ref{theorem:convergence} and~\ref{theorem:quantitative conv}. The last two subsections prove Corollary~\ref{theorem:hyper riesz} and Proposition~\ref{proposition:replusion}.

\subsection{Reduction to a correlation estimate}
We compare the following two quantities:
\begin{equation*}
    \dE_{\dGi^\g}[{G}(x_1,\ldots,x_n)]\quad \text{and}\quad \dE_{\dGip^\g}[{G}(x_1,\ldots,x_n)],\quad \text{with}\quad G:\dR^n\to\dR \quad \text{smooth}\,,
\end{equation*}
where $1\leq n\leq N\leq N'$. Let $I=\{1,\ldots,n\}$, and let $\pi:\mc{M}_N\to\pi(\mc{M}_N)$ be the projection on the coordinates $(x_1,\ldots,x_n)$, and let~$\pi_{I^c}$ be the projection on the other coordinates. We draw an exterior configuration $y=(y_{n+1},\ldots,y_N)\in \pi_{I^c}(\mc{M}_N)$ from $\dGi^\g$ and an exterior configuration $z=(z_{n+1},\ldots,z_{N'})\in \pi_{I^c}(\mc{M}_{N'})$ from $\dGip^\g$ and compare the conditioned measures $\dGi^\g(\cdot \mid y)$ and $\dGip^\g(\cdot\mid z)$. We tilt the gap laws $\dGi^\g$ and $\dGip^\g$ by the penalty of~\eqref{eq:def nu hypersing} with $\iota=1$, namely
\begin{equation}\label{eq:F forcing}
    \mathrm{F}^\g:x\in\dR^n\mapsto\sum_{i=1}^{n} \theta(n^{-\ve}x_i)+\sum_{i=1}^n\theta(n^{-\ve}x_i^{-1})
    \,.
\end{equation}
Here $\theta$ is the convex function of~\eqref{eq:def nu hypersing}. From now on, $\dGiQ^\g$ and~$\dGipQ^\g$ denote the tilted measures
\begin{equation*}
   \dd \dGiQ^\g\propto e^{-\beta \mathrm{F}^\g\circ\pi}\dd\dGi^\g,  \quad \dd \dGipQ^\g\propto e^{-\beta \mathrm{F}^\g\circ\pi}\dd\dGip^\g
   \,.
\end{equation*}
For an exterior configuration $y$, set $L_N(y)=N-\sum_{a=n+1}^N y_a$ and $\mathcal K_y(x)=\Hc_N^\g(x,y)$, using the extension~\eqref{eq:ambient extension}. The conditional measure is supported on the simplex
\begin{equation}\label{eq:domain}
    \mc M_n(L)=\left\{x\in(0,+\infty)^n:\sum_{i=1}^n x_i=L\right\}
    \,.
\end{equation}
Denote the Hausdorff measure of~$\mc M_n(L)$ by ${\dd x}$. For every $y$ with $L_N(y)>0$, define on~$\mc M_n(L_N(y))$
\begin{equation}\label{eq:def cond m}
    \nu_n^y=\dGiQ^\g(\cdot\mid y),\qquad
    \dd\nu_n^y(x)\propto e^{-\beta(\mathcal K_y(x)+\FF^\g(x))}{\dd x}
    \,.
\end{equation}
The analogous definition for $z$ uses $\dGipQ^\g$ and $\mathcal K_z(x)=\Hc_{N'}^\g(x,z)$.

Put ${\bar s=\min(s,1)}$ and $\Upsilon(x)=1+\sum_{i=1}^n x_i^{-4s-12}$. For $\eta>0$, let
\begin{equation}\label{eq:interior rigidity event}
    \mc R_\eta(L)=\left\{x\in\mc M_n(L):
    \begin{array}{l}
        n^{-\eta}\leq x_i\leq n^\eta\text{ for }i\in I,\\
        \left|\sum_{i\in J}x_i-|J|\right|\leq n^\eta |J|^{{\bar s/2}}
        \text{ for every interval }J\subset I
    \end{array}\right\}
    \,.
\end{equation}
Let $\mc R_\eta=\bigcup_{L>0}\mc R_\eta(L)$. Fix constants $A_s,K_s>1$ and then $\ve>0$ sufficiently small. Their choices are specified in Appendix~\ref{section:loc laws}. We say that~$y$ is admissible if
\begin{equation}\label{eq:def ad}
    \begin{gathered}
    |L_N(y)-n|\leq n^{\bar s/2+\ve},\\
    \Bigl|\sum_{a=n+1}^{n+r}y_a-r\Bigr|\leq 4n^\ve r^{\bar s/2+\ve}\quad\text{and}\quad
    \Bigl|\sum_{a=N-r+1}^{N}y_a-r\Bigr|\leq 4n^\ve r^{\bar s/2+\ve}
    \quad(1\leq r\leq N-n)\,,
    \end{gathered}
\end{equation}
and
\begin{equation}\label{eq:ad endpoints}
    \nu_n^y(\mc R_\ve(L_N(y)))\geq\frac12,
    \qquad
    \dE_{\nu_n^y}\left[\Upsilon\mathds{1}_{\mc R_{K_s\ve}(L_N(y))^c}\right]
    \leq e^{-n^{2A_s\ve}}
    \,.
\end{equation}
By Lemma~\ref{lemma:local laws mu(t)} below, typical exterior configurations are admissible. We say that~$y$ and~$z$ are compatible if
\begin{equation}\label{eq:co}
    L_N(y)=L_{N'}(z)=L
    \,.
\end{equation}
For compatible configurations, the interpolating energy~$\Hc_n^\g(t)$ and the interpolating measure~$\nu(t)$ on the common simplex~$\mc M_n(L)$ are given by
\begin{equation}\label{eq:def mun(t)}
    \Hc_n^\g(t)=\FF^\g+(1-t)\mathcal K_y+t\mathcal K_z,
    \qquad
    \dd\nu(t)\propto e^{-\beta\Hc_n^\g(t)}{\dd x}
    \,.
\end{equation}
Thus $\nu(0)=\nu_n^y$ and $\nu(1)=\nu_n^z$.

Let $G:\dR^n\to\dR$ be a measurable bounded function. Differentiating the density of~$\nu(t)$ gives, for all $t\in (0,1)$,
\begin{equation*}
    {\frac{\dd}{\dd t}\dE_{\nu(t)}[G]}=\beta\Cov_{\nu(t)}[G,\mathcal K_y-\mathcal K_z]
    \,.
\end{equation*}
Integrating this identity between~$0$ and~$1$ gives the following representation.

\begin{lemma}\label{lemma:integrated repre}
Let $G:\dR^n\to\dR$ be a measurable bounded function, and let $\nu(t)$ be the measure~\eqref{eq:def mun(t)}. Then
\begin{equation}\label{eq:transport NN'}
    \dE_{\nu_n^{z}}[G]=\dE_{\nu_n^{y}}[G]+\beta\int_0^1 \Cov_{\nu(t)}[G,\mathcal K_y-\mathcal K_z]\dd t
    \,.
\end{equation}
\end{lemma}

We consider observables supported in the middle half of~$I$. The comparison uses decay of correlations under~$\nu(t)$ and the decay, towards the middle of~$I$, of the differences of consecutive coordinates of~${\nabla(\mathcal K_y-\mathcal K_z)}$. Appendix~\ref{section:loc laws} proves the following local laws.

\begin{lemma}\label{lemma:local laws mu(t)}
Let $s\in(0,1)\cup(1,+\infty)$ and $\beta>0$. Let $\ve>0$ be sufficiently small depending on~$s$. Let $n\geq2$ and $\max(4n,n^{2/s})\leq N\leq N'$ be integers. Then there exist $C(\beta,s,\ve),c(\beta,s,\ve),\delta(s,\ve)>0$, locally uniform in~$\beta$, such that, under~$\dGiQ^\g$, the exterior configuration~$y$ satisfies
\begin{equation}\label{eq:ad typical}
    \dGiQ^\g(y\text{ is not admissible})
    \leq Ce^{-cn^\delta}
    \,.
\end{equation}
For every admissible and compatible~$y$ and~$z$, and every $t\in[0,1]$, we have
\begin{equation}\label{eq:interpolation weighted tail}
    \dE_{\nu(t)}\left[\Upsilon\mathds{1}_{\mc R_{K_s\ve}(L)^c}\right]
    \leq Ce^{-cn^\delta}
    \,.
\end{equation}
In particular, for every $i\in I$ and every interval $J\subset I$,
\begin{align}\label{eq:est1}
    \nu(t)(n^{-K_s\ve}\leq x_i\leq n^{{K_s}\ve})
    &\geq1-Ce^{-cn^\delta},\\\label{eq:est2}
    \nu(t)\left(\left|\sum_{i\in J}x_i-|J|\right|>n^{{K_s}\ve}|J|^{{\bar s/2}}\right)
    &\leq Ce^{-cn^\delta}
    \,.
\end{align}
\end{lemma}

For both ranges of~$s$, let
\begin{equation*}
    P_I(t)=(1-t)\nabla_x^2\mathcal K_y+t\nabla_x^2\mathcal K_z
    \,,
\end{equation*}
and let $P_\partial(t)$ be the part of~$P_I(t)$ coming from the intervals that meet~$I$ without being contained in~$I$. For $0<s<1$, we split the operator~$A_1^{\nu(t)}$ as in~\eqref{eq:dec A1mu}, with the construction of Subsection~\ref{sub:re} and~$K_s\ve$ in place of~$\ve$. In particular, $A$ and~$A_\partial$ are as in Subsection~\ref{sub:re}, and~$\mc A$ is the good event~\eqref{eq:good5 gap} on~$\mc M_n(L)$, with~$P_I(t)$ in place of~$P_I$ and~$K_s\ve$ in place of~$\ve$. Let
\begin{equation*}
    A_0(t)=(\nabla^2\FF^\g+P_I(t)-P_\partial(t))\mathds{1}_{\mc A}
    +(A-A_\partial)\mathds{1}_{\mc A^c}+{A_\partial}
    \,.
\end{equation*}
Then $A_1^{\nu(t)}=\bar A_1^{\nu(t)}+\tM(t)$, with
\begin{equation*}
    \bar A_1^{\nu(t)}=\beta A_0(t)+\mc L^{\nu(t)}\otimes I_n
    -\sS^{\nu(t)}
    \,,
\end{equation*}
\begin{equation*}
    \tM(t)=\beta(P_\partial(t)-A_\partial)\mathds{1}_{\mc A}
    +\beta(\nabla^2\FF^\g+P_I(t)-A)\mathds{1}_{\mc A^c}
    +\sS^{\nu(t)}
    \,.
\end{equation*}
Here~$\sS^{\nu(t)}$ is the Schur term~\eqref{eq:schur term} of~$\nu(t)$. For $s>1$, let
\begin{equation*}
    {M_{\mc R}(t)}=\nabla^2\Hc_n^\g(t)\mathds{1}_{\mc R_{K_s\ve}(L)}
        +I\mathds{1}_{\mc R_{K_s\ve}(L)^c}
    \,.
\end{equation*}

For $i\in I$, let $\delta_i=\min(i,n+1-i)$. The next lemma lets us apply the results of Sections~\ref{section:decay short} and~\ref{section:large scale} to~$\nu(t)$.

\begin{lemma}\label{lemma:interpolation operator}
Let $s\in(0,1)\cup(1,+\infty)$ and $\beta>0$, and let $\ve>0$ be sufficiently small depending on~$s$. Let $n\geq2$ and $\max(4n,n^{2/s})\leq N\leq N'$ be integers, with~$n$ sufficiently large depending on~$s$ and~$\ve$. Let~$y$ and~$z$ be admissible and compatible, let $t\in[0,1]$, and let~$\nu(t)$ be the interpolating measure~\eqref{eq:def mun(t)}. Then there exists $C>0$ such that $(P_I(t))_{ij}\leq Cn^{C\ve}(1+|i-j|)^{-s}$ on~$\mc R_{K_s\ve}(L)$ for every $i,j\in I$, and
\begin{equation}\label{eq:whole simplex}
    |(P_I(t))_{ij}|^2+|\partial_{ij}\FF^\g|^2
    \leq Cn^C \Upsilon(x)
    \qquad\text{on }\mc M_n(L)
    \,.
\end{equation}
Moreover:
\begin{enumerate}
\item If $0<s<1$, then $\mc R_{K_s\ve}(L)\subset\mc A$, the measure~$\nu(t)$ satisfies Assumptions~\ref{assumptions:mu2} on $\Omega=\mc M_n(L)$, and the matrix~\eqref{eq:defsM'} built from~$A_0(t)$ satisfies Assumptions~\ref{assumptions:sM}. There also exist $C(\beta),c(\beta)>0$ and $\delta,\kappa>0$ such that, up to an additive error $C(\beta)e^{-c(\beta)n^\delta}\|\phi_1\|_{L^\infty(\nu(t))}\|\phi_2\|_{L^\infty(\nu(t))}$, for all ${\phi_1,\phi_2\in L^\infty(\nu(t))}$ and $1\leq i,j,\ell\leq n$,
\begin{equation}\label{eq:tildeHest}
    |\dE_{\nu(t)}[\phi_1 e_i\cdot\tM(t)(\phi_2 e_j)]|
    \leq \frac{C(\beta)n^{\kappa\ve}}{\delta_i^{s/2}\delta_j^{s/2}}
    \|\phi_1\|_{L^2(\nu(t))}\|\phi_2\|_{L^2(\nu(t))}
    \,,
\end{equation}
\begin{equation}\label{eq:tildeHest inc}
    |\dE_{\nu(t)}[\phi_1 e_i\cdot\tM(t)(\phi_2(e_j-e_{\ell}))]|
    \leq \frac{C(\beta)n^{\kappa\ve}|j-{\ell}|}
    {\delta_i^{s/2}\min(\delta_j,\delta_{\ell})^{1+s/2}}
    \|\phi_1\|_{L^2(\nu(t))}\|\phi_2\|_{L^2(\nu(t))}
    \,.
\end{equation}
\item If $s>1$, then~$\nu(t)$ satisfies Assumptions~\ref{assumptions:mu}, and the matrix~$M_{\mc R}(t)$ satisfies Assumptions~\ref{assumptions:sM1}.
\end{enumerate}
\end{lemma}

\begin{proof}
For both ranges of~$s$, we bound $N^{-s-2}g_s''(L_J/N)$ on~$\mc R_{K_s\ve}(L)$ for every interval~$J$ meeting~$I$. This gives the bounds on the interval coefficients. By~\eqref{eq:def ad}, an exterior prefix or suffix of cardinality at least ${n^{C\ve}}$ has gap sum at least half its cardinality. Every interval meeting~$I$ has at most two such exterior pieces. Its interior gaps and the pieces of cardinality below~${n^{C\ve}}$ therefore give $L_J\geq n^{-C\ve}|J|$. Let~$|J|\leq N/2$. By~\eqref{eq:def ad}, the at most two exterior pieces of~$J$ have gap sum at most~$|J|+8n^\ve N^{\bar s/2+\ve}$, and~$L\leq n+n^{\bar s/2+\ve}$. Since~$N\geq4n$ and~$\bar s/2+2\ve<1$, the complement of~$J$ has gap sum at least~$N/2-n-n^{\bar s/2+\ve}-8n^\ve N^{\bar s/2+\ve}\geq N/8$. Hence
\begin{equation*}
    N^{-s-2}g_s''(L_J/N)\leq Cn^{C\ve}|J|^{-s-2}
    \,.
\end{equation*}
The same bound holds at period~$N'$. Counting intervals gives $(P_I(t))_{ij}\leq Cn^{C\ve}(1+|i-j|)^{-s}$. On the whole simplex, the same split at cardinality~$n^{C\ve}$, together with $x_i\leq L\leq2n$, gives~\eqref{eq:whole simplex}.

For $0<s<1$, we have $\bar s=s$, so the gap conditions of~$\mc A$ are those of~$\mc R_{K_s\ve}(L)$, and the entry bound on~$P_I(t)$ gives its last condition. Hence $\mc R_{K_s\ve}(L)\subset\mc A$, and~$\nu(t)$ satisfies Assumptions~\ref{assumptions:mu2} on $\Omega=\mc M_n(L)$, with the first item given by~\eqref{eq:def mun(t)}, the second by Lemma~\ref{lemma:collision} and Appendix~\ref{section:existence c}, and the third by~\eqref{eq:interpolation weighted tail}. The coefficients of~$P_I(t)$ are convex combinations of those at the periods~$N$ and~$N'$. The proof of the first item of Lemma~\ref{lemma:ident} therefore shows that the matrix~\eqref{eq:defsM'} built from~$A_0(t)$ satisfies Assumptions~\ref{assumptions:sM}.

We now prove~\eqref{eq:tildeHest} and~\eqref{eq:tildeHest inc}. By the entry bound on~$P_I(t)$, Lemma~\ref{lemma:boundary intervals} controls the boundary term. Since $\mc R_{K_s\ve}(L)\subset\mc A$ and $\Upsilon\geq1$, the bounds~\eqref{eq:whole simplex} and~\eqref{eq:interpolation weighted tail} bound the exceptional terms by the additive error. Lemma~\ref{lemma:BDC} bounds the Schur term, which completes the proof of the first item.

Let $s>1$. Lemma~\ref{lemma:collision} and Appendix~\ref{section:existence c} show that~$\nu(t)$ satisfies Assumptions~\ref{assumptions:mu}. Since the exterior is fixed, $\nabla^2\Hc_n^\g(t)=\nabla^2\FF^\g+P_I(t)$ has no conditional covariance term, and the interval count above gives $|({M_{\mc R}(t)})_{ij}|\leq C_n(1+|i-j|)^{-s}$. The nearest-pair terms and the penalty give ${\nabla^2\Hc_n^\g(t)}\geq C_n^{-1}I$, so ${M_{\mc R}(t)}$ satisfies Assumptions~\ref{assumptions:sM1}.
\end{proof}

We apply the results of these sections with the gradient and the generator of~$\mc M_n(L)$, and with unknowns in~$\dR^n$.

\subsection{Estimate on the main equation}
We now bound the solution of the Helffer--Sjöstrand equation under~$\nu(t)$ when the source is supported on a single label~$i_0$ in the middle half of~$I$. For $0<s<1$, the preceding bounds let us apply Proposition~\ref{proposition:loc} to~$\bar A_1^{\nu(t)}$. We subtract the mean of its solution before comparing it with the solution of the constrained equation~\eqref{eq:real eq}.
\begin{lemma}\label{lemma:re eq}
Let $s\in(0,1)\cup(1,+\infty)$ and $\beta>0$, and let $\ve>0$ be sufficiently small depending on~$s$. Let $n\geq2$ and $N'\geq N\geq\max(4n,n^{2/s})$ be integers, let~$y$ and~$z$ be admissible and compatible configurations, and let $\nu(t)$ be defined by~\eqref{eq:def mun(t)}. Let $\chi_n\in L^2(\nu(t))$ be bounded, and let $1\leq i_0\leq n$ with $|i_0-n/2|\leq n/4$. Let $\psi\in H^1(\nu(t);\dR^I)$ and $\lambda\in L^2(\nu(t))$ solve
\begin{equation}\label{eq:real eq}
    \left\{
    \begin{array}{ll}
       \beta \nabla^2 {\Hc}_n^\g(t)\psi+\mc L^{\nu(t)}\psi
           =\chi_n e_{i_0}+\lambda\mathbf1_I &\text{on }\mc M_n(L),\\
       \sum_{i\in I}\psi_i=0 &\text{on }\mc M_n(L),\\
       \psi\cdot\vec n=0 &\text{on }\partial\mc M_n(L).
    \end{array}
    \right.
\end{equation}
Then there exist $C(\beta,s,\ve),c(\beta,s,\ve),\kappa(s),\delta(s,\ve)>0$ such that
\begin{equation*}
    \sum_{j=1}^n\delta_j^{\bar s/2-s}\|\psi_j\|_{L^2(\nu(t))}
    \leq Cn^{\kappa\ve-\bar s/2}
        \left(\|\chi_n\|_{L^2(\nu(t))}
            +e^{-cn^\delta}\|\chi_n\|_\infty\right)
\end{equation*}
and
\begin{equation}\label{eq:fixed unweighted}
    \sum_{j=1}^n\|\psi_j\|_{L^2(\nu(t))}
    \leq Cn^{\kappa\ve}
        \left(\|\chi_n\|_{L^2(\nu(t))}
            +e^{-cn^\delta}\|\chi_n\|_\infty\right)
    \,.
\end{equation}
\end{lemma}

\begin{proof}
We may assume that~$n$ is sufficiently large. Indeed, the nearest-pair terms and the penalty give~$\nabla^2\Hc_n^\g(t)\geq c(s)n^{-(s+2)\ve}I_n$ on~$\mc M_n(L)$. Testing~\eqref{eq:real eq} with~$\psi$, which removes the multiplier, gives~$\|\psi\|_{L^2(\nu(t))}\leq C\beta^{-1}n^{(s+2)\ve}\|\chi_n\|_{L^2(\nu(t))}$. Since~$\delta_j^{\bar s/2-s}\leq1$, both assertions follow for bounded~$n$ after increasing~$C$.
By linearity, we may assume that $\|\chi_n\|_{L^2(\nu(t))}=1$. First let $0<s<1$. Let $v$ solve $\bar A_1^{\nu(t)}v=\chi_ne_{i_0}$, without the constraint of~\eqref{eq:real eq}. By Lemma~\ref{lemma:ident}, $v$ consists of the first~$n$ coordinates of the solution~$\Psi$ of~\eqref{eq:equ 54} under~$\nu(t)$. Let~$E$ be as defined after~\eqref{eq:Lalpha5}. Summing the equation of~$\Psi$ over~$\bar I$ gives
\begin{equation*}
    (2\beta\vartheta_{\bn}+\mc L^{\nu(t)})\sum_{k\in\bar I}\Psi_k
    =\chi_n-\beta\mathbf1\cdot E\Psi
    \,.
\end{equation*}
The source has $L^2$ norm at most~${C_n}$. Testing this equation with $\sum_{k\in\bar I}\Psi_k$, and using the value and gradient decay of Proposition~\ref{proposition:loc}, therefore yields
\begin{equation}\label{eq:fixed fiber mean}
    \|\mathbf1_I\cdot v\|_{L^2(\nu(t))}\leq C_n n^{-(1-s)},
    \qquad
    \|\nabla(\mathbf1_I\cdot v)\|_{L^2(\nu(t))}\leq C_n n^{-(1-s)/2}
    \,.
\end{equation}
Indeed, we subtract the exterior sum $\sum_{k\in\bar I\setminus I}\Psi_k$ from~$\sum_{k\in\bar I}\Psi_k$, and sum the value and gradient bounds over $k\notin I$, whose distance from~$i_0$ is at least~$n/4$.

Let $\bar v={n^{-1}(\mathbf1_I\cdot v)}\mathbf1_I$. Let $\mathcal E_t$ be the quadratic form of $\beta\nabla^2\Hc_n^\g(t)+\mc L^{\nu(t)}$. On~$\mc R_{K_s\ve}(L)$, the entry bounds give $\mathbf1_I\cdot\nabla^2\Hc_n^\g(t)\mathbf1_I\leq C_n n^{2-s}$. Proposition~\ref{proposition:maximum principle}, applied to the equation of~$\Psi$, gives $\|\mathbf1_I\cdot v\|_{L^\infty(\nu(t))}\leq C_n\sqrt n\|\chi_n\|_\infty$. Hence~\eqref{eq:interpolation weighted tail} controls the contribution of the complement of~$\mc R_{K_s\ve}(L)$. Thus
\begin{equation*}
    \mathcal E_t(\bar v,\bar v)
    \leq C_n n^{-s}\|\mathbf1_I\cdot v\|_2^2+n^{-1}\|\nabla(\mathbf1_I\cdot v)\|_2^2
        +Ce^{-cn^\delta}{\|\chi_n\|_\infty^2}
    \leq C_n n^{-2+s}+Ce^{-cn^\delta}{\|\chi_n\|_\infty^2}
    \,.
\end{equation*}
The field $u=v-\bar v-\psi$ is tangent to~$\mc M_n(L)$. Testing~\eqref{eq:real eq} against~$u$ removes its multiplier and gives
\begin{equation*}
    \mathcal E_t(u,u)=\dE_{\nu(t)}[u\cdot\tM(t)v]
        -\mathcal E_t(u,\bar v)
    \,.
\end{equation*}
Proposition~\ref{proposition:maximum principle}, applied to~\eqref{eq:real eq} in an orthonormal basis of~$\mathbf1_I^\perp$ and to the equation of~$\Psi$, bounds $\|\psi\|_{L^\infty(\nu(t))}$ and~$\|v\|_{L^\infty(\nu(t))}$. Since $|\bar v|\leq|v|$, this also bounds~$u$. By~\eqref{eq:tildeHest}--\eqref{eq:tildeHest inc}, the decay of~$v$ and~\eqref{eq:fixed fiber mean}, Lemma~\ref{lemma:column subtraction} gives
\begin{equation*}
    |\dE_{\nu(t)}[u\cdot\tM(t)v]|
    \leq C_n n^{-1/2}\|u\|_2+Ce^{-cn^\delta}{\|\chi_n\|_\infty^2}
    \,.
\end{equation*}
Coercivity and Cauchy--Schwarz for~$\mathcal E_t$, followed by Young's inequality, imply
\begin{equation*}
    \|u\|_2^2\leq C_n n^{-1}(1+e^{-cn^\delta}{\|\chi_n\|_\infty^2})
    \,.
\end{equation*}
Also $\|\bar v\|_2\leq C_n n^{-3/2+s}\leq C_n n^{-1/2}$. Combining these bounds with the decay of~$v$ and $\sum_{j\in I}\delta_j^{-s}\leq Cn^{1-s}$ proves the claimed weighted sum for $0<s<1$.

For $s>1$, let $v$ solve $(\beta {M_{\mc R}(t)}+\mc L^{\nu(t)})v=\chi_n e_{i_0}$. Proposition~\ref{proposition:optimal illus} and Lemma~\ref{lemma:illustrative} give
\begin{equation*}
    \|v_j\|_2\leq C_n(1+|j-i_0|)^{-s},
    \qquad \|\mathbf1_I\cdot v\|_2+\|\nabla(\mathbf1_I\cdot v)\|_2\leq C_n
    \,.
\end{equation*}
Let $\bar v={n^{-1}(\mathbf1_I\cdot v)}\mathbf1_I$. On~$\mc R_{K_s\ve}(L)$, we have $\|\nabla^2\Hc_n^\g(t)\|\leq C_n$, and~\eqref{eq:whole simplex} and~\eqref{eq:interpolation weighted tail} control its complement. Hence
\begin{equation*}
    \mathcal E_t(\bar v,\bar v)\leq C_nn^{-1}+Ce^{-cn^\delta}{\|\chi_n\|_\infty^2}
    \,.
\end{equation*}
Testing with $u=v-\bar v-\psi$ as above, the matrix difference ${\nabla^2\Hc_n^\g(t)}-{M_{\mc R}(t)}$ is supported on the complement of~$\mc R_{K_s\ve}(L)$. The maximum principle, the bound~\eqref{eq:whole simplex} and Young's inequality yield
\begin{equation*}
    \|\psi-v\|_2
    \leq C_nn^{-1/2}+Ce^{-cn^\delta}{\|\chi_n\|_\infty}
    \,.
\end{equation*}
Since $s>1$, the weight $\delta_j^{1/2-s}$ is square summable. The pointwise bound for~$v$ also gives
\begin{equation*}
    \sum_j\delta_j^{1/2-s}\|v_j\|_2
    \leq C_nn^{-(s-1/2)}\leq C_nn^{-1/2}
    \,.
\end{equation*}
Cauchy--Schwarz for $\psi-v$ proves the short-range assertion. In both regimes, the decay of~$v$ is summable in~$j$, and the difference of~$\psi$ with $v-\bar v$ for $0<s<1$, or with~$v$ for $s>1$, has $L^2(\nu(t);\ell^2)$ norm at most $C_nn^{-1/2}({1}+e^{-cn^\delta}{\|\chi_n\|_\infty})$. The Cauchy--Schwarz inequality over~$j\in I$ and the bound $\sum_j\|\bar v_j\|_2=\|\mathbf1_I\cdot v\|_2\leq C_n$ prove~\eqref{eq:fixed unweighted}.
\end{proof}

\subsection{\texorpdfstring{Proof of Theorem~\ref{theorem:convergence} and Theorem~\ref{theorem:quantitative conv}}{Proof of Theorem \getrefnumber{theorem:convergence} and Theorem \getrefnumber{theorem:quantitative conv}}}\label{sub:proof conv}

Inserting the decay estimate of Lemma~\ref{lemma:re eq} into the identity~\eqref{eq:transport NN'}, we compare the measures $\nu_n^y$ and $\nu_n^z$ when the exterior configurations $y$ and $z$ are compatible. After integrating $y$ and $z$ over the set of admissible and compatible configurations, it remains to show that typical fluctuations of $x_1+\ldots+x_n$ do not affect the measures much.

We compare the conditional expectations at a common gap sum after averaging over a measure dominated by both laws of the gap sum. The following estimate controls the error of this averaging.

\begin{lemma}\label{lemma:length averaging}
Let $s\in(0,1)\cup(1,+\infty)$ and $\beta>0$, let $\ve>0$ be sufficiently small, and let $n\geq2$ and $N\geq\max(4n,n^{2/s})$ be integers. Let $\nu_N=\pi\#\dGiQ^\g$ and $S(x)=\sum_{i\in I}x_i$, and let $f_N$ be the density of $S$ under~$\nu_N$. Let $G:\dR^n\to\dR$ be Lipschitz and depend only on coordinates in a set $J\subset\{i\in I:|i-n/2|\leq n/4\}$. Let
\begin{equation*}
    {D_{\mc R}}=\sum_{i\in J}\sup_{{\mc R_{K_s\ve}}}|\partial_iG|,
    \qquad D_\infty=\sum_{i\in J}\|\partial_iG\|_\infty,
    \qquad h_N(L)=\dE_{\nu_N}[G\mid S=L]
    \,.
\end{equation*}
Then there exist $C(\beta,s,\ve),c(\beta,s,\ve),\kappa(s),\delta(s,\ve)>0$ such that
\begin{equation}\label{eq:length expectation variance}
    \Var_{f_N}[h_N]^{1/2}
    \leq Cn^{\kappa\ve-\bar s/2}{D_{\mc R}}
        +Ce^{-cn^\delta}D_\infty
    \,.
\end{equation}
Moreover,
\begin{equation}\label{eq:length variance}
    |\dE_{\nu_N}S-n|\leq Ce^{-cn^\delta},
    \qquad cn^{\bar s}\leq\Var_{\nu_N}[S]\leq Cn^{{\bar s}+\kappa\ve}
    \,.
\end{equation}
\end{lemma}

\begin{proof}
Write $\nu^L=\nu_N(\cdot\mid S=L)$, and let~$\|\cdot\|_2$ be the norm of~$L^2(\nu^L)$. Since the density~$f_N$ is log-concave by Pr\'ekopa's theorem, the one-dimensional Poincar\'e inequality reduces~\eqref{eq:length expectation variance} to a bound on~$h_N'$ in~$L^2(f_N)$. In Step~2, we show that $h_N'(L)=-\sum_{i\in J}\dE_{\nu^L}\lambda^{(i)}$, where~$\lambda^{(i)}$ is the multiplier of the equation of Step~1 on the fiber~$\{S=L\}$. Step~1 bounds these multipliers. We first treat $0<s<1$. All estimates below hold for $f_N$-almost every~$L$. We use the marginal energy~$\widetilde{\Hc}_n^\g$ of Lemma~\ref{lemma:marginal hessian}, the matrices $P_I,Q$ of~\eqref{eq:short marginal Hessian} and the coefficients $b_i,\bar b$ of~\eqref{eq:ijH}. Fix $\kappa$ sufficiently large. Let $\mc B$ be the event~\eqref{eq:good5 gap}, with $\ve$ replaced by $K_s\ve$, intersected with
\begin{equation*}
    N^s\left(\bar b+\max_{i\in I}b_i\right)\leq n^{\kappa\ve},
    \qquad
    N\max_{i<n}\delta_i^s|b_{i+1}-b_i|\leq n^{\kappa\ve}
    \,.
\end{equation*}
Let $C_0(s)$ be a large constant, fixed below, and let
\begin{equation*}
    W_{\mc B}=n^{C_0}\left(1+\sum_i|\partial_{ii}\FF^\g|
        +\sum_{i,j}(P_I)_{ij}+\sum_i b_i+\bar b\right),
    \qquad r(L)=\dE_{\nu^L}[W_{\mc B}\mathds{1}_{\mc B^c}]
    \,.
\end{equation*}
The moment bounds of Lemmas~\ref{lemma:good event long} and~\ref{lemma:compensation} imply
\begin{equation}\label{eq:length bad remainder}
    \int (r+r^2)f_N\leq Ce^{-cn^\delta}
    \,.
\end{equation}
Indeed, $\nu_N(\mc B^c)$ is exponentially small, $W_{\mc B}$ has polynomially bounded fourth moment, and conditional Jensen bounds $\int r^2f_N$ by $\dE_{\nu_N}[W_{\mc B}^2\mathds{1}_{\mc B^c}]$.

\paragraph{\emph{Step 1: the multiplier on each fiber}}
For $i\in J$, let $\psi^{(i)}$ solve the tangent equation
\begin{equation*}
    \beta {\nabla^2\widetilde\Hc_n^\g}\psi^{(i)}+\mc L^{\nu^L}\psi^{(i)}
        =(\partial_iG) e_i+\lambda^{(i)}\mathbf1_I,
    \qquad \sum_j\psi_j^{(i)}=0
    \,.
\end{equation*}
We split $A_1^{\nu^L}=\bar A_1^{\nu^L}+\tM$ as in~\eqref{eq:dec A1mu}, with the matrix~$A_0$ of~\eqref{eq:A0 e}, built with~$K_s\ve$ in place of~$\ve$ as before Lemma~\ref{lemma:interpolation operator}. Assumptions~\ref{assumptions:sM} are pointwise bounds, so the first item of Lemma~\ref{lemma:ident} gives them on every fiber. By Lemma~\ref{lemma:collision} and the end of Appendix~\ref{section:existence c}, the measure~$\nu^L$ satisfies the first two items of Assumptions~\ref{assumptions:mu2} on~$\{S=L\}$, and the remainder~$r(L)$ replaces the third item. We repeat the proof of Lemma~\ref{lemma:cond variance} with~$\nu^L$ in place of~$\nu$ and with events~$\mc G$ that contain~$\mc B$, and use the first bound of~\eqref{eq:cond truncation} instead of the second. Since~$C_0$ is large, the bounds~\eqref{eq:Mg1}--\eqref{eq:Mg2} then hold for~$\tM$ up to the additive error $r(L)\|\phi_1\|_{L^\infty(\nu^L)}\|\phi_2\|_{L^\infty(\nu^L)}$. Lemma~\ref{lemma:BDC} bounds the Schur term.

Let $v$ solve $\bar A_1^{\nu^L}v=(\partial_iG)e_i$. The proof of~\eqref{eq:fixed fiber mean}, with~$\nu^L$ in place of~$\nu(t)$, gives these bounds after multiplication of their right sides by~$\|\partial_iG\|_2$. The proof of Lemma~\ref{lemma:re eq} for $0<s<1$, with~$r(L)$ in place of the tail bound~\eqref{eq:interpolation weighted tail}, gives
\begin{equation*}
    \left\|\psi^{(i)}-v+n^{-1}(\mathbf1_I\cdot v)\mathbf1_I\right\|_2^2
    \leq C_n n^{-1}{\|\partial_iG\|_2^2}+n^{-2}{\|\partial_iG\|_\infty^2}r(L)
    \,.
\end{equation*}
Here $\bar v=n^{-1}(\mathbf1_I\cdot v)\mathbf1_I$, as in the proof of Lemma~\ref{lemma:re eq}, and the quadratic form of $\beta\nabla^2\widetilde\Hc_n^\g+\mc L^{\nu^L}$ at~$\bar v$ is at most
\begin{equation*}
    C_n n^{-2+s}\|\partial_iG\|_2^2+n^{-2}\|\partial_iG\|_\infty^2r(L)
    \,,
\end{equation*}
since ${\mathbf1_I\cdot\nabla^2\widetilde\Hc_n^\g\mathbf1_I\leq C_n n^{2-s}}$ on~$\mc B$ and~$W_{\mc B}$ controls the complement. By the maximum principle and the definition of~$W_{\mc B}$, each exceptional term is at most $Cn^{C-C_0}\|\partial_iG\|_\infty^2r(L)$. We choose $C_0\geq C+2$, so that these terms are at most $n^{-2}\|\partial_iG\|_\infty^2r(L)$. The decay of~$v$ given by Proposition~\ref{proposition:loc} and the Cauchy--Schwarz inequality now give
\begin{equation}\label{eq:length fiber sums}
    \begin{aligned}
    \sum_j\|\psi_j^{(i)}\|_2
        &\leq C_n\|\partial_iG\|_2+\|\partial_iG\|_\infty\sqrt{r(L)},\\
    \sum_j\delta_j^{-s/2}\|\psi_j^{(i)}\|_2
        &\leq C_n n^{-s/2}\|\partial_iG\|_2+\|\partial_iG\|_\infty\sqrt{r(L)}
    \,.
    \end{aligned}
\end{equation}

On~$\mc B$, the entry bounds for $P_I$ and the boundary bounds for $Q$ and $\Cov_{\mu_x}[\Xi,\Xi]$ yield
\begin{equation*}
    \sum_k|({\nabla^2\widetilde\Hc_n^\g})_{kj}|
    \leq C_n\left(n^{1-s}+n^{1-s/2}\delta_j^{-s/2}\right)
    \,.
\end{equation*}
Since $\sum_j\psi^{(i)}_j=0$, summing the tangent equation over~$j$ removes the generator term and gives $n\lambda^{(i)}=\beta\mathbf1_I\cdot\nabla^2\widetilde\Hc_n^\g\psi^{(i)}-\partial_iG$. We integrate this identity under~$\nu^L$. Equations~\eqref{eq:length fiber sums}, the maximum principle on~$\mc B^c$ and the bound $|\nabla^2\widetilde\Hc_n^\g|\leq CW_{\mc B}$, which follows from~\eqref{eq:short marginal Hessian} and~\eqref{eq:condQ}, give
\begin{equation}\label{eq:length multiplier}
    |\dE_{\nu^L}\lambda^{(i)}|
    \leq C_n n^{-s}{\|\partial_iG\|_2}+\|\partial_iG\|_\infty\bigl(\sqrt{r(L)}+r(L)\bigr)
    \,.
\end{equation}

For $s>1$, we use the good event~$\mc B$ of~\eqref{eq:short good coefficients}, with $K_s\ve$ in place of~$\ve$ in its gap conditions. We define $W_{\mc B}$ and~$r(L)$ as above, after increasing~$C_0$ if necessary. Lemma~\ref{lemma:good event}, the moments of Lemma~\ref{lemma:single gap tails} and the conditional Jensen inequality prove~\eqref{eq:length bad remainder} for this event as well. With the matrices of~\eqref{eq:short frozen matrices}, let $v$ and~$u$ solve $(\beta M_0+\mc L^{\nu^L})v=(\partial_iG)e_i$ and $(\beta M_{\mc B}+\mc L^{\nu^L})u=(\partial_iG)e_i$. Equations~\eqref{eq:short marginal frozen comparison}--\eqref{eq:short marginal sums} hold on every fiber. In particular,
\begin{equation*}
    \sum_j\|u_j\|_2+\|\nabla(\mathbf1_I\cdot u)\|_2\leq C_n\|\partial_iG\|_2,
    \qquad
    \sum_j\delta_j^{1/2-s}\|u_j\|_2
    \leq C_n n^{-1/2}{\|\partial_iG\|_2}
    \,.
\end{equation*}
The second bound uses $\sum_j\delta_j^{1-2s}\leq C(s)$, the decay of~$v$, and $N\geq4n$. The quadratic form of $\beta\nabla^2\widetilde\Hc_n^\g+\mc L^{\nu^L}$ at $n^{-1}(\mathbf1_I\cdot u)\mathbf1_I$ is at most
\begin{equation*}
    C_n n^{-1}{\|\partial_iG\|_2^2}+n^{-2}{\|\partial_iG\|_\infty^2}r(L)
    \,.
\end{equation*}
Indeed, ${\nabla^2\widetilde\Hc_n^\g}\leq C_nI$ on~$\mc B$, and the maximum principle and~$W_{\mc B}$ control its complement. We test the tangent equation against $u-n^{-1}(\mathbf1_I\cdot u)\mathbf1_I-\psi^{(i)}$. The difference $\nabla^2\widetilde\Hc_n^\g-M_{\mc B}$ vanishes on~$\mc B$. The same maximum-principle bound and the definition of~$r(L)$ give
\begin{equation*}
    {\left\|\psi^{(i)}-u+n^{-1}(\mathbf1_I\cdot u)\mathbf1_I\right\|_2^2}
    \leq C_n n^{-1}{\|\partial_iG\|_2^2}+n^{-2}{\|\partial_iG\|_\infty^2}r(L)
    \,.
\end{equation*}
Consequently,
\begin{equation*}
    \sum_j\|\psi_j^{(i)}\|_2
        \leq C_n\|\partial_iG\|_2+\|\partial_iG\|_\infty\sqrt{r(L)},
    \qquad
    \sum_j\delta_j^{1/2-s}\|\psi_j^{(i)}\|_2
        \leq C_n n^{-1/2}{\|\partial_iG\|_2}+\|\partial_iG\|_\infty\sqrt{r(L)}
    \,.
\end{equation*}
Since $s>1$, we have $\sum_j\delta_j^{1/2-s}\leq C_nn^{1/2}$. Expanding the covariance in~\eqref{eq:short marginal Hessian} and using~\eqref{eq:short score bounds} therefore gives, on~$\mc B$,
\begin{equation*}
    \sum_k|({\nabla^2\widetilde\Hc_n^\g})_{kj}|
    \leq C_n\left(1+n^{1/2}\delta_j^{1/2-s}\right)
    \,.
\end{equation*}
The terms independent of~$\delta_j$ are bounded because $P_I$ has bounded column sums, $n\leq N$ and the sum of~$\delta_j^{1/2-s}$ over~$j$ is at most~$C_nn^{1/2}$. We sum the tangent equation over~$j$ and integrate. The above bounds on $\sum_j\|\psi_j^{(i)}\|_2$ and $\sum_j\delta_j^{1/2-s}\|\psi_j^{(i)}\|_2$ then prove the short-range counterpart of~\eqref{eq:length multiplier}:
\begin{equation*}
    |\dE_{\nu^L}\lambda^{(i)}|
    \leq C_n n^{-1}{\|\partial_iG\|_2}
        +\|\partial_iG\|_\infty\bigl(\sqrt{r(L)}+r(L)\bigr)
    \,.
\end{equation*}
Thus both regimes give the multiplier bound with ${n^{-\bar s}}$.

\paragraph{\emph{Step 2: differentiation of the conditional expectation}}
We identify each fiber~$\{S=L\}$ with~$\mathbf1_I^\perp$ by the translation by~$(L/n)\mathbf1_I$, and extend the density by zero outside the positive simplex. Its flatness at collisions permits differentiation, and
\begin{equation*}
    {h_N'}(L)=n^{-1}\dE_{\nu^L}[{\mathbf1_I\cdot\nabla G}]
        -\beta{n^{-1}}\Cov_{\nu^L}[G,{\mathbf1_I\cdot\nabla\widetilde\Hc_n^\g}]
        =-\sum_{i\in J}\dE_{\nu^L}\lambda^{(i)}
    \,.
\end{equation*}
By the tangent Helffer--Sjöstrand representation, $\Cov_{\nu^L}[G,\mathbf1_I\cdot\nabla\widetilde\Hc_n^\g]=\sum_{i\in J}\dE_{\nu^L}[\mathbf1_I\cdot\nabla^2\widetilde\Hc_n^\g\psi^{(i)}]$. Summing the identity of Step~1 over~$i\in J$ gives $\beta\sum_{i\in J}\mathbf1_I\cdot\nabla^2\widetilde\Hc_n^\g\psi^{(i)}=\mathbf1_I\cdot\nabla G+n\sum_{i\in J}\lambda^{(i)}$. Taking the expectation under~$\nu^L$ proves the last equality of the display. Since $\mc B\subset\mc R_{K_s\ve}(L)$ and $W_{\mc B}\geq1$, we have $\|\partial_iG\|_2\leq\sup_{\mc R_{K_s\ve}}|\partial_iG|+\|\partial_iG\|_\infty\sqrt{r(L)}$. Hence~\eqref{eq:length bad remainder} and~\eqref{eq:length multiplier} give
\begin{equation}\label{eq:length derivative average}
    \|h_N'\|_{L^2(f_N)}
    \leq C_n n^{-\bar s}{D_{\mc R}}+Ce^{-cn^\delta}D_\infty
    \,.
\end{equation}
The one-dimensional Poincar\'e inequality~\cite[Lemma~2.1]{eldan2014dimensionality} gives
\begin{equation*}
    \Var_{f_N}[h_N]
    \leq12\Var_{\nu_N}[S]\int|h_N'|^2f_N
    \,.
\end{equation*}
Lemmas~\ref{theorem:almost optimal rigidity c} and~\ref{lemma:local hyper}, with the moments of Lemma~\ref{lemma:single gap tails}, give $\Var_{\nu_N}[S]\leq Cn^{{\bar s}+C\ve}$, so~\eqref{eq:length derivative average} proves~\eqref{eq:length expectation variance}.

\paragraph{\emph{Step 3: the lower variance bound}}
On the full gap simplex, set
\begin{equation*}
    w=\frac1n\mathbf1_I-\frac1{N-n}\mathbf1_{I^c}
    \,.
\end{equation*}
Let $D_w=w\cdot\nabla$ be the derivative in the direction~$w$. Then $D_wS=1$. For $r\leq N/2$, with labels modulo~$N$,
\begin{equation*}
    \sum_{i=1}^N\Bigl(\sum_{j=i}^{i+r-1}w_j\Bigr)^2
    \leq C\min\left(\frac{r^2}{n},r\right)
    \,.
\end{equation*}
By~\eqref{eq:coefficient moments} with $p=1$, and since the contribution of the penalty is exponentially small, we obtain $\dE_{\dGiQ^\g}[D_w^2{(\Hc_N^\g+\FF^\g)}]\leq Cn^{-\bar s}$. Integration by parts and Cauchy--Schwarz imply
\begin{equation*}
    1=\beta\dE_{\dGiQ^\g}[(S-\dE S)D_w(\Hc_N^\g+\FF^\g)],
    \qquad
    \Var_{\nu_N}[S]\geq
    \frac1{\beta\dE_{\dGiQ^\g}[D_w^2(\Hc_N^\g+\FF^\g)]}
    \geq cn^{\bar s}
    \,.
\end{equation*}
Finally, the penalty vanishes outside an exponentially small event. Comparing with $\dGi^\g$, under which $\dE S=n$, and using the moments of Lemma~\ref{lemma:single gap tails} proves the mean bound in~\eqref{eq:length variance}.
\end{proof}

\begin{proposition}\label{proposition:distinct}
Let $s\in(0,1)\cup(1,+\infty)$ and $\beta>0$, and let $\ve>0$ be sufficiently small. Let $n\geq2$, $N$ and~$N'$ be integers with ${N'\geq N\geq\max(4n,n^{2/s})}$, and let $K\leq n/4$ be an integer, with $K<\lfloor n/4\rfloor$ if $n\equiv1\pmod4$. Let $J=\{\lfloor n/2\rfloor-K,\ldots,\lfloor n/2\rfloor+K\}$, and let $G:\dR^n\to\dR$ be Lipschitz and depend only on the variables~$x_i$ for $i\in J$. Then there exist $C(\beta),c(\beta)>0$ and $\delta,\kappa>0$ such that
\begin{multline}\label{eq:comp scales}
     \Bigl|\dE_{\dGiQ^\g}[G(x_1,\ldots,x_n)]-\dE_{\dGipQ^\g}[G(x_1,\ldots,x_n)]\Bigr|\\
     \leq C(\beta)n^{\kappa\ve-\min(s,1)/2}\Bigl(\sum_{i\in J}\sup_{{\mc R_{K_s\ve}}}|\partial_i G|+e^{-c(\beta)n^{\delta}}\sum_{i\in J}\|\partial_i G\|_\infty\Bigr)
     \,.
\end{multline}
The same bound holds with $\dGi^\g$ and~$\dGip^\g$ in place of~$\dGiQ^\g$ and~$\dGipQ^\g$.
\end{proposition}

\begin{proof}
Let $D_{\mc R}$ and~$D_\infty$ be as in Lemma~\ref{lemma:length averaging}. We may assume that~$n$ is sufficiently large. Indeed, for bounded~$n$, subtracting~$G(\mathbf1{_I})$ and using the moments of Lemma~\ref{lemma:single gap tails} under both laws bound the left side of~\eqref{eq:comp scales} by~$CD_\infty$, which the exponential term of~\eqref{eq:comp scales} covers after increasing~$C$. The penalty~\eqref{eq:F forcing} is nonnegative and vanishes unless a gap with label in~$I$ lies outside $[2n^{-\ve},n^\ve/2]$. Lemma~\ref{lemma:single gap tails} gives exponentially small probability to this event. Hence
\begin{equation}\label{eq:tv6}
    \mathrm{TV}(\dGi^\g,\dGiQ^\g)
    +\mathrm{TV}(\dGip^\g,\dGipQ^\g)
    \leq Ce^{-cn^\delta}
    \,.
\end{equation}
Subtracting $G(\mathbf1{_I})$ and using the moments of Lemma~\ref{lemma:single gap tails} bounds each corresponding difference of expectations by $Ce^{-cn^\delta}D_\infty$.
Let $y\in\pi_{I^c}(\mc{M}_N)$ and $z\in \pi_{I^c}(\mc{M}_{N'})$ be two admissible configurations in the sense of~\eqref{eq:def ad}--\eqref{eq:ad endpoints}. Assume that $y$ and $z$ are compatible in the sense of~\eqref{eq:co}. Let $\mc M_n(L)$, $\nu_n^y$, $\nu_n^z$ and~$\nu(t)$ be as in~\eqref{eq:domain}, \eqref{eq:def cond m} and~\eqref{eq:def mun(t)}. For each $i\in J$, let $\psi^{(t,i)}$ and~$\lambda^{(t,i)}$ solve~\eqref{eq:real eq} with $\chi_n=\partial_iG$ and $i_0=i$. Since $\sum_{j\in I}\psi_j^{(t,i)}=0$, subtracting the middle entry ${\partial_{\lceil n/2\rceil}(\mathcal K_y-\mathcal K_z)}$ from every entry of~$\nabla(\mathcal K_y-\mathcal K_z)$ does not change the product ${\nabla(\mathcal K_y-\mathcal K_z)\cdot\psi^{(t,i)}}$. After this subtraction, the $j$-th entry is, up to sign, the sum of $(\partial_k-\partial_{k+1})(\mathcal K_y-\mathcal K_z)$ over $\min(j,\lceil n/2\rceil)\leq k<\max(j,\lceil n/2\rceil)$. On~$\mc R_{K_s\ve}(L)$, we have $x_1+\cdots+x_k\geq n^{-K_s\ve}k$ and $x_{k+1}+\cdots+x_n\geq n^{-K_s\ve}(n-k)$. Summing~\eqref{eq:exterior force oscillation} over~$k$ therefore bounds this entry by $Cn^{C\ve}(\delta_j^{\bar s/2-s}+nN^{-s})$. The additional term is controlled by~\eqref{eq:fixed unweighted}, since $nN^{-s}\leq n^{-1}$. On the whole simplex, the sum of the squares of these entries is at most $Cn^CW$, with the function~$\Upsilon$ defined before~\eqref{eq:interior rigidity event}. By the Helffer--Sj\"ostrand representation, which the end of Appendix~\ref{section:existence c} gives for~$\nu(t)$, the covariance in~\eqref{eq:transport NN'} equals $\sum_{i\in J}\dE_{\nu(t)}[\nabla(\mathcal K_y-\mathcal K_z)\cdot\psi^{(t,i)}]$. On~$\mc R_{K_s\ve}(L)$, we bound this sum with Lemma~\ref{lemma:re eq} for $\chi_n=\partial_iG$, and with $\|\partial_iG\|_{L^2(\nu(t))}\leq\sup_{\mc R_{K_s\ve}}|\partial_iG|+\|\partial_iG\|_\infty\nu(t)(\mc R_{K_s\ve}(L)^c)^{1/2}$. On the complement, we use Proposition~\ref{proposition:maximum principle} for~$\psi^{(t,i)}$ and~\eqref{eq:interpolation weighted tail}. Lemma~\ref{lemma:integrated repre} therefore gives
\begin{equation}\label{eq:fixed}
    |\dE_{\nu_n^y}[G]-\dE_{\nu_n^z}[G]|
    \leq\beta\sum_{i\in J}\int_0^1
        |\dE_{\nu(t)}[\nabla(\mathcal K_y-\mathcal K_z)\cdot\psi^{(t,i)}]|\dd t
    \leq C(\beta)n^{\kappa\ve-\bar s/2}
        {\bigl(D_{\mc R}+e^{-c(\beta)n^\delta}D_\infty\bigr)}
    \,.
\end{equation}
We now average~\eqref{eq:fixed} over a common law of the gap sum. Let $f_N,f_{N'}$ and ${h_N},{h_{N'}}$ be as in Lemma~\ref{lemma:length averaging}, with~$\nu_{N'}$ defined as~$\nu_N$, with~$N'$ in place of~$N$. By~\cite[Lemma~2.2(b)]{eldan2014dimensionality}, the log-concave density~$f_N$ is at least $c\Var_{\nu_N}[S]^{-1/2}$ on the interval of radius~$c\Var_{\nu_N}[S]^{1/2}$ around~$\dE_{\nu_N}S$, and the same holds for~$f_{N'}$. Let $a_0=\int\min(f_N,f_{N'})$ and $f_{N,N'}=\min(f_N,f_{N'})/a_0$. By~\eqref{eq:length variance}, both intervals have radius at least~$cn^{\bar s/2}$ and centers within~$1$ of~$n$, so $a_0\geq C^{-1}n^{-C\ve}$. Since $f_{N,N'}\leq f_N/a_0$ and $f_{N,N'}\leq f_{N'}/a_0$, Cauchy--Schwarz and~\eqref{eq:length expectation variance} give
\begin{equation*}
    \left|\int {h_N}f_{N,N'}-\dE_{\nu_N}G\right|
    +\left|\int {h_{N'}}f_{N,N'}-\dE_{\nu_{N'}}G\right|
    \leq Cn^{C\ve-\bar s/2}{D_{\mc R}}+Ce^{-cn^\delta}D_\infty
    \,.
\end{equation*}
Conditionally on~$L$, we draw $y$ and~$z$ independently from the exterior laws of $\dGiQ^\g$ and $\dGipQ^\g$ given~$S=L$. They are compatible. Under the law with density~$f_{N,N'}$, the probability that either exterior is inadmissible is bounded by $a_0^{-1}Ce^{-cn^\delta}$, by~\eqref{eq:ad typical}. We apply~\eqref{eq:fixed} when both are admissible. On the complement, we subtract $G(\mathbf1_I)$ and use Cauchy--Schwarz with the moments of Lemma~\ref{lemma:single gap tails}. This yields
\begin{equation*}
    \left|\int(h_N-h_{N'})f_{N,N'}\right|
    \leq Cn^{C\ve-\bar s/2}{D_{\mc R}}+Ce^{-cn^\delta}D_\infty
    \,.
\end{equation*}
Combining the last two bounds proves~\eqref{eq:comp scales}. The differences of expectations bounded after~\eqref{eq:tv6} are at most $Ce^{-cn^\delta}D_\infty$, so the second assertion follows after decreasing~$c$.
\end{proof}

\begin{proof}[Proof of Theorems~\ref{theorem:convergence} and~\ref{theorem:quantitative conv}]\
\paragraph{\emph{Step 1: the limiting labelled gap law}}
We extend the finite gap sequence $(N(x_{j+1}-x_j))_{1\leq j\leq N}$ periodically to~$\mathbb Z$. Its fixed-dimensional marginals are tight by Lemma~\ref{lemma:single gap tails}. Let~$G$ be a Lipschitz function of~$k$ consecutive gaps, and fix~$\ve$ with $\kappa\ve<\bar s/2$. For $n\geq4k$, the cyclic invariance of the gap laws of~$\dGi$ and~$\dGip$ moves these gaps to the labels~$\lfloor n/2\rfloor,\ldots,\lfloor n/2\rfloor+k-1$. The second assertion of Proposition~\ref{proposition:distinct} then bounds the difference of the expectations of~$G$ under~$\dGi^\g$ and~$\dGip^\g$ by $C(\beta)n^{\kappa\ve-\bar s/2}\sum_j\|\partial_jG\|_\infty$ for all $N'\geq N\geq\max(4n,n^{2/s})$. Since~$n$ is arbitrary, these expectations form a Cauchy sequence as $N\to\infty$, which identifies all subsequential limits. Denote the resulting law on $(0,\infty)^{\mathbb Z}$ by~$\dP_\infty$, and let $(Y_j)_{j\in\mathbb Z}$ be the coordinates on~$(0,\infty)^{\mathbb Z}$. Under~$\dP_\infty$, the coordinate~$Y_j$ is the limit in law of the $j$-th gap of the extended sequence. The law~$\dP_\infty$ is invariant under index shifts, and $\dE_{\dP_\infty}Y_{j}=1$. The lower gap tail excludes zero gaps. Lemmas~\ref{theorem:almost optimal rigidity c} and~\ref{lemma:local hyper} pass to~$\dP_\infty$ and, with the Borel--Cantelli lemma, give
\begin{equation*}
    \frac1k\sum_{j=0}^{k-1}Y_j\mathrel{\xrightarrow[k\to\infty]{}}1,
    \qquad \frac1k\sum_{j=-k}^{-1}Y_j\mathrel{\xrightarrow[k\to\infty]{}}1
    \quad \dP_\infty\text{-almost surely}
    \,.
\end{equation*}
Thus $z_0=0$ and $z_{j+1}-z_j=Y_j$ for $j\in\mathbb Z$ define a locally finite simple configuration $\mc C^0$. Its law is denoted by~$\Riesz_{s,\beta}^0$.

\paragraph{\emph{Step 2: convergence in total variation of the finite gap marginals}}
The finite gap marginals have log-concave densities by Pr\'ekopa's theorem. By~\cite[Proposition~2]{cule2010theoretical}, their weak convergence implies $L^1$ convergence if the limit has a density. We therefore bound from below, uniformly in~$N$, the variances of the linear forms of a fixed block of $k$ gaps. For $a\in\dR^k$, extend $a$ to the tangent vector
\begin{equation*}
    v_i=a_i\ (i\leq k),\qquad
    v_i=-\frac{\sum_{j=1}^k a_j}{N-k}\ (i>k)
    \,.
\end{equation*}
For cyclic intervals of length $r\leq N/2$, convolution and Cauchy--Schwarz give
\begin{equation*}
    \sum_{|J|=r}\left(\sum_{j\in J}v_j\right)^2
    \leq Ck|a|^2(r+r^2/N)
    \,.
\end{equation*}
Let $D_v=v\cdot\nabla$ be the derivative in the direction~$v$. By~\eqref{eq:expHg2} and~\eqref{eq:coefficient moments} with $p=1$,
\begin{equation*}
    \dE_{\dGi^\g}D_v^2\Hc_N^\g\leq C(\beta,s,k)|a|^2
    \,.
\end{equation*}
Integration by parts, with $D_v(a\cdot x)=|a|^2$, yields the Cram\'er--Rao bound
\begin{equation*}
    \Var_{\dGi^\g}[a\cdot x]
    \geq \frac{|a|^4}{\beta\dE_{\dGi^\g}D_v^2\Hc_N^\g}
    \geq c(\beta,s,k)|a|^2
    \,.
\end{equation*}
The moments of Lemma~\ref{lemma:single gap tails} pass this lower bound to the limit. Hence the limit is a log-concave law that charges no hyperplane, and it has a density.

\paragraph{\emph{Step 3: Palm inversion}}
We define the stationary law by Palm inversion. For every bounded measurable function~$\Phi$ of locally finite configurations, we set
\begin{equation}\label{eq:palm inversion}
    \dE_{\Riesz_{s,\beta}}\Phi
    =\dE_{\dP_\infty}\int_0^{Y_0}\Phi(\mc C^0-u)\,\dd u
    \,.
\end{equation}
This is a probability law since $\dE_{\dP_\infty}Y_0=1$. Index-shift invariance and partitioning the line into consecutive gaps show that it is translation invariant, with intensity~$1$ and Palm law~$\Riesz_{s,\beta}^0$.

The periodically extended finite point process satisfies the same identity with its finite labelled gap law, by translation and cyclic invariance. Fix a bounded window $[-R,R]$. If
\begin{equation*}
    \sum_{j=-m}^{-1}Y_j>R,
    \qquad \sum_{j=1}^{m-1}Y_j>R
    \,,
\end{equation*}
then $(\mc C^0-u)|_{[-R,R]}$, for every $0\leq u\leq Y_0$, is determined by the gaps indexed from $-m$ to $m-1$. Lemmas~\ref{theorem:almost optimal rigidity c} and~\ref{lemma:local hyper}, with the moments of Lemma~\ref{lemma:single gap tails}, make the expectation of $Y_0$ times the exceptional indicator tend to zero as $m\to\infty$, uniformly in $N\geq2m$. On this finite block, Step~2 and uniform integrability of~$Y_0$ pass~\eqref{eq:palm inversion} to the limit, uniformly over $|\Phi|\leq1$. Letting $m\to\infty$ proves local total variation convergence. For $N>4R$, the periodic process and~\eqref{eq:def Qx} agree on $[-R,R]$.

Finally, we choose $n=\lfloor N^{\min(s,2)/2}/8\rfloor$ in Proposition~\ref{proposition:distinct}{, move the~$K$ requested gaps to the labels~$\lfloor n/2\rfloor,\ldots,\lfloor n/2\rfloor+K-1$ as in Step~1, apply its second assertion with radius~$K-1<\lfloor n/4\rfloor$}, and let $N'\to\infty$. The conditions $N\geq\max(4n,n^{2/s})$ hold. We take~$\ve/\kappa$ in place of~$\ve$, where~$\kappa$ is the exponent of~\eqref{eq:comp scales}. Its right-hand side then contains the factor~$n^{\ve-\bar s/2}$, which is at most ${CN^{-\frac14\min(s,2)\min(s,1)+\ve}}$. This proves Theorem~\ref{theorem:quantitative conv}.
\end{proof}

\subsection{Proof of the hyperuniformity result}
\begin{proof}[Proof of Corollary~\ref{theorem:hyper riesz}]
In the Fourier normalization $\widehat\xi(u)=\int_{\dR}\xi(x)e^{-2\pi iux}\dd x$ of~\cite[Theorems~2 and~3]{boursier2021optimal}, the limiting variance functional is
\begin{equation}\label{eq:linear statistic variance normalization}
    \frac{1}{2\beta c_s}
       \int_{\dR}(2\pi|u|)^{1-s}|\widehat\xi(u)|^2\dd u
    \,,
\end{equation}
with $c_s$ given by~\eqref{eq:riesz Fourier constant}.
For $\xi=\mathds1_{(0,1]}$, we have
\begin{equation*}
    \int_{\dR}|u|^{1-s}
       |\widehat{\mathds1_{(0,1]}}(u)|^2\dd u
       =\frac{(2\pi)^s}{2\pi\Gamma(1+s)\sin(\pi s/2)}
    \,,
\end{equation*}
so that~\eqref{eq:linear statistic variance normalization} equals~$\sigma^2$.

Under~$\dGi$, we evaluate the count~$N_{(0,R]}$ at the configuration~$\sum_i\delta_{Nx_i}$. By rotation invariance, $N_{(0,R]}$ has the law of the number of points~$Nx_i$ in~$(-R/2,R/2]$. We apply the cited theorems to $\xi=\mathds1_{(-1/4,1/4]}$, which is supported in~$(-1/2,1/2)$ as they require, at the scale~$2R/N$. By scaling, \eqref{eq:linear statistic variance normalization} equals~$2^{-s}\sigma^2$ for this~$\xi$, and we obtain
\begin{equation}\label{eq:mesoscopic count limit}
    R^{-s/2}(N_{(0,R]}-R)\mathrel{\underset{\substack{R\to\infty\\ R/N\to0}}{\Longrightarrow}}\mathcal N(0,\sigma^2),
    \qquad
    R^{-s}\Var_{\dGi}[N_{(0,R]}]\mathrel{\underset{\substack{R\to\infty\\ R/N\to0}}{\longrightarrow}}\sigma^2
    \,.
\end{equation}
To pass to the stationary limit, fix a sequence $R_k\to\infty$. Local total variation convergence and the local count moments given by Lemmas~\ref{theorem:almost optimal rigidity c} and~\ref{lemma:local hyper} allow us to choose $N_k\geq kR_k$ so that the finite and infinite normalized counts have bounded-Lipschitz distance at most $1/k$, and their variances divided by $R_k^s$ differ by at most $1/k$. Applying~\eqref{eq:mesoscopic count limit} along $(N_k,R_k)$ proves both assertions under~$\Riesz_{s,\beta}$.

For the gaps, Lemma~\ref{lemma:block variance} gives the Gaussian limit of the normalized span $K^{-s/2}(N(x_{K+1}-x_1)-K)$ under~$\dGi$, and its variance~\eqref{eq:mesoscopic gap variance}.
Finally, fix a sequence $K_k\to\infty$. Convergence of the finite gap marginals in total variation and the moments of Lemma~\ref{lemma:single gap tails} allow us to choose $N_k\geq kK_k$ so that the finite and Palm normalized $K_k$-gap spans have bounded-Lipschitz distance at most $1/k$, and their variances differ by at most~$1/k$. The mesoscopic gap limit therefore gives the Gaussian limit for $K^{-s/2}(z_K-K)$, and~\eqref{eq:mesoscopic gap variance} gives the Palm variance assertion.
\end{proof}

\subsection{Proof of the repulsion estimate}
\begin{proof}[Proof of Proposition~\ref{proposition:replusion}]
Lemma~\ref{lemma:single gap tails} bounds $\dGi(N(x_2-x_1)\leq u)$ by $Ce^{-cu^{-s}}$, uniformly in~$N$. The law of this gap converges in total variation to the law of~$z_1-z_0$ under~$\Riesz_{s,\beta}^0$, which proves the bound for $i=0$. The invariance of the Palm law under shifts of the labels gives it for every~$i$.
\end{proof}

\appendix

\section{Well-posedness results}\label{section:existence c}
We prove Propositions~\ref{proposition: existence c} and~\ref{proposition: HS gaps c} by following~\cite[Appendix~A]{boursier2021optimal}. This argument uses condition~\eqref{eq:collision criterion} at collisions, and we first verify this condition for the measures of Sections~\ref{section:decay short} to~\ref{section:change measure}.

\begin{lemma}[{Collision condition}]\label{lemma:collision}
Let $s\in(0,1)\cup(1,+\infty)$, $\beta>0$, $\ve>0$ and $\iota\in\{0,1\}$, let $N\geq2$ be an integer, let $n\leq N/2$ or $n=N$, and let $I=\{1,\ldots,n\}$. Let~$\nu$ be one of the following probability measures, with density proportional to~$e^{-H^\g}$ on a domain~$\Omega$.
\begin{enumerate}
    \item The gap law~$\dGi^\g$ of the Riesz gas, or its tilted law~$\dGiQ^\g$ of~\eqref{eq:def nu hypersing}, on $\Omega=\mc M_N$.
    \item If $n\leq N/2$, the marginal $\pi\#\dGiQ^\g$ of this tilted law, on $\Omega=\pi(\mc M_N)$.
    \item If $0<L<N$, the conditional law of the marginal of the second item given $x_1+\ldots+x_n=L$, or an interpolating measure~$\nu(t)$ of~\eqref{eq:def mun(t)} with common gap sum~$L$ and $0\leq t\leq1$, on $\Omega=\mc M_n(L)$.
\end{enumerate}
Then, for every $a>0$, there exists $C(a,N,\nu)>0$ such that, with the gradient and the Laplacian of the intrinsic metric of~$\Omega$,
\begin{equation}\label{eq:gap collision criterion}
    a|\nabla H^\g|^2-\Delta H^\g\geq\operatorname{dist}(\,\cdot\,,\partial\Omega)^{-2}-C(a,N,\nu)
    \,.
\end{equation}
In the first case, $H=H^\g\circ\Gap_N$ satisfies~\eqref{eq:collision criterion} on~$D_N$. Moreover, \eqref{eq:gap collision criterion} holds on $\Omega=\mc M_N$ for $\nu=\Gap_N\#\mu$ whenever~$\mu$ satisfies Assumptions~\ref{assumptions:gibbs measure c}.
\end{lemma}

\begin{proof}\
\paragraph{\emph{Step 1: change of metric}}
We show that, if $H=H^\g\circ\Gap_N$ is convex on~$D_N$, then~$H$ satisfies~\eqref{eq:collision criterion} for every $a>0$ if and only if~$H^\g$ satisfies~\eqref{eq:gap collision criterion} on~$\mc M_N$ for every $a>0$. The energy of a measure satisfying Assumptions~\ref{assumptions:gibbs measure c} is convex, so this equivalence proves the last assertion of the lemma.

Since~$H$ is invariant under common rotations, its gradient is orthogonal to~$\mathbf1$ and its Hessian vanishes in the direction~$\mathbf1$. The linear part of~$\Gap_N$ maps~$\mathbf1^\perp$ invertibly onto the tangent space of~$\mc M_N$, and the norms of this map and of its inverse depend only on~$N$. Hence the gradients of~$H$ and~$H^\g$ are comparable, and so are the distances to the boundaries of~$D_N$ and~$\mc M_N$. The Laplacians of~$H$ and~$H^\g$ are the traces of their Hessians, which are nonnegative, so these Laplacians are comparable too. There exists $C(N)\geq1$ such that
\begin{equation*}
    {|\nabla H|\leq C(N)|\nabla H^\g|,
    \qquad
    \Delta H^\g\leq C(N)\Delta H,
    \qquad
    d_\partial\leq C(N)\operatorname{dist}(\Gap_N(X),\partial\mc M_N)}
    \,,
\end{equation*}
and these three bounds also hold with the particle and gap coordinates exchanged.

Assume that~$H$ satisfies~\eqref{eq:collision criterion}. Since $\Delta H\geq0$ and $C(N)\geq1$, the first two bounds and~\eqref{eq:collision criterion} with~$C(N)^{-4}a$ in place of~$a$ give
\begin{equation*}
    {a|\nabla H^\g|^2-\Delta H^\g\geq C(N)^{-2}a|\nabla H|^2-C(N)^2\Delta H\geq C(N)^2d_\partial^{-2}-C(N)^2C(C(N)^{-4}a,N)}
    \,.
\end{equation*}
By the third bound, $C(N)^2d_\partial^{-2}\geq\operatorname{dist}(\Gap_N(X),\partial\mc M_N)^{-2}$, which proves~\eqref{eq:gap collision criterion}. The converse implication follows in the same way from the exchanged bounds, since $\Delta H^\g\geq0$.

\paragraph{\emph{Step 2: Riesz energy and penalties}}
We prove~\eqref{eq:gap collision criterion} in the first case, where $H^\g=\beta(\Hc_N^\g+\FF^\g)$, with $\FF^\g=0$ and $\iota=0$ for the gap law. Let $y\in\mc M_N$ and $d=\min_iy_i$, so that $\operatorname{dist}(y,\partial\mc M_N)\geq d$. We compare~$y$ with the equally spaced configuration~$\mathbf1$, and show that there exist $c,C>0$, depending on~$\nu$, such that
\begin{equation}\label{eq:radial bounds}
    \nabla H^\g(y)\cdot(y-\mathbf1)\geq c\Bigl(d^{-s-1}+\iota\sum_{i\in I}y_i^{-3}\Bigr)-C,
    \qquad
    \Delta H^\g(y)\leq C\Bigl(d^{-s-2}+\iota\sum_{i\in I}y_i^{-4}\Bigr)
    \,.
\end{equation}

We first deduce~\eqref{eq:gap collision criterion} from~\eqref{eq:radial bounds}. Since~$\mc M_N$ is bounded, we have $|y-\mathbf1|\leq C$, so the first bound gives a lower bound on~$|\nabla H^\g(y)|$. Squaring it, we obtain, for every $a>0$,
\begin{equation*}
    {a|\nabla H^\g|^2-\Delta H^\g\geq ca\Bigl(d^{-2s-2}+\iota\sum_{i\in I}y_i^{-6}\Bigr)-C\Bigl(d^{-s-2}+\iota\sum_{i\in I}y_i^{-4}\Bigr)-Ca}
    \,.
\end{equation*}
Since $s>0$, we have $2s+2>\max(s+2,2)$ and $6>4$. The right-hand side is therefore at least $Nd^{-2}-C$, where~$C$ depends on~$a$, and this proves~\eqref{eq:gap collision criterion}. Since $H^\g\circ\Gap_N$ is convex, Step~1 then gives~\eqref{eq:collision criterion} on~$D_N$.

We now prove the first bound of~\eqref{eq:radial bounds}. As in the proof of Lemma~\ref{lemma:single gap tails}, the function obtained from~$H^\g$ by subtracting the nearest-pair terms $2\beta y_i^{-s}$ and the inverse-gap penalties $\iota\beta\theta(n^{-\ve}y_i^{-1})$ is convex. Its gradient is monotone. Its derivative at~$y$ in the direction $y-\mathbf1$ is therefore at least its derivative at~$\mathbf1$ in this direction, which is at least~$-C$. In the same direction, the nearest-pair term and the inverse-gap penalty of~$y_i$ have the derivatives
\begin{equation*}
    {2\beta sy_i^{-s-1}(1-y_i)\geq cy_i^{-s-1}-C,
    \qquad
    \beta n^{-\ve}\theta'(n^{-\ve}y_i^{-1})\,y_i^{-2}(1-y_i)\geq cy_i^{-3}-C}
    \,.
\end{equation*}
The first inequality holds since $y_i^{-s-1}(1-y_i)\geq-2^s$ for $y_i\geq1/2$. The second holds since $\theta'\geq0$ and $\theta''\geq1$ on $[1,+\infty)$, so that $\theta'(t)\geq t-1$ for $t\geq1$. Summing over~$i$ gives the first bound of~\eqref{eq:radial bounds}.

For the second bound, since the ambient Hessian of~$H^\g$ is nonnegative, the intrinsic Laplacian of~$H^\g$ is at most the trace of this Hessian. In~\eqref{eq:expHg2}, every gap is at least~$d$, and $g_s''(t)={s(s+1)}\sum_{k\in\mathbb Z}|t+k|^{-s-2}$. The coefficient of each interval of gaps is therefore at most~$Cd^{-s-2}$. Since $0\leq\theta''\leq C$ and $\theta'(t)\leq Ct$, the penalties of~$y_i$ have second derivatives at most $C(1+\iota y_i^{-4})$. This proves the second bound of~\eqref{eq:radial bounds}.

\paragraph{\emph{Step 3: marginals}}
Let $n\leq N/2$ and $x\in\pi(\mc M_N)$. As in Subsection~\ref{sub:marginal}, we write the full gap configuration as~$(x,y)$, and we use the conditional law~$\mu_x$ and the notation~$q$, $\Xi_i$, $P_I$ and~$Q$ of that subsection. Let~$d$ be the smallest of $x_1,\ldots,x_n$ and $N-\sum_{i\in I}x_i$, so that $\operatorname{dist}(x,\partial\pi(\mc M_N))\geq d/\sqrt N$. Since the deduction of Step~2 gives the lower bound $Nd^{-2}-C$, it suffices to prove~\eqref{eq:radial bounds} for $H^\g=\beta\widetilde\Hc_n^\g$, with~$x$ and~$\mathbf1_I$ in place of~$y$ and~$\mathbf1$.

We first split~$\beta\widetilde\Hc_n^\g$ as in Step~2. On the fiber of~$\mu_x$, the exterior gaps~$y_a$ have sum $N-\sum_{i\in I}x_i$, so that
\begin{equation*}
    {\sum_{a\notin I}y_a^{-s}=\Bigl(N-\sum_{i\in I}x_i\Bigr)^{-s}+\Bigl(\sum_{a\notin I}y_a^{-s}-\Bigl(\sum_{a\notin I}y_a\Bigr)^{-s}\Bigr)}
    \,.
\end{equation*}
The last function is convex in~$y$. Indeed, for every vector $(v_a)_{a\notin I}$, the Cauchy--Schwarz inequality gives
\begin{equation*}
    {\Bigl(\sum_{a\notin I}v_a\Bigr)^2\leq\sum_{a\notin I}y_a^{-s-2}v_a^2\sum_{a\notin I}y_a^{s+2}\leq\sum_{a\notin I}y_a^{-s-2}v_a^2\Bigl(\sum_{a\notin I}y_a\Bigr)^{s+2}}
    \,.
\end{equation*}
By Step~2, $\Hc_N^\g(x,y)$ is therefore the sum of $2\sum_{i\in I}x_i^{-s}$, of $2(N-\sum_{i\in I}x_i)^{-s}$ and of a convex function of~$(x,y)$. By Pr\'ekopa's theorem, the marginal of a log-concave density is log-concave. Hence $\beta\widetilde\Hc_n^\g$ is the sum of $2\beta\sum_{i\in I}x_i^{-s}$, of $2\beta(N-\sum_{i\in I}x_i)^{-s}$, of the inverse-gap penalties and of a convex function.

The second term did not appear in Step~2. In the direction $x-\mathbf1_I$, its derivative is
\begin{equation*}
    {2\beta s\Bigl(N-\sum_{i\in I}x_i\Bigr)^{-s-1}\Bigl(\sum_{i\in I}x_i-n\Bigr)}
    \,.
\end{equation*}
If $N-\sum_{i\in I}x_i\leq N/4$, then $\sum_{i\in I}x_i-n\geq N/4$, since $n\leq N/2$. Otherwise, this derivative is bounded. In both cases, it is at least $c(N-\sum_{i\in I}x_i)^{-s-1}-C$. The other terms are treated as in Step~2, which gives the first bound of~\eqref{eq:radial bounds}.

For the second bound, we take the trace of~\eqref{eq:short marginal Hessian}. On the diagonal, the covariance term is $-\beta\Var_{\mu_x}[\Xi_i]$, which is nonpositive, and the penalty term is bounded as in Step~2. The entry $(P_I)_{ii}+Q_{ii}$ is the expectation under~$\mu_x$ of the second derivative of~$\Hc_N^\g$ in the direction $e_i-q$. By~\eqref{eq:expHg2}, this second derivative is at most $C(d^{-s-2}+\sum_{a\notin I}y_a^{-s-2})$. It remains to show that, for every $a\notin I$,
\begin{equation*}
    \dE_{\mu_x}[y_a^{-s-2}]\leq Cd^{-s-2}
    \,.
\end{equation*}

Let~$y^*$ minimize~$\Hc_N^\g(x,\cdot)$ on the fiber. Then $\Hc_N^\g(x,y^*)$ is at most the energy with equal exterior gaps. In that configuration, every gap is at least~$d/N$, so that $\beta\Hc_N^\g(x,y^*)\leq Cd^{-s}$. Since $\Hc_N^\g(x,\cdot)$ is convex, its value at $(y+y^*)/2$ is at most the average of its values at~$y$ and~$y^*$. The change of variables $y\mapsto(y+y^*)/2$, with Jacobian $2^{n+1-N}$ on the fiber, therefore gives
\begin{equation*}
    {\dE_{\mu_x}\Bigl[e^{\beta(\Hc_N^\g(x,y)-\Hc_N^\g(x,y^*))/2}\Bigr]\leq2^{N-n-1}}
    \,.
\end{equation*}
The nearest-pair term of~$y_a$ gives $y_a^{-s}\leq C(\beta\Hc_N^\g(x,y)+C)$. We raise this bound to the power $(s+2)/s$, and split $\Hc_N^\g(x,y)$ into $\Hc_N^\g(x,y)-\Hc_N^\g(x,y^*)$ and $\Hc_N^\g(x,y^*)$. The exponential moment above bounds the expectation of the first part to the power $(s+2)/s$, and the second part is at most $Cd^{-s}$. This proves the moment bound, and hence the second bound of~\eqref{eq:radial bounds}.

\paragraph{\emph{Step 4: measures with a fixed gap sum}}
We may assume $n\geq2$, since~$\mc M_1(L)$ is a point. Let $x\in\mc M_n(L)$ and $d=\min_{i\in I}x_i$, so that $\operatorname{dist}(x,\partial\mc M_n(L))\geq d$. As in Step~3, it suffices to prove~\eqref{eq:radial bounds} with~$x$ and~$(L/n)\mathbf1_I$ in place of~$y$ and~$\mathbf1$.

The energy of the conditional law of the marginal is the restriction of~$\beta\widetilde\Hc_n^\g$ to~$\mc M_n(L)$, on which $2\beta(N-L)^{-s}$ is constant. The energy of~$\nu(t)$ is $\beta(\FF^\g+(1-t)\mathcal K_y+t\mathcal K_z)$, where the exterior gaps are fixed and~$\FF^\g$ is the penalty~\eqref{eq:F forcing}. By Steps~2 and~3, each of these energies is the sum of $2\beta\sum_{i\in I}x_i^{-s}$, of the inverse-gap penalties and of a convex function. The derivatives computed in Step~2 therefore give the first bound, with constants depending on~$L$, $y$ and~$z$.

Both energies are convex. Their intrinsic Laplacians on~$\mc M_n(L)$ are therefore at most their Laplacians on~$\dR^n$, which Step~3 and~\eqref{eq:expHg2} bound. This gives the second bound.
\end{proof}

\begin{proof}[Proof of Proposition~\ref{proposition: existence c}]
Since $\Delta H\geq0$, condition~\eqref{eq:collision criterion} with $a=1/(4N)$ gives
\begin{equation*}
    {\frac{|\nabla H|^2}{4N}-\frac{\Delta H}{2}\geq d_\partial^{-2}-C(1/(4N),N)}
    \,,
\end{equation*}
which is the collision condition of~\cite[Assumption~3.1]{boursier2021optimal}. Hence~\cite[Proposition~3.1]{boursier2021optimal} applies to~$\mu$. It gives a solution $\nabla\phi\in\mathcal V_\mu$ of~\eqref{eq:the HS c} with $\dE_\mu\sum_i\partial_i\phi=0$, its variational characterization, and the variance and covariance identities. Let $u\in\mathcal V_\mu$ be a solution of~\eqref{eq:the HS c} with $\dE_\mu\sum_iu_i=0$. For every $h\in\mathcal V_\mu$ with $\dE_\mu\sum_ih_i=0$, the functional at~$h$ exceeds its value at~$u$ by $\mathfrak q_\mu(h-u,h-u)$. Hence~$u$ is the minimizer, and $u=\nabla\phi$.
\end{proof}

\begin{proof}[Proof of Proposition~\ref{proposition: HS gaps c}]
Since $g_{ij}''\geq c$, the tangent Hessian of~$H^\g$ is at least a constant $c(N)>0$. Lemma~\ref{lemma:bakry} then gives a Poincar\'e inequality for~$\nu$, hence a unique mean-zero solution $\psi\in H^1(\nu)$ of $\mc L^\nu\psi=\chi-\dE_\nu\chi$. Since $\Delta H^\g\geq0$, Lemma~\ref{lemma:collision} with $a=1/4$ gives $|\nabla H^\g|^2/4-\Delta H^\g/2\geq\operatorname{dist}(\,\cdot\,,\partial\mc M_N)^{-2}-C$. The argument of~\cite[Subsection~A.2]{boursier2021optimal} therefore gives smooth functions~$\psi_k$, compactly supported in~$\mc M_N$, such that
\begin{equation*}
    \psi_k\mathrel{\xrightarrow[k\to\infty]{}}\psi,\qquad
    \mc L^\nu\psi_k\mathrel{\xrightarrow[k\to\infty]{}}\chi-\dE_\nu\chi
    \quad\text{in }L^2(\nu)
    \,.
\end{equation*}
For smooth~$f$ and tangent fields~$h$ compactly supported in~$\mc M_N$, integration by parts and the identity $\nabla\mc L^\nu f=\mc L^\nu\nabla f+\nabla^2H^\g\nabla f$ give
\begin{equation*}
    {\|\nabla f\|_2^2=\dE_\nu[f\mc L^\nu f],
    \qquad
    \mathfrak q_\nu(\nabla f,h)=\dE_\nu[\mc L^\nu f\,(\nabla H^\g\cdot h-\operatorname{div}h)]}
    \,.
\end{equation*}
Applied to $f=\psi_k-\psi_\ell$ and $h=\nabla f$, these identities show that~$\nabla\psi_k$ converges to~$\nabla\psi$ in~$\mathcal V_\nu$. For $f=\psi_k$, we let $k\to\infty$ in the second identity. Its right-hand side tends to $\dE_\nu[(\chi-\dE_\nu\chi)(\nabla H^\g\cdot h-\operatorname{div}h)]$, which equals $\dE_\nu[\nabla\chi\cdot h]$ by integration by parts. By density in~$\mathcal V_\nu$, we obtain, for every $h\in\mathcal V_\nu$,
\begin{equation*}
    \mathfrak q_\nu(\nabla\psi,h)=\dE_\nu[\nabla\chi\cdot h]
    \,.
\end{equation*}
This is~\eqref{eq:HS gaps form c} in tangent directions. With~$\nabla\chi$ and~$\nabla^2H^\g$ taken from ambient extensions, the component of~\eqref{eq:HS gaps form c} along~$\mathbf1$ defines the distribution~$\lambda$. Since~$\nabla\psi$ is tangent, $\mc L^\nu\nabla\psi$ has no such component, which gives~\eqref{eq:exp la}. Coercivity proves uniqueness and the variational assertion. Testing the scalar equation with $\chi-\dE_\nu\chi$ proves the variance identity.
\end{proof}

The same proof applies to each measure~$\nu$ of Lemma~\ref{lemma:collision}. By Steps~2 to~4 of the proof of that lemma, the energy of~$\nu$ is the sum of nearest-pair terms and of a convex function, so its tangent Hessian is at least a positive constant on~$\Omega$. We obtain the conclusions of Proposition~\ref{proposition: HS gaps c} for~$\nu$, with~$n$ in place of~$N$ in~\eqref{eq:exp la}, $\mathbf1_I$ in place of~$\mathbf1$ on~$\mc M_n(L)$, and $\lambda=0$ on~$\pi(\mc M_N)$. In particular, the Dirichlet form $\dE_\nu[|\nabla f|^2]$ is closed and conservative on~$H^1(\nu)$, which is the second item of Assumptions~\ref{assumptions:mu2}. The argument of~\cite[Subsection~A.2]{boursier2021optimal} also shows that smooth functions compactly supported in~$\Omega$ are dense in~$H^1(\nu)$. Hence~$H^1(\nu)$ satisfies the chain rule for Lipschitz radial truncations of vector fields, and~$\nu$ satisfies Assumptions~\ref{assumptions:mu}.

\section{Rigidity for the short-range gas}\label{section:short local laws}
For $s>1$, this appendix proves the rigidity estimate for blocks of gaps that replaces Lemma~\ref{theorem:almost optimal rigidity c}, with the one-gap tails of Lemma~\ref{lemma:single gap tails}. Section~\ref{section:decay short} and Appendix~\ref{section:loc laws} use it.

\begin{lemma}\label{lemma:local hyper}
Let $s>1$, $\beta>0$ and $\eta>0$. Then there exist $C,c>0$, locally uniform in~$\beta$, such that, for every $N\geq2$ and every cyclic block~$J$ of $k\leq N/2$ gaps,
\begin{equation}\label{eq:gaphyper}
    \dGi^\g\left(\left|\sum_{i\in J}x_i-k\right|
        \geq k^{1/2+\eta}\right)
    \leq Ce^{-ck^{\eta/[4(s+2)]}}
    \,.
\end{equation}
\end{lemma}

\begin{proof}
We may assume that~$k$ is large. Let~$\theta$ be the convex function of~\eqref{eq:def nu hypersing}, and let $R=k^{\eta/[2(s+2)]}$. Let~$\mu_J$ be the full gap law tilted by the penalty~$\sum_{i\in J}\theta(x_i/R)$. Lemma~\ref{lemma:single gap tails} implies
\begin{equation*}
    \dGi^\g\Bigl(\sum_{i\in J}\theta(x_i/R)\neq0\Bigr)\leq Ck e^{-cR},\qquad
    \mathbb E_{\dGi^\g}[e^{-\beta\sum_{i\in J}\theta(x_i/R)}]\geq\frac12,
    \qquad
    \operatorname{TV}(\mu_J,\dGi^\g)\leq Ck e^{-cR}
    \,.
\end{equation*}

The nearest-pair terms and the penalty satisfy
\begin{equation*}
    \nabla^2\Bigl(\Hc_N^\g+\sum_{i\in J}\theta(x_i/R)\Bigr)
    \geq cR^{-s-2}\operatorname{diag}(\mathds{1}_J)
\end{equation*}
on the full gap simplex. Indeed, the nearest-pair curvature supplies this bound when $x_i\leq R$, and the penalty supplies it when $x_i\geq R$. After subtracting the corresponding quadratic, Pr\'ekopa's theorem shows that the marginal on~$J$ is uniformly log-concave with constant $c\beta R^{-s-2}$. The logarithmic Sobolev inequality of Lemma~\ref{lemma:bakry} and Herbst's argument give, for every~$u>0$,
\begin{equation*}
    \mu_J(|\mathbf1_J\cdot x-\mathbb E_{\mu_J}[\mathbf1_J\cdot x]|\geq u)
    \leq2\exp(-cR^{-s-2}u^2/k)
    \,.
\end{equation*}
The gap law is invariant under cyclic shifts, so $\mathbb E_{\dGi^\g}[\mathbf1_J\cdot x]=k$. Since the gaps have uniformly bounded second moments, the Cauchy--Schwarz inequality and the bound on~$\operatorname{TV}(\mu_J,\dGi^\g)$ give
\begin{equation*}
    |\mathbb E_{\mu_J}[\mathbf1_J\cdot x]-k|
    \leq Ck^{3/2}e^{-cR/2}
    \,.
\end{equation*}
We take $u=k^{1/2+\eta}/2$, so that the exponent $cR^{-s-2}u^2/k$ is of order~$k^{3\eta/2}$. For large~$k$, the difference of the means is at most~$u$, so the event in~\eqref{eq:gaphyper} has $\mu_J$-probability at most $2e^{-ck^{3\eta/2}}$. Its $\dGi^\g$-probability is larger by at most $\operatorname{TV}(\mu_J,\dGi^\g)\leq Cke^{-cR}$. Decreasing the exponent to $\eta/[4(s+2)]$ proves~\eqref{eq:gaphyper}.
\end{proof}

\section{Local laws for the interpolating measure}\label{section:loc laws}
We prove Lemma~\ref{lemma:local laws mu(t)}. The full-system local laws show that typical exteriors are admissible. A bound on the oscillation of~$\mathcal K_y-\mathcal K_z$ then transfers the tails of~$\nu(0)$ and~$\nu(1)$ to every~$\nu(t)$.

\begin{proof}[Proof of Lemma~\ref{lemma:local laws mu(t)}]
Let~$\ve_0(s)$ be the exponent of Lemma~\ref{theorem:almost optimal rigidity c} if $0<s<1$, and let $\ve_0=1$ if $s>1$. For $0<b<\ve_0$, this lemma and Lemma~\ref{lemma:local hyper} give, for every block~$J$ of $k\leq N/2$ gaps and with $\bar s=\min(s,1)$ as before~\eqref{eq:interior rigidity event},
\begin{equation}\label{eq:rigidity both ranges}
    \dGi^\g\left(\left|\sum_{i\in J}x_i-k\right|\geq k^{\bar s/2+b}\right)
    \leq C(b)e^{-c(b)k^{b/[4(s+2)]}}
    \,.
\end{equation}

\paragraph{\emph{Step 1: local laws for the full system}}
Let~$\Upsilon$ be the function defined before~\eqref{eq:interior rigidity event}, and let~$\mc R_\eta(L)$ be given by~\eqref{eq:interior rigidity event}. We show that there exists $a_s>0$ such that, for every $0<\eta<\ve_0/4$ and uniformly in $N\geq\max(4n,n^{2/s})$,
\begin{equation}\label{eq:full weighted rigidity}
    \dE_{\dGiQ^\g}\left[\Upsilon\mathds{1}_{\{(x_1,\ldots,x_n)\notin\mc R_\eta\}}\right]
    \leq C(\beta,s,\eta)e^{-c(\beta,s,\eta)n^{a_s\eta}}
    \,.
\end{equation}
We start with~$\dGi^\g$ and a block $J\subset I$ of~$k$ gaps. If $k\geq n^{\eta/2}$, we apply~\eqref{eq:rigidity both ranges} with the largest power of two~$b$ such that $2b\leq\ve_0$ and $k^{2b}\leq n^\eta$. Since $k\leq n$, we have $b\geq\min(\ve_0,\eta)/4$, so~$b$ takes finitely many values, depending only on~$s$ and~$\eta$. By the choice of~$b$, we have $k^b\leq n^{\eta/2}$ and $k^b\geq n^{c\eta}$, with~$c$ depending only on~$s$. Hence $|\sum_{i\in J}x_i-k|\leq n^\eta k^{\bar s/2}$, except with probability at most $Ce^{-cn^{c\eta}}$.

If $k<n^{\eta/2}$, then $k<n^\eta k^{\bar s/2}$, so a negative deviation of size~$n^\eta k^{\bar s/2}$ is impossible. A positive one forces the block of~$\lceil n^{\eta/2}\rceil$ gaps that contains~$J$ to have sum at least~$n^\eta$. We apply~\eqref{eq:rigidity both ranges} to this block with $b=\min(\ve_0,1-\bar s/2)/2$. Since $\bar s/2+b<1$, its sum is at most twice its length, hence smaller than~$n^\eta$, except with probability at most $Ce^{-cn^{c\eta}}$.

A union bound over the at most~$n^2$ blocks, and Lemma~\ref{lemma:single gap tails} for the bounds $n^{-\eta}\leq x_i\leq n^\eta$, prove~\eqref{eq:full weighted rigidity} under~$\dGi^\g$ without the factor~$\Upsilon$, for a sufficiently small $a_s>0$. By Lemma~\ref{lemma:single gap tails}, $\dE_{\dGi^\g}[\Upsilon^2]\leq Cn^2$, and the Cauchy--Schwarz inequality inserts the factor~$\Upsilon$. The penalty~\eqref{eq:F forcing} vanishes on~$\mc R_{\ve/2}$ for large~$n$, so $\dE_{\dGi^\g}[e^{-\beta\FF^\g}]\geq1/2$. This transfers the bound to~$\dGiQ^\g$.

\paragraph{\emph{Step 2: oscillation of the energy difference}}
Let~$y$ and~$z$ be compatible exteriors satisfying~\eqref{eq:def ad}, and let~$\mathcal K_y$ and~$\mathcal K_z$ be as in~\eqref{eq:def cond m}. We place the particle preceding the gap~$1$ at~$0$ and the particle following the gap~$n$ at~$L$. For $1\leq i<n$, the particle between the gaps~$i$ and~$i+1$ is then at distances $x_1+\cdots+x_i$ and $x_{i+1}+\cdots+x_n$ from them, and the derivative $\partial_i-\partial_{i+1}$ moves only this particle. Hence $(\partial_i-\partial_{i+1})(\mathcal K_y-\mathcal K_z)$ is, up to sign, the difference of the forces that the exteriors~$y$ and~$z$, with their periodic images, exert on this particle. For $1\leq r\leq\lfloor N/4\rfloor$ and on each side of~$[0,L]$, we pair the $r$-th exterior particle of~$y$ with that of~$z$. Since $\bar s/2+\ve\leq3/4$, the sums of~$r$ exterior gaps in~\eqref{eq:def ad} are at least~$r/2$ for $r\geq Cn^{4\ve}$. At distance~$a$ from a boundary particle, the pairs with $r<Cn^{4\ve}$ contribute at most $Cn^{4\ve}a^{-s-1}$ to this difference. For the other pairs, the mean value theorem and~\eqref{eq:def ad} bound the contribution by
\begin{equation*}
    {Cn^\ve\sum_{r\geq Cn^{4\ve}}\frac{r^{\bar s/2+\ve}}{(a+r)^{s+2}}
    \leq Cn^\ve a^{\bar s/2+\ve-s-1}}
    \,.
\end{equation*}
The unpaired particles and the periodic images are at distance at least~$cN$, and exert a total force at most~$CN^{-s}$ on each interior particle. The interior pair interactions differ only through the periods~$N$ and~$N'$, which changes the force by at most $CnN^{-s-1}\leq CN^{-s}$. Thus, for $1\leq i<n$,
\begin{equation}\label{eq:exterior force oscillation}
    |(\partial_i-\partial_{i+1})(\mathcal K_y-\mathcal K_z)|
    \leq {Cn^{C\ve}}\sum_{a\in\{x_1+\cdots+x_i,x_{i+1}+\cdots+x_n\}}
        {\left(a^{-s-1}+a^{\bar s/2+\ve-s-1}\right)}+CN^{-s}
    \,.
\end{equation}
We join two points of~$\mc R_\ve(L)$ by a segment, which stays in this convex set. Along it, $x_1+\cdots+x_i\geq i/n^\ve$ and $x_{i+1}+\cdots+x_n\geq(n-i)/n^\ve$, and the values of $x_1+\cdots+x_i$ at the two endpoints differ by at most $2n^\ve\min(i,n-i)^{\bar s/2}$. Since both endpoints have gap sum~$L$, summation by parts writes the derivative of~$\mathcal K_y-\mathcal K_z$ along the segment as the sum over $1\leq i<n$ of $(\partial_i-\partial_{i+1})(\mathcal K_y-\mathcal K_z)$ times the difference of the values of $x_1+\cdots+x_i$ at the endpoints. We integrate~\eqref{eq:exterior force oscillation} along the segment with this identity. Since $\bar s\leq s$ and $N\geq n^{2/s}$, we obtain
\begin{equation}\label{eq:coarse oscillation}
    \operatorname{osc}_{\mc R_\ve(L)}{(\mathcal K_y-\mathcal K_z)}
    \leq {Cn^{C\ve}\Bigl(\sum_{i=1}^ni^{\bar s/2-s-1}+\sum_{i=1}^ni^{\ve-1}+n^{1+\bar s/2}N^{-s}\Bigr)}
    \leq n^{A_s\ve}
    \,,
\end{equation}
for a constant~$A_s>1$ depending only on~$s$ and for $n$ sufficiently large depending on~$s$ and~$\ve$. We choose~$K_s$ with $a_sK_s>3A_s$, and then~$\ve$ with $K_s\ve<\min(1/4,\ve_0/4)$. When $0<s<1$, we also require $4(\mathfrak m+1)K_s\ve\leq1$, with~$\mathfrak m$ as in Subsection~\ref{sub:re}, so that the first item of Lemma~\ref{lemma:ident} applies with~$K_s\ve$ in place of~$\ve$.

\paragraph{\emph{Step 3: typical exteriors are admissible}}
The rigidity bound~\eqref{eq:rigidity both ranges} at scale~$n$ gives the first condition in~\eqref{eq:def ad}. For $r<n$, the block argument of Step~1 controls the sums of~$r$ exterior gaps. For $n\leq r\leq N/2$, we apply~\eqref{eq:rigidity both ranges} with exponent~$\ve/2$ and sum over~$r$. For $r>N/2$, we apply it to the complementary block of length $N-r\geq n$. The factor~$r^\ve$ in~\eqref{eq:def ad} makes these sums uniform in~$N$.

For~\eqref{eq:ad endpoints}, we disintegrate~\eqref{eq:full weighted rigidity} along the exterior, into the conditional measures~$\nu_n^y$ of~\eqref{eq:def cond m}. With $\eta=\ve$, Markov's inequality gives
\begin{equation*}
    \dGiQ^\g\left(\nu_n^y(\mc R_\ve(L_N(y)))<\frac12\right)
    \leq Ce^{-cn^{a_s\ve}}
    \,.
\end{equation*}
With $\eta=K_s\ve$, it gives
\begin{equation*}
    \dGiQ^\g\left(
        \dE_{\nu_n^y}[\Upsilon\mathds{1}_{\mc R_{K_s\ve}(L_N(y))^c}]
        >e^{-n^{2A_s\ve}}\right)
    \leq C\exp\left(n^{2A_s\ve}-cn^{a_sK_s\ve}\right)
    \,.
\end{equation*}
Since $a_sK_s>3A_s$, this proves~\eqref{eq:ad typical}.

\paragraph{\emph{Step 4: tails along the interpolation}}
Let $y$ and~$z$ be admissible and compatible. Let~$p_0$ and~$p_1$ be the densities of~$\nu_n^y$ and~$\nu_n^z$ on~$\mc M_n(L)$. By~\eqref{eq:def mun(t)}, the measure~$\nu(t)$ has density $Z_t^{-1}p_0^{1-t}p_1^t$, where $Z_t=\int p_0^{1-t}p_1^t$. By~\eqref{eq:ad endpoints}, both endpoint laws give~$\mc R_\ve(L)$ mass at least~$1/2$. On this set, $\log(p_1/p_0)$ equals~$\beta(\mathcal K_y-\mathcal K_z)$ up to a constant. By~\eqref{eq:coarse oscillation}, the normalized restrictions of~$p_0$ and~$p_1$ to this set have a density ratio between $e^{-\beta n^{A_s\ve}}$ and $e^{\beta n^{A_s\ve}}$. Hence
\begin{equation*}
    Z_t\geq\frac12e^{-\beta n^{A_s\ve}}
    \,.
\end{equation*}
H\"older's inequality and~\eqref{eq:ad endpoints} give
\begin{equation*}
    {\dE_{\nu(t)}[\Upsilon\mathds{1}_{\mc R_{K_s\ve}(L)^c}]
    \leq Z_t^{-1}\dE_{\nu_n^y}[\Upsilon\mathds{1}_{\mc R_{K_s\ve}(L)^c}]^{1-t}\dE_{\nu_n^z}[\Upsilon\mathds{1}_{\mc R_{K_s\ve}(L)^c}]^t
    \leq2\exp\left(\beta n^{A_s\ve}-n^{2A_s\ve}\right)
    \leq Ce^{-cn^\delta}}
    \,.
\end{equation*}
This proves~\eqref{eq:interpolation weighted tail}, hence~\eqref{eq:est1}--\eqref{eq:est2}.
\end{proof}

\bibliographystyle{alpha}
\bibliography{main.bib}

\end{document}